\documentclass[12pt, a4paper]{amsart}
\usepackage{amscd, amssymb, pb-diagram}
\usepackage{mathtools}
\usepackage{enumerate}
\usepackage[shortlabels]{enumitem}
\usepackage{amsmath}
\usepackage[margin=2.9cm]{geometry}
\usepackage{amscd, amssymb, pb-diagram}
\usepackage{tikz-cd}
\usepackage{mathtools}
\usepackage{amsmath}
\usepackage{amsthm} 
\usepackage[nospace,noadjust]{cite}
\usepackage{lmodern}
\usepackage{graphicx}
\graphicspath{{images/}{../images/}}
\usepackage[toc,page]{appendix}
\usepackage[normalem]{ulem}
\usepackage{emptypage}
\usepackage{hyperref}
\usepackage{subfiles}
\usepackage{datetime}

\usepackage{tensor}
\usepackage{wrapfig}
\usepackage{tikz}
\DeclarePairedDelimiter\ceil{\lceil}{\rceil}
\DeclarePairedDelimiter\floor{\lfloor}{\rfloor}

\allowdisplaybreaks

\def\Int{\int}
\def\dim{\operatorname{dim}}

\def\Fr{\operatorname{F}}
\def\sup{\operatorname{sup}}
\def\max{\operatorname{max}}
\def\min{\operatorname{min}}

\def\num{\operatorname{N}}
\def\mul{\operatorname{m}}
\def\Leb{\operatorname{Leb}}

\def\ent{\operatorname{h}}

\def\Ar{\operatorname{A}}
\def\cl{\operatorname{cl}}
\def\diam{\operatorname{diam}}
\def\Bow{\operatorname{B}}
\def\Par{\operatorname{P}}

\def\Lip{\operatorname{Lip}}
\def\Int{\operatorname{int}}

\newcommand{\ep}{\varepsilon}

\newtheorem{thm}{Theorem}[section]
\newtheorem{cor}[thm]{Corollary}

\newtheorem{lemma}[thm]{Lemma}
\newtheorem{prop}[thm]{Proposition}

\newtheorem*{thm*}{Theorem}

\theoremstyle{definition}
\newtheorem{definition}[thm]{Definition}

\theoremstyle{remark}
\newtheorem{remark}[thm]{Remark}

\newtheorem{example}[thm]{Example}

\newtheorem*{Acknowledgements}{Acknowledgements}

\numberwithin{equation}{section}

\tikzstyle{vertex}=[circle]
\tikzstyle{goto}=[->,shorten >=1pt,>=stealth,semithick]

\begin{document}
\setcounter{section}{0}
\title{Ahlfors regularity and fractal dimension of Smale spaces}
\author[Dimitris Michail Gerontogiannis]{Dimitris Michail Gerontogiannis}
\address{Dimitris Michail Gerontogiannis, School of Mathematics and Statistics, University of Glasgow, University Place, Glasgow Q12 8QQ, United Kingdom}
\email{d.gerontogiannis@hotmail.com}
\thanks{This research was supported by EPSRC grants NS09668/1 and M5086056/1.}

\begin{abstract}
We prove that, up to topological conjugacy, every Smale space admits an Ahlfors regular Bowen measure. Bowen's construction of Markov partitions implies that Smale spaces are 
factors of topological Markov chains. The latter are equipped with Parry's measure which is Ahlfors regular. By extending Bowen's construction we create a tool for 
transferring the Ahlfors regularity of the Parry measure down to the Bowen measure of the Smale space. An essential part of our method uses a refined notion of approximation graphs over compact metric spaces. Moreover, we obtain new estimates for the Hausdorff, box-counting and Assouad dimensions of a large class of Smale spaces.
\end{abstract}

\maketitle

\section{Introduction}\label{sec:intro}
Ahlfors regular measures have been fundamental to the study of fractal structures. In particular, if an iterated function system has the open set condition then it admits an Ahlfors regular measure \cite{MT}. In the 1960's Parry \cite{Parry}, using the Perron-Frobenius Theorem, constructed an Ahlfors regular measure for topological Markov chains. In the 1970's Bowen \cite{Bowen2,Bowen3,Bowen4}, using Markov partitions, showed that topological Markov chains provide a combinatorial model of arbitrary precision for Smale spaces. However, Smale spaces themselves have not been shown to admit an Ahlfors regular measure in general, and the purpose of this paper is to bridge this gap.

The other known examples of Smale spaces with an Ahlfors regular measure are the (Euclidean) mixing parts of $C^{1+\varepsilon}$-conformal Axiom A systems \cite{Pesin}. Here we develop tools that can also be applied to the study of non-Euclidean Smale spaces, like Wieler solenoids \cite{Wieler}, since the existing techniques are restricted to the Euclidean setting. One of our main results regarding Ahlfors regularity is the following.

\begin{thm*}[see Corollary \ref{cor:TopconjAhlfors}]\label{cor:TopconjAhlforsIntro}
Any mixing Smale space is topologically conjugate to a mixing Smale space on which Bowen's measure of maximal entropy is Ahlfors regular.
\end{thm*}
\enlargethispage{\baselineskip}
\enlargethispage{\baselineskip}

A Borel measure $\mu$ on a compact metric space $(X,d)$ is Ahlfors $s$-regular if it is of the order $r^s$ on every closed ball of radius $r$. In this case, the measure $\mu$ is comparable to the $s$-dimensional Hausdorff measure and the typically distinct Hausdorff, box-counting and Assouad dimensions of $(X,d)$ are equal to $s$, see Subsection \ref{sec:Premdimensiontheory}. Metric spaces equipped with Ahlfors regular measures provide an abstract framework for the tools of harmonic analysis to be applied since, in particular, they are uniformly perfect and doubling, two very useful properties in analysis (see \cite[Chapter 11]{Heinonen} and the numerous references therein). In addition, Lebesgue's 
Differentiation Theorem holds \cite[Theorem 5.2.6]{AT}. Moreover, Ahlfors regularity lies deep in the heart of fractional calculus \cite{Ahlfors,SKM} and opens a window to apply Connes' noncommutative machinery \cite{Connes_Book} in the study of metric spaces and dynamical systems. A major example comes from the recent work of Goffeng and Mesland \cite{GM2} where they used the theory of Riesz potentials \cite{Zahle} to study $C^*$-algebras build from full-shifts \cite{Cuntz}.

Smale spaces were defined by Ruelle \cite{Ruelle} as models for hyperbolic dynamical systems and provide a topological description of the typically fractal-like non-wandering parts of Smale's Axiom A systems \cite{Smale}. Roughly speaking, a Smale space is a compact metric space $(X,d)$ together with a homeomorphism $\varphi$ having exponential contraction and expansion behaviour. Basic examples are the non-wandering parts of Axiom A systems, hyperbolic toral automorphisms, subshifts of finite type and the solenoidal attractors defined by Williams \cite{Williams} that were later extended by Wieler \cite{Wieler}. In this paper we focus on non-wandering Smale spaces, which due to Smale's Decomposition Theorem \cite[Section~7.4]{Ruelle} can be studied through their irreducible or mixing parts. 

Mixing Smale spaces are equipped with Bowen's measure \cite{Bowen4}; the unique invariant probability measure that is ergodic and maximises the topological entropy. Bowen's measure defined on topological Markov chains coincides with the quite tractable Parry measure \cite{Parry}.
A deeper connection between these two measures comes from Bowen's construction of Markov partitions of arbitrarily small diameter \cite{Bowen2}. With this, given a mixing Smale space $(X,d,\varphi)$ one can build a topological Markov chain $(\Sigma,\rho,\sigma)$ and a factor map $\pi:(\Sigma,\rho,\sigma)\to (X,d,\varphi)$. Among many nice properties (see Theorem \ref{thm:Bowen factor map}), the map $\pi$ becomes a measure-theoretic isomorphism when both Smale spaces are equipped with Bowen's measure. Therefore, Parry's measure provides a combinatorial description of Bowen's measure.

Using the Perron-Frobenius Theorem one can show that Parry's measure on $(\Sigma,\rho,\sigma)$ is Ahlfors regular, see Subsection \ref{sec:TMC}. This fact is straightforward, mainly because $\Sigma$ is equipped with an ultrametric $\rho$. On the other hand, Bowen's measure on $(X,d,\varphi)$ is not necessarily Ahlfors regular, see Remark \ref{rem:non_Ahlfors_regular} about horseshoes in $\mathbb R^3$ and $\mathbb R^4$. However, if the metric $d$ is homogeneous enough, for instance exhibits self-similarity, in Theorem \ref{thm:AhflorsregularitySmalespaces} we show that it is possible to transfer the Ahlfors regularity of the Parry measure down to the Bowen measure using the factor map $\pi$.

The main tool for transferring Ahlfors regularity using factor maps is Theorem \ref{thm:theoremgraphSmalespaces} which essentially provides a way to approximate the metric structure of Smale spaces. This theorem extends the work of Bowen on Markov partitions. More precisely, given a Smale space $(X,d,\varphi)$ equipped with a Markov partition, using the dynamics, we build a refining sequence of open covers $(\mathcal{V}_n)_{n\geq 0}$ of $X$ with diameters converging to zero. The sequence $(\mathcal{V}_n)_{n\geq 0}$ encodes various topological properties of the dynamical system from which the most important is derived from the Neighbouring Rectangles Lemma \ref{lem:neighbours}. Without having any assumption on the metric $d$, this lemma implies that 
\begin{equation}\label{eq:NRL}
\sup\limits_n \max\limits_{V\in \mathcal{V}_n}\#\{W\in \mathcal{V}_n: W\cap V\neq \varnothing\}<\infty,
\end{equation}
which manifests that Smale spaces are quite homogeneous on a topological level. 

Moreover, the sequence $(\mathcal{V}_n)_{n\geq 0}$ encodes the metric $d$. Depending on the behaviour of $\varphi$, it is possible to estimate the rate of decay of the Lebesgue covering numbers and the diameters of $\mathcal{V}_n$ as $n$ goes to infinity. The best estimates can be obtained in the case where $(X,d,\varphi)$ has self-similar dynamics, which occurs when both Lipschitz constants of $\varphi$ and $\varphi^{-1}$ are equal to the contraction/expansion constant of $(X,d,\varphi)$. In this case, the uniform upper bound in (\ref{eq:NRL}) yields that for every $r\in (0,\diam(X))$ with $n_r=\min\{n\in \mathbb N: \diam(V) \leq r, \> \text{for every}\> V\in \mathcal{V}_n\}$ it holds that
\begin{equation}\label{eq:NRL2}
\sup\limits_{x\in X} \sup\limits_{r}\#\{W\in \mathcal{V}_{n_r}: W\cap \overline{B}(x,r)\neq \varnothing\}<\infty.
\end{equation} 
This uniformity should be interpreted as a homogeneity condition of the metric space at every scale.

In Theorem \ref{thm:AhflorsregularitySmalespaces} we prove that for every homogeneous enough Smale space (semi-conformal that satisfies  condition (\ref{eq:NRL2}), see Section \ref{sec:SemiconformalSmalespaces}) the Bowen measure is Ahlfors regular. As a consequence, the Hausdorff, box-counting and Assouad dimensions coincide and have a nice description. Moreover, this makes it possible to obtain new estimates for the Hausdorff and box-counting dimensions (see Corollary \ref{cor:Hausdorffdimension}) of Smale spaces, where $\varphi$ and $\varphi^{-1}$ are Lipschitz maps (with some restrictions).

At this point, a reasonable question to ask is whether there exist many homogeneous Smale spaces for which Theorem \ref{thm:AhflorsregularitySmalespaces} holds. The answer to this is positive and can be derived from the recent work of Artigue \cite{Artigue} who, based on the work of Fried  \cite{Fried} and Fathi \cite{Fathi}, constructed self-similar metrics for expansive dynamical systems. More precisely, any Smale space $(X,d,\varphi)$ admits a self-similar metric $d'$, which induces the same topology as $d$, so that $(X,d',\varphi)$ is again a Smale space with self-similar dynamics. 

There is a vast literature on Ahlfors regularity, and hyperbolicity plays a prominent role in its study. A basic example in the setting of Smale spaces is the work of Ma$\tilde{\text{n}}${\'e} \cite{Mane2} about regularity properties of the Hausdorff dimension of horseshoes of closed surface $C^2$-diffeomorphisms. In this work, Ma$\tilde{\text{n}}${\'e} calculates the Hausdorff dimension of a horseshoe $\Lambda$ by constructing an Ahlfors regular measure on it. First, he notes that the local product structure on $\Lambda$ is bi-Lipschitz, since the stable and unstable leaves in the surface are one-dimensional and the diffeomorphism is $C^2$, hence the stable and unstable foliations extend to $C^1$-foliations on a neighbourhood of $\Lambda$. Using the same philosophy and Theorem 22.1 of Pesin \cite{Pesin}, one can construct (product) Ahlfors regular measures on the mixing components of $C^{1+\varepsilon}$-conformal Axiom A systems. These are measures of maximal dimension associated to geometric H{\"o}lder potentials on stable and unstable sets, and do not necessarily maximise the topological entropy, see \cite[Section 5]{Barreira_DR} and \cite[Section 24]{Pesin}.

Although it is not proved by Pesin or, to our knowledge, anywhere else in the literature, it seems possible that arguments similar to \cite{Pesin} can be used to prove that the Bowen measure on the aforementioned hyperbolic sets is Ahlfors regular. However, these arguments are limited to Smale spaces that lie in an Euclidean space that have been constructed by smooth dynamics. Specifically, the Ahlfors regularity of the stable and unstable sets in \cite[Theorem 22.1]{Pesin} requires a crucial volume argument that makes sense in Euclidean spaces; for every $n\in \mathbb N$ and $c\in (0,1]$ there is $k\in \mathbb N$ such that every ball in $\mathbb R^n$ of radius $r>0$ can fit up to $k$ disjoint balls of radius $cr$. This is strongly related to Moran constructions in Euclidean spaces, see \cite[Section 13]{Pesin} and \cite[Section 2.1.2]{Barreira}. Moreover, this volume argument is made clear in equation (4.18) of \cite[Theorem 4.1.8]{Barreira_DR}. Further, without smoothness, the local product structure is not guaranteed to be bi-Lipschitz (see the proof in \cite[Theorem 6.2.9]{Barreira_TF} and the discussion after \cite[Theorem 6.3.2]{Barreira_TF}). 

In the present paper we study general Smale spaces which are typically non-Euclidean, like Wieler solenoids, or might not be derived from the smooth dynamics of Smale's Axiom A systems. Roughly speaking, conformality is replaced by self-similarity, and volume-type arguments (in possibly non-Euclidean Smale spaces) can be made using conditions (\ref{eq:NRL}) and (\ref{eq:NRL2}). While Markov partitions play a crucial role both in \cite{Pesin} and in the present paper, our work is not an extension of the aforementioned techniques in \cite{Pesin}. For instance, we do not consider Ahlfors regular measures on stable and unstable sets or bi-Lipschitz conditions on the local product structure of Smale spaces. Further, we obtain results for the Bowen measure in a direct way. 

Pesin and Weiss studied equilibrium measures for conformal continuous expanding maps. In \cite{PW} they proved that every equilibrium measure for a H{\"o}lder potential of a H{\"o}lder conformal expanding map $g:X\to X$, where $X$ is a compact metric space, has the doubling property, see Definition \ref{def:doublingmeasure}. Their main ingredient was a family of Markov partitions that was adapted to the points of $X$. More precisely, they showed that there are $c_1,c_2>0,\, k\in \mathbb N$ so that for any small enough $r>0$ and $x\in X$ there exists a Markov partition $\mathcal{R}_x=\{R_1,\ldots , R_m\}$ for the map $g^k$ with $\diam(R_i)\leq c_2r$, for $i=1,\ldots, m$, and the ball $B(x,c_1r)\subset R(x)$, where $R(x)$ denotes the unique $R(x)\in \mathcal{R}_x$ that contains $x$. In other words, for each $x\in X$ they were creating Markov partitions of arbitrarily small diameters centred around the point $x$. This remarkable doubling property holds also in the setting of $C^{1+\varepsilon}$-conformal Axiom A systems \cite[Prop. 24.1]{Pesin}. We should note again that our methods and techniques differ from theirs, since for a given Smale space we only require a single Markov partition instead of an adapted family of Markov partitions. The uniformity condition (\ref{eq:NRL2}) plays again an important role in this fact. Finally, an Ahlfors regular measure is doubling but the converse is not necessarily true.

Another major example of Ahlfors regularity in the general framework of hyperbolic dynamical systems is the work of Patterson \cite{Patterson}, where he constructed measures on the limit sets of finitely generated Fuchsian groups of the second kind. In many cases such measures are Ahlfors regular, for instance if there are no parabolic elements. However, in the presence of parabolic elements this is not necessarily true \cite[Corollary 2.5]{Fraser3}. His work was later extended by Sullivan \cite{Sullivan} in the framework of Kleinian groups and such measures are now known as Patterson-Sullivan measures. A few years later their results were extended by Coornaert \cite{Coornaert} in the framework of Gromov hyperbolic groups. Our result about Ahlfors regularity of Smale spaces is essentially complementary to the results of Patterson, Sullivan and Coornaert, when viewing Smale spaces as fractal-like structures in negatively curved spaces. In fact, Patterson-Sullivan measures are strongly related with Bowen's measure for a geodesic flow on the unit tangent bundle of a compact negatively curved Riemanian manifold \cite{Kaimanovich}. In this setting, Hamenst{\"a}dt \cite{Hamenstadt} showed that Bowen's measure defined on stable or unstable sets, when equipped with an appropriate metric, is Ahlfors regular. Hasselblatt \cite{Hasselblatt} generalised this result for Anosov flows. Nekrashevych \cite{Nekr} has obtained analogous results with Hasselblatt for hyperbolic groupoids.

In the context of fractal dimensions, the idea of using Markov partitions to calculate Hausdorff dimension can be traced back to the work of Bowen for quasi-circles \cite{Bowen5}. Ruelle extended the work of Bowen to mixing $C^{1+\varepsilon}$-conformal repellers \cite{Ruelle2}. In the same spirit, McCluskey and Manning \cite{MM} calculated Hausdorff dimensions of basic sets of surface $C^2$-diffeomorphisms satisfying Axiom A. Takens \cite{Takens} improved the latter result by showing that the Hausdorff dimension coincides with the box-counting dimension. Palis and Viana extended Takens' result to $C^1$-diffeomorphisms \cite{PViana}. 

A few years later, Barreira \cite{Barreira} studied the dimension theory of mixing components of $C^{1+\varepsilon}$-Axiom A systems. In the conformal case he showed that the Hausdorff and box-counting dimensions coincide \cite[Theorem 4.3.2]{Barreira_DR}, obtaining in this way Pesin's result \cite[Theorem 22.2]{Pesin}, but with an alternative proof. Nevertheless, he extended this result to (not necessarily Euclidean) Smale spaces with bi-Lipschitz local product structure, and with asymptotically conformal dynamics on stable and unstable sets \cite[Theorem 3.15]{Barreira}. While we do not use Barreira's results, the coincidence of the Hausdorff and box-counting dimensions for self-similar Smale spaces (Theorem \ref{thm:AhflorsregularitySmalespaces}) can be also obtained by \cite[Theorem 3.15]{Barreira}. This is proved in Remark \ref{rem:Hausdorff_equals_box}. However, our techniques are different since we do not consider the dimension on stable and unstable sets or use bi-Lipschitz local product structures. Further, Barreira's arguments are particularly designed to compute Hausdorff and box-counting dimensions, and cannot give Ahlfors regularity results, even for the Hausdorff measure of the corresponding dimension, see also the proof of \cite[Theorem 4.3.2]{Barreira_DR}.

To be precise, the purpose of Theorem \ref{thm:AhflorsregularitySmalespaces} is not primarily an attempt to compute fractal dimensions of Smale spaces, but rather to understand Ahlfors regularity itself. Ahlfors regularity is a very strong homogeneity condition that is reflected in the fact that the Assouad dimension is equal to the other two dimensions. In fact, there are fractals with equal Hausdorff and box-counting dimensions but strictly larger Assouad dimension \cite[Theorem 6.4.3 (ii)]{Fraser_book}. Also, Ahlfors regularity is such a fine condition that there are fractals with equal Hausdorff, box-counting and Assouad dimensions that do not have Ahlfors regular measures \cite[Theorem 6.4.3 (i)]{Fraser_book}. In addition, Theorem \ref{thm:AhflorsregularitySmalespaces} concerns a specific measure (Bowen measure) and not just the Hausdorff measure of the corresponding dimension. Finally, our Ahlfors regularity result might shed some light on embeddability problems for Smale spaces, for instance, characterising Smale spaces that admit bi-Lipschitz embeddings in Euclidean spaces, see Subsection \ref{sec:Premdimensiontheory}. Although this problem (for general metric spaces) is extremely difficult to solve \cite{Heinonen}, we would be very interested to know whether all self-similar Smale spaces admit such embeddings, and which of them admit Poincar{\'e} inequalities.

Let us present a variety of applications of our results. First, our main results will be used in forthcoming papers regarding the noncommutative topology of Smale space $C^*$-algebras \cite{KPW,Putnam_algebras,PS}. This work is related to the PhD project of the author. Second, from Proposition \ref{prop:mainresultMarkovpartitions} we obtain a refining sequence of Markov partitions $(\mathcal{R}_n)_{n\in \mathbb N}$ that encodes topological properties of the Smale space. Such refining sequences are related to the existence of projective covers by Markov partitions over Smale spaces, see \cite[Remark 4.2]{Valerio}. Such projective covers can be used in the study of Putnam's homology \cite{Putnam_Book} for Smale spaces, see Theorem 4.3 and Remark 4.4 in \cite{Valerio}. In a subsequent paper we intend to use the results of the present paper and study Smale spaces from a noncommutative point of view, using Riesz potentials and Laplace-Beltrami operators in the sense of Bellisard and Pearson \cite{BP}.

We now conclude the introduction by outlining the main results we obtain in each section. In Section \ref{sec:Approx_metric_spaces} we start with some preliminaries in topological dynamical systems and dimension theory. Then we introduce and study in detail the notion of an approximation graph which provides a convenient way to study refining sequences. Approximation graphs have also been considered in \cite{JKS,Palmer,BP} to some extent. Also, we investigate structural properties of approximation graphs and their behaviour under dynamics. Section \ref{sec:Geometric approximations and embeddings} introduces the concept of geometric approximation graphs and a sufficient condition for a compact metric space to have finite Assouad dimension. 

\enlargethispage{\baselineskip}

In Section \ref{sec:Smalespaces} we provide a basic introduction to Smale spaces and present a detailed proof that the Parry measure is Ahlfors regular. Moreover, we discuss the work of Fried, Fathi and Artigue on metrics of expansive dynamical systems and make an observation that leads to new dimension estimates for Smale spaces with Lipschitz homeomorphisms. In Section \ref{sec:Markovpartitions} we start by introducing the notion of a Markov partition and then we show how to construct an approximation graph from a given Markov partition. The structural properties of such an approximation graph are presented in Proposition \ref{prop:mainresultMarkovpartitions}. One of the key tools of the whole paper is the Neighbouring Rectangles Lemma \ref{lem:neighbours}. 

In Section \ref{sec:Geometric approximations of Smale spaces}, using Markov partitions, we build refining sequences of open covers that allow us to transfer the Ahlfors regularity of the shift space down to the Smale space. For such refining sequences we  study the multiplicities, cardinalities and rates of decay of Lebesgue covering numbers and diameters of the covers. All these results establish Theorem \ref{thm:theoremgraphSmalespaces}. Finally, in Section \ref{sec:SemiconformalSmalespaces} we study Smale spaces with some degree of homogeneity, namely, we introduce the concept of semi-conformal Smale spaces which include self-similar Smale spaces and Wieler solenoids. For semi-conformal Smale spaces satisfying (\ref{eq:NRL2}), we prove Theorem \ref{thm:AhflorsregularitySmalespaces} and obtain Corollary \ref{cor:TopconjAhlfors}. The section concludes with the dimension estimates of Corollary \ref{cor:Hausdorffdimension}.

\begin{Acknowledgements} I would like to express my sincere gratitude to my  supervisors, Mike Whittaker and Joachim Zacharias, for their constant support and the endless hours of fruitful discussions during my doctoral studies. I would also like to thank Ian Putnam for hosting me at the University of Victoria during the Autumn of 2019, where a great deal of this work was conceived. Finally, I am greatly indebted to the anonymous referee, who I would like to thank for many very helpful comments which improved the paper considerably.
\end{Acknowledgements}
 
\section{Approximations of compact metric spaces and dynamics}\label{sec:Approx_metric_spaces}

Let $Z$ be an infinite topological space, $\psi:Z\to Z$ be a continuous map and denote the corresponding dynamical system by ($Z,\psi$). If $Z$ is equipped with a metric $d$, the dynamical system will be denoted by $(Z,d,\psi)$. However, if there is no risk of confusion the notation will be reduced to $(Z,\psi)$.   
\subsection{Preliminaries on topological dynamics}\label{sec:TDS} If $Z$ is compact, the topological entropy of ($Z,\psi$) is defined using open covers in the following way. Let $\mathcal{U}$ be a finite open cover of $Z$ and $\num(\mathcal{U})$ denote the minimal cardinality of a subcover. By continuity of $\psi$ we have that $\psi^{-1}(\mathcal{U})=\{\psi^{-1}(U):U\in \mathcal{U}\}$ is also an open cover. Then using the  \textit{joint cover} notation 
\begin{equation}
\mathcal{W}\vee \mathcal{W}'=\{W\cap W':W\in \mathcal{W}, W'\in \mathcal{W}'\},
\end{equation}
for covers $\mathcal{W}$ and $\mathcal{W}'$ of $Z$, one can prove (see \cite{AKM}) that the following limit exists and is finite
\begin{equation}
\ent(\psi, \mathcal{U})= \lim_{n\to \infty}\frac{1}{n}\log \num( \bigvee_{i=0}^{n-1} \psi^{-i}(\mathcal{U})).
\end{equation}
\begin{definition}{(\cite{AKM})}\label{def:entropy defn}
The \textit{topological entropy} of ($Z,\psi$) is defined by $$\ent(\psi)= \sup\limits_{\mathcal{U}}\ent(\psi, \mathcal{U}),$$ where the supremum is taken over all open covers of $Z$.
\end{definition}

We will be interested in dynamical systems with topological recurrence conditions. 

\begin{definition}\label{def:Recurrence}
Let ($Z,\psi$) be a dynamical system. 
\begin{enumerate}[(1)]
\item A point $z\in Z$ is called \textit{non-wandering} if for every open neighbourhood $U$ of $z$ there is some $n\in \mathbb N$ such that $\psi^n(U)\cap U\neq \varnothing$. Moreover, we say that $(Z,\psi)$ is \textit{non-wandering} if every $z\in Z$ is non-wandering.
\item ($Z,\psi$) is called \textit{irreducible} if for every ordered pair of non-empty open sets $U,V\subset Z$, there is some $n\in \mathbb N$ such that $\psi^n(U)\cap V\neq \varnothing$.
\item ($Z,\psi$) is called \textit{mixing} if for every ordered pair of non-empty open sets $U,V\subset Z$, there is some $N\in \mathbb N$ such that $\psi^n(U)\cap V\neq \varnothing$, for every $n\geq N$.
\end{enumerate}
\end{definition}

The following is a simple consequence of irreducibility. Recall that $Z$ is infinite.

\begin{prop} \label{thm:no isolated points}
If $(Z,\psi)$ is irreducible and $Z$ is Hausdorff then $Z$ has no isolated points.
\end{prop}

We introduce some notation that will be used later. First let $\#S$ denote the cardinality of any finite set $S$. Suppose that $(Z,d)$ is a compact metric space. For a finite cover $\mathcal{U}$ of $Z$, for which $\varnothing, Z\not \in \mathcal{U}$, and $Y\subset Z$ we will be interested in the quantities 
\begin{align*}
\hypertarget{Q1}{\tag{Q1}} \overline{\text{diam}}(\mathcal{U})&=\max\limits_{U\in \, \mathcal{U}}\text{diam}(U)\\
\hypertarget{Q1}{\tag{Q2}}\underline{\text{diam}}(\mathcal{U})&=\min\limits_{U\in \, \mathcal{U}}\text{diam}(U)\\
\hypertarget{Q1}{\tag{Q3}}\text{N}_{\mathcal{U}}(Y)&=\{U\in \mathcal{U}: U\cap Y\neq \varnothing\}\\
\hypertarget{Q1}{\tag{Q4}}\mul(\mathcal{U})&= \max \{n: U_{i_1}\cap \ldots \cap U_{i_n}\neq \varnothing\}\\
\intertext{where $U_{i_1}, \ldots , U_{i_n}\in \mathcal{U}$ are different, and in the case where $\mathcal{U}$ is open we will also consider}
\hypertarget{Q1}{\tag{Q5}}\text{Leb}(\mathcal{U})&= \min\limits_{z\in Z} \max\limits_{U\in \mathcal{U}}d(z,Z\setminus U)>0.
\end{align*}
The last quantity is the Lebesgue covering number of $\mathcal{U}$ and it holds that for every $z\in Z$ there is some $U\in \mathcal{U}$ so that the ball $B(z,\ell)\subset U$, where $\ell=\text{Leb}(\mathcal{U})$.

\enlargethispage{\baselineskip}
\enlargethispage{\baselineskip}

\subsection{Preliminaries on dimension theory}\label{sec:Premdimensiontheory} We introduce several types of dimension and comment on their relationship with one another. For further details see \cite{Falconer, MT}. 
\begin{definition}{(\cite[Def. I.4]{Nagata})} We say that $Z$ has \textit{covering dimension} $\text{dim}Z\leq n$ if every finite open cover $\mathcal{U}$ has an open refinement $\mathcal{W}$ with $\mul(\mathcal{W})\leq n+1$. We say that $\text{dim}Z= n$ if it is true that $\text{dim}Z\leq n$ and it is false that $\text{dim}Z\leq n-1$.
\end{definition}
Suppose now that $Z$ has a metric $d$. If $\{U_i\}$ is a countable (or finite) cover of $Z$ with diameter at most $\delta$, we say that $\{U_i\}$ is a $\delta$-\textit{cover} of $Z$. The Hausdorff measure and Hausdorff dimension are defined as follows. Let $s\geq 0$ and for every $\delta >0$ define 
\begin{equation}\label{eq:Hausdorffmeasure1}
\mathcal{H}^s_{\delta}(Z)=\inf \{\sum_i (\diam (U_i))^s: \{U_i\}\thinspace \thinspace \text{is a}\thinspace \thinspace \delta \text{-cover of}\thinspace \thinspace Z\}.
\end{equation}
One can show that the limit 
\begin{equation}\label{eq:Hausdorffmeasure2}
\mathcal{H}^s(Z)=\lim\limits_{\delta \to 0}\mathcal{H}^s_{\delta}(Z)
\end{equation}
exists and that $\mathcal{H}^s$ defines a measure, see \cite[Section 2.1]{Falconer}.
\begin{definition}{(\cite[Section 2.1]{Falconer})}\label{def:Hausdorffmeasure3}
We call $\mathcal{H}^s$ the s-\textit{dimensional Hausdorff measure} on $(Z,d)$.
\end{definition}
\begin{definition}{(\cite[Section 2.2]{Falconer})}
The \textit{Hausdorff dimension} of $(Z,d)$ is defined as $$\dim_{H}Z=\inf \{ s\geq 0: \mathcal{H}^s(Z)=0\}.$$
\end{definition}
Let $\text{N}_{\delta}(Z)$ be the smallest number of sets of diameter at most $\delta>0$ which can cover $Z$. 
\begin{definition}{(\cite[Section 3.1]{Falconer})}
The \textit{lower} and \textit{upper box-counting dimensions} of $(Z,d)$ are defined as
\begin{align*}
\underline{\dim}_BZ&=\liminf\limits_{\delta \to 0}\frac{\log\text{N}_{\delta}(Z)}{-\log \delta}\\
\overline{\dim}_BZ&=\limsup\limits_{\delta \to 0}\frac{\log\text{N}_{\delta}(Z)}{-\log \delta}.
\end{align*}
If these are equal, their common value is called the \textit{box-counting dimension} and denoted by $\dim_BZ$. As we will shortly see, in many interesting cases the box-counting dimension coincides with the Hausdorff dimension.
\end{definition}
Another important metric dimension was introduced by Assouad in \cite{Assouad1, Assouad2, Assouad3} in the framework of bi-Lipschitz embeddability of metric spaces into Euclidean spaces. For a thorough exposition see \cite{Luk}. Moreover, there is a vast literature on Assouad dimension, see \cite{Fraser1, Mackay, Olsen}. We say that $(Z,d)$ is \textit{bi-Lipschitz embeddable} into some $\mathbb R^N$ if there exists a bi-Lipschitz map $f:(Z,d)\to \mathbb R^N$. Any such metric space should have the following doubling property \cite[Lemma 9.4]{Robinson}.
\begin{definition}\label{def:doublingdefinition}{(\cite[p.84]{Robinson})}
A metric space $(Z,d)$ is called $K$-\textit{doubling}, where $K\geq 1$, if every ball of radius $2r$ can be covered by $K$ balls of radius $r$, where $K$ is independent of $r$.
\end{definition}
Assouad, in an attempt to study the converse question, that is, whether every $K$-doubling metric space admits a bi-Lipschitz embedding into some $\mathbb R^N$, obtained the following.
\begin{thm}[{\cite[Proposition 2.6]{Assouad3}}]\label{thm:Assouad}
Let $(Z,d)$ be a $K$-doubling space. For every $\varepsilon \in (0,1)$ there is a bi-Lipschitz embedding $f:(Z,d^{\varepsilon})\to \mathbb R^N$, for some $N\in \mathbb N$ that depends on $K$ and $\varepsilon$.
\end{thm}
We note that $d^{\varepsilon}$ is the metric defined by $d^{\varepsilon}(z,y)=d(z,y)^{\varepsilon}$. Hence Assouad's Theorem does not offer a bi-Lipschitz embedding of the actual space but of a \textit{snowflaked} version of it. However, it turns out that there are doubling spaces which do not admit bi-Lipschitz embeddings \cite{LS,Semmes}. The dependence of $N$ on $\varepsilon$ has been studied in \cite{NN}.
\begin{definition}[\cite{Assouad3}]\label{def:Assouad}
Let $(Z,d)$ be a metric space. Suppose $s\geq 0$ and $C\geq 0$ are numbers such that 
\begin{equation*}\label{eq:Assouadcondition}
\# Y \leq C(b/a)^s
\end{equation*}
wherever $0<a\leq b$ and $Y\subset Z$ is a finite subset with $a\leq d(y,y')\leq b$ if $y,y' \in Y$ and $y\neq y'$. Then $Z$ is called $(C,s)$-\textit{homogeneous}. We say that $Z$ is $s$-\textit{homogeneous} if it is $(C,s)$-\textit{homogeneous} for some $C$. The \textit{Assouad dimension} is defined to be
\begin{equation*}\label{eq:Assouaddimension}
\text{dim}_{\text{A}}Z=\text{inf}\{s\in [0,\infty): Z \enspace \text{is s-homogeneous}\}.
\end{equation*}
\end{definition}
It is straightforward to show that $\text{dim}_{\text{A}}Z$ is finite if and only if $Z$ is $K$-doubling \cite[Lemma 9.4]{Robinson}. Specifically, if $Z$ is $(C,s)$-homogeneous then it is $C2^s$-doubling. Before passing to the interplay of measure theory and dimension theory let us summarise the known relations  between the dimensions discussed so far.
\begin{prop}[{\cite{Falconer, MT}}]\label{prop:alldimensions}
For a totally bounded metric space $(Z,d)$ it holds
\begin{align*}
\dim Z \leq \dim_H Z \leq \underline{\dim}_BZ\leq \overline{\dim}_BZ\leq \dim_A Z.
\end{align*}
\end{prop}
A class of measures with an important role in the study of metric spaces is the following.
\begin{definition}[{\cite[Def. 4.1.2]{MT}}]\label{def:doublingmeasure}
A Borel measure $\mu$ on $(Z,d)$ is called $D$-\textit{doubling}, where $D\geq 1$, if 
$$0< \mu(\overline{B}(z,2r))\leq D\mu(\overline{B}(z,r))<\infty$$
for every $z\in Z$ and $r\in [0,\diam Z)$.
\end{definition}
It turns out that for a complete metric space the doubling property is equivalent to the existence of a doubling measure, see \cite[Section 13]{Heinonen}. The doubling property is not uncommon, for instance if $Z$ is a separable metrizable space with $\dim Z<\infty$ then there is a totally bounded metric so that $\dim_A Z=\dim Z$ \cite[Theorem 4.3]{Luk}. This is also true to some extent for doubling measures \cite[Theorem 4.5]{Luk}. 

Significantly more regular measures can be constructed on spaces that exhibit self-similarity, like the middle-third Cantor set or, more generally, sets generated by iterated function systems satisfying the open-set condition, see \cite{MT}. A prominent case of measures with the doubling property are the Ahlfors regular measures.
\begin{definition}{(\cite[Def. 1.4.13]{MT})}
A metric space $(Z,d)$ is \textit{Ahlfors s-regular} for some $s>0$ if there exists a Borel measure $\mu$ on $Z$  and some $C>0$ so that $$C^{-1}r^s\leq \mu(\overline{B}(z,r))\leq Cr^s$$ for every $z\in Z$ and $r\in [0,\diam Z)$. Such measure $\mu$ is called \textit{Ahlfors s-regular} (or Ahlfors regular).
\end{definition}
\begin{remark}\label{rem:Ahlforsrem} If $\mu$ is an Ahlfors $s$-regular measure on $(Z,d)$ then it is comparable to the $s$-dimensional Hausdorff measure $\mathcal{H}^s$, in the sense that $\mu$ is within constant multiples of $\mathcal{H}^s$. Therefore, $\mathcal{H}^s$ is strictly positive. A typical example of an Ahlfors regular space is the classical Sierpinski gasket on which the $\log3/\log2$-dimensional Hausdorff measure is Ahlfors $\log3/\log2$-regular, see \cite[Example 8.3.4]{MT}.
\end{remark}
\begin{prop}[{\cite[Section 1.4]{MT}}]\label{prop:equalityofdimensions}
If the metric space $(Z,d)$ is Ahlfors s-regular then $\dim_HZ=\dim_AZ=s$. 
\end{prop}

\subsection{Refining sequences}\label{sec:refining}
The next concept provides a way to study topological or dynamical properties on a compact metric space by means of finite approximations. It was first introduced in \cite[Cor. p.314]{AKM} as a way to study topological entropy. Here we adjust it to our needs.
\begin{definition}\label{def:refining sequence}
Let $(Z,d)$ be a compact metric space. A sequence $(\mathcal{V}_n)_{n\geq 0}$ of finite covers of $Z$, which are either all open or all closed with non-empty interiors, is called \textit{refining} if, $\mathcal{V}_0=\{Z\}$ and for every $n\geq 0$ any element of $\mathcal{V}_{n+1}$ lies inside some element of $\mathcal{V}_n$, so that $$\lim\limits_{n\to \infty}\overline{\text{diam}}(\mathcal{V}_n)=0.$$ 
\end{definition}

It is straightforward to check that, in the case of open covers, the collection $\bigcup_{n\in\mathbb N}\mathcal{V}_n$ forms a countable basis for the topology on $Z$ and that $\underline{\text{diam}}(\mathcal{V}_n)>0$, for every $n\in \mathbb N$, if and only if $Z$ does not have isolated-points. Interesting refining sequences exist over spaces $Z$ that admit an expansive homeomorphism $\psi$; that is, there is some $\varepsilon_Z>0$ so that if $d(\psi^n(x),\psi^n(y))\leq \varepsilon_Z$ for every $n\in \mathbb Z$ then $x=y$. 
\begin{definition}[{\cite[Def. 5.10]{Walters}}]\label{def:generator}
Let $(Z,d)$ be a compact metric space and $\psi$ a homeomorphism. A \textit{generator} for $(Z,\psi)$ is a finite open cover $\mathcal{V}_1$ of $Z$ such that for each bi-infinite sequence $\{V_i\}_{i\in \mathbb Z}$ of elements of $\mathcal{V}_1$, it holds that $\bigcap_{i\in \mathbb Z} \psi^{-i}(\cl({V_i}))$ is at most one point.
\end{definition}
\enlargethispage{\baselineskip}
\enlargethispage{\baselineskip}
It turns out that the existence of a generator is equivalent to the expansiveness of the system \cite[Theorem 5.22]{Walters}. Given a generator we obtain a refining sequence $(\mathcal{V}_n)_{n\geq 0}$ of open covers by setting $\mathcal{V}_0=\{Z\}$ and for $n\in \mathbb N$
\begin{equation}\label{eq: dynamic refining sequence}
\mathcal{V}_n=\bigvee_{i=1-n}^{n-1}\psi^{-i}(\mathcal{V}_1),
\end{equation}
since $\lim\limits_{n\to \infty}\overline{\text{diam}}(\mathcal{V}_n)=0$. Also, $\text{h}(\psi, \mathcal{V}_1)=\text{h}(\psi)$. For these we refer to \cite{Walters}.

The notion of refining sequences was also used to prove that any compact metrizable space $Z$ is the quotient of a Cantor space, basically, built from the non-isolated points of $Z$ \cite{Schoenfeld}. The dynamical analogue is that any expansive dynamical system $(Z,\psi)$ is a factor of some $(\Sigma, \sigma)$ where $\Sigma$ is a compact zero-dimensional space and $\sigma$ is a homeomorphism. If $Z$ has no isolated points then $\Sigma$ will be a Cantor space, see Corollary \ref{cor:quotientofCantor}. The basic idea is that $\Sigma$ corresponds to the path space of an infinite rooted graph induced by a refining sequence, as in (\ref{eq: dynamic refining sequence}).
\enlargethispage{\baselineskip}
\subsection{Approximation graphs}\label{sec:approximationgraphs} Given a refining sequence of open or closed covers for a compact metric space $(Z,d)$ we construct a rooted graph, with vertices the sets in the covers and edges defined by inclusion of the sets in preceding refinements. Such a graph will be called an approximation graph since its infinite path space will provide an approximation of $(Z,d)$. We study how precise this approximation can be and how it behaves in the dynamical context. 

This notion was previously used by Palmer \cite{Palmer} who proved the existence of an abstract refining sequence of open covers whose approximation graph can be used to obtain the Hausdorff measure and Hausdorff dimension of $(Z,d)$. However, Palmer did not study the structure of approximation graphs nor did he study dynamics on them. Here we extend Palmer's definition by including refining sequences of closed covers and make a few adjustments that suit our needs. A related but different concept known as approximating graphs has been used in \cite{BP} for ultrametric Cantor spaces and in \cite{JKS} for transversals of substitution tilings. 

\begin{definition}
Let $(Z,d)$ be a compact metric space and $(\mathcal{V}_n)_{n\geq 0}$ be a refining sequence of $Z$. The corresponding \textit{approximation graph} is the rooted graph $\Gamma=(\mathcal{V},\mathcal{E})$ where 
\begin{enumerate}[(1)]
\item the set of vertices $\mathcal{V}$ is given by the disjoint union $\coprod_{n\geq 0} \mathcal{V}_n$;
\item the set of edges $\mathcal{E}$ is given by the disjoint union $\coprod_{n\geq 0}\mathcal{E}_n$, where $$\mathcal{E}_n=\{(v_{n+1},v_n): v_{n+1}\in \mathcal{V}_{n+1},v_{n}\in \mathcal{V}_{n}, v_{n+1}\subset v_n\}$$ is the set of edges that have sources in $\mathcal{V}_n$ and ranges in $\mathcal{V}_{n+1}$. The source map $s$ is given by $s(v_{n+1},v_n)=v_n$, the range map $r$ is given by $r(v_{n+1},v_n)=v_{n+1}$;
\item the root is $Z$.
\end{enumerate}
\end{definition}
We consider approximation graphs that have \textit{no sinks}; for every $v\in \mathcal{V}$ it holds that $s^{-1}(v)\neq \varnothing$. Also the \textit{only source is the root}; for every $v\neq Z$ we have $r^{-1}(v)\neq \varnothing$. This is because $(\mathcal{V}_n)_{n\geq 0}$ is a refining sequence. An approximation graph with these two conditions is an example of a Bratteli diagram \cite[Definition 2.1]{HPS}. A symbolic description of $Z$ comes by considering the \textit{space of infinite paths}
\begin{equation}\label{eq:spaceofpaths}
\mathcal{P}_{\Gamma}=\{\widetilde{p}=(p_n)\in \prod_n \mathcal{E}_n: s(p_{n+1})=r(p_n)\}.
\end{equation}

For a finite path $\mu=\mu_0\mu_1\ldots \mu_{\ell}$ in $\Gamma$, where each $\mu_i \in \mathcal{E}_i$, denote by $C_{\mu}$ the \textit{cylinder set} 
\begin{equation}\label{eq:cylinders}
C_{\mu}=\{\widetilde{p}\in \mathcal{P}_{\Gamma}: p_i=\mu_i, \enspace \text{for} \enspace 0\leq i\leq \ell \}
\end{equation}
which is non-empty since there are no sinks. The collection of all these sets forms a clopen basis for a compact Hausdorff topology on $\mathcal{P}_{\Gamma}$ \cite[p.18]{AEG}. Therefore, $\mathcal{P}_{\Gamma}$ is a compact zero-dimensional space.
Moreover, for $v\in \mathcal{V}_n$ let $[v]$ denote the set of paths from the root $Z$ to $v$. These sets are non-empty since there are no sources. Let 
\begin{equation}
C_{[v]}=\bigcup_{\mu \in [v]}C_{\mu}
\end{equation}
and because $[v]$ is finite, $C_{[v]}$ is clopen. The collection of these sets forms a sub-basis for the cylinder set topology on $\mathcal{P}_{\Gamma}$. Finally, let $\text{C}(v)=\{v_{n+1}\in \mathcal{V}_{n+1}:v_{n+1}\subset v\}$ denote the \textit{descendants} of $v\in \mathcal{V}_{n}$ in $\mathcal{V}_{n+1}$.
\begin{remark}
$\mathcal{P}_{\Gamma}$ is a Cantor space if and only if for every $n\geq 0$ and every $v\in \mathcal{V}_n$, there is a path from $v$ to some $w\in \mathcal{V}_m$ where $m\geq n, \#(s^{-1}(w))\geq 2$ \cite[Lemma 6.4]{AEG}. It follows that $\mathcal{P}_{\Gamma}$ is a Cantor space if and only if $Z$ has no isolated points.
\end{remark}
Due to Cantor's intersection theorem we can define the map $\pi_{\Gamma}:\mathcal{P}_{\Gamma}\to Z$ given by 
\begin{equation}\label{eq: projection map}
\widetilde{p}\mapsto \bigcap_{n\geq 0} \cl(r(p_n)).
\end{equation}
\begin{prop}
The map $\pi_{\Gamma}$ is continuous and surjective. Consequently, it is a quotient map.
\end{prop}
\begin{proof}
Continuity of $\pi_{\Gamma}$ follows because if $U$ is an open neighbourhood of $\pi_{\Gamma}(\widetilde{p})$, since $\text{diam}(r(p_n))$ tends to zero, there is some $n_0\in \mathbb N$ such that $\cl(r(p_{n_0}))\subset U$ and hence $\pi_{\Gamma}(C_{p_0\ldots p_{n_0}})\subset \cl (r(p_{n_0})) \subset U$.

The surjectivity is more interesting and is basically a consequence of K{\"o}nig's Lemma \cite[Lemma 8.1.2]{Diestel}. Let $z\in Z$ and for each $n\geq 0$ consider the sets in $\mathcal{V}_n$ that contain $z$, that is $\num_{\mathcal{V}_n}(\{z\})$. Let $\Gamma_z$ be the sub-graph of $\Gamma$ restricted on the set of vertices $\coprod_{n\geq 0}\num_{\mathcal{V}_n}(\{z\})$. It is an infinite rooted graph with vertices of finite degree and contains some infinite rooted tree $\Gamma'_z$. From K{\"o}nig's Lemma, the tree $\Gamma'_z$ has an infinite path $\widetilde{p}\in \mathcal{P}_{\Gamma}$ and $\pi_{\Gamma}(\widetilde{p})=z$. 
\end{proof}

\subsection{Essentiality of approximation graphs}\label{sec:Essentialgraphs} We introduce some structural properties of approximation graphs that can be used to construct spectral triples over compact metric spaces in the sense of Christensen and Ivan \cite{CI}.  However, we do not deal with this here. First we define the \textit{overlapping set} of $\mathcal{V}_n$ to be
\begin{equation}\label{eq:overlappingset}
\mathcal{Y}_n=\{ \cl(v_n)\cap \cl(w_n): v_n\neq w_n\in \mathcal{V}_n\},
\end{equation}
and its \textit{essential part} to be
\begin{equation}\label{eq:essentialpartofset}
\mathcal{V}_n^{\text{ess}}=\{v_n^{\text{ess}}:v_n\in \mathcal{V}_n\} \enspace \text{where}\enspace v_n^{\text{ess}}=\text{int}(v_n)\setminus \bigcup\mathcal{Y}_n.
\end{equation}
\begin{definition}\label{def:essentialgraph}
An approximation graph $\Gamma$ is called \textit{regular} if for every $n\in \mathbb N$ and $v_n\in \mathcal{V}_n$ we have
\begin{equation}\label{eq:E1}
v_n = \bigcup \{v_{n+1} \in \mathcal{V}_{n+1}: v_{n+1}\subset v_n\}\tag{E1} \enspace \text{and}
\end{equation}
\begin{equation}\label{eq:E2}
v_n^{\text{ess}}\neq \varnothing.\tag{E2}
\end{equation}
If $\Gamma$ consists of closed covers and $\bigcup \mathcal{V}_n^{\text{ess}}$ is dense in $Z$, for every $n\in\mathbb N$, we will say that $\Gamma$ is \textit{essential}. 
\end{definition}
Any $\Gamma$ induced by a generator (see equation (\ref{eq: dynamic refining sequence})) satisfies condition (\ref{eq:E1}). Also, an arbitrary $\Gamma$ can always be modified to satisfy it, but at the cost of increasing the cardinality of the covers. Condition (\ref{eq:E2}) is a type of regularity assumption on the covers. More precisely, if $\cl(v_n)=\cl(w_n)$ then $v_n=w_n$, because otherwise $\cl(v_n)\in \mathcal{Y}_n$ and $v_n^{\text{ess}}=\varnothing.$ Moreover, for every $v_n\in \mathcal{V}_n$ we have 
\begin{equation}\label{eq: equation of quotient map}
\Int (\pi_{\Gamma}(C_{[v_n]}))\neq \varnothing
\end{equation}
since $\varnothing \neq \pi_{\Gamma}^{-1}(v_n^{\text{ess}})\subset C_{[v_n]}$. Indeed, if $z\in v_n^{\text{ess}}$ and $\widetilde{p}\in \mathcal{P}_{\Gamma}$ is such that $\pi_{\Gamma}(\widetilde{p})=z$, then $z\in \cl(r(p_n))\cap v_n^{\text{ess}}$ and hence $r(p_n)=v_n$. 
Finally, essentiality really means that for every $m\geq n$ and $v_m\in \mathcal{V}_m$ there is a unique $v_n\in \mathcal{V}_n$ such that $v_m\subset v_n$. Consequently, for every cylinder set it holds that
\begin{equation}\label{eq:stronglyessentialcylinder}
C_{\mu_0\ldots \mu_{\ell}}=C_{[r(\mu_{\ell})]}.
\end{equation}

To see how well $\mathcal{P}_{\Gamma}$ approximates $Z$ consider the set 
\begin{equation}\label{eq:injectivityset}
\mathcal{I}_{\Gamma}=\{\widetilde{p}\in \mathcal{P}_{\Gamma}: \# \pi_{\Gamma}^{-1}(\widetilde{p})=1\}
\end{equation}
on which $\pi_{\Gamma}$ is injective. 
\begin{prop}\label{prop:injectivityset}
For a regular approximation graph $\Gamma$ we have $$\mathcal{I}_{\Gamma}=\pi_{\Gamma}^{-1}( \bigcap_{n\in \mathbb N} \bigcup  \mathcal{V}_n^{\text{ess}}).$$
\end{prop}
\begin{proof}
Let $z\in \bigcap_{n\in \mathbb N} \bigcup \mathcal{V}_n^{\text{ess}}$ and assume to the contrary that $\pi_{\Gamma}^{-1}(z) \not \subset \mathcal{I}_{\Gamma}$. Then there are two different $\widetilde{p},\widetilde{q} \in \mathcal{P}_{\Gamma}$ such that $\pi_{\Gamma}(\widetilde{p})=z=\pi_{\Gamma}(\widetilde{q}).$ Let $n_0$ be the first time when $(r(p_{n_0}),s(p_{n_0}))\neq (r(q_{n_0}),s(q_{n_0}))$, meaning $s(p_{n_0})=s(q_{n_0})$ and $r(p_{n_0})\neq r(q_{n_0})$. Then $z\in \cl({r(p_{n_0})})\cap \cl({r(q_{n_0})})\in \mathcal{Y}_{n_0}$ which is a contradiction. 

To prove that $\mathcal{I}_{\Gamma}\subset \pi_{\Gamma}^{-1}( \bigcap_{n\in \mathbb N}\bigcup  \mathcal{V}_n^{\text{ess}})$ let $\widetilde{p}\in \mathcal{I}_{\Gamma}$ and assume to the contrary that $\pi_{\Gamma}(\widetilde{p})\not \in \bigcap_{n\in \mathbb N}\bigcup  \mathcal{V}_n^{\text{ess}}.$ This means there is some $n_0\in \mathbb N$ such that $\pi_{\Gamma}(\widetilde{p})\in \bigcup \mathcal{Y}_{n_0}$. Consequently, for some $v_{n_0}\in \mathcal{V}_{n_0}$ that is different from $r(p_{n_0})$ we have $\pi_{\Gamma}(\widetilde{p})\in \cl({v_{n_0}})\cap \cl({r(p_{n_0})})$. Since $\pi_{\Gamma}$ is surjective, there is a path $\widetilde{q}\in \mathcal{P}_{\Gamma}$ such that $\pi_{\Gamma}(\widetilde{q})=\pi_{\Gamma}(\widetilde{p})$ and specifically due to condition (\ref{eq:E1}) we can arrange that $r(q_{n_0})=v_{n_0}$. This contradicts the fact that $\widetilde{p}\in \mathcal{I}_{\Gamma}$. 
\end{proof}
\begin{prop}\label{prop:stronginjectivityset}
Suppose $\Gamma$ is an essential approximation graph. Then $\mathcal{I}_{\Gamma}$ is a dense $G_{\delta}$-set.
\end{prop}
\begin{proof}
The fact that $\mathcal{I}_{\Gamma}$ is a $G_{\delta}$-subset of $\mathcal{P}_{\Gamma}$ follows from the continuity of $\pi_{\Gamma}$. To show that $\mathcal{I}_{\Gamma}$ is dense we note that every cylinder set satisfies $C_{\mu_0\ldots \mu_{\ell}}=C_{[r(\mu_{\ell})]}$, see equation (\ref{eq:stronglyessentialcylinder}). Then from (\ref{eq: equation of quotient map}) we get that $\text{int}(\pi_{\Gamma}(C_{\mu_0\ldots \mu_{\ell}}))\neq \varnothing$ and hence for every open set $U\subset \mathcal{P}_{\Gamma}$ we have $\text{int}(\pi_{\Gamma}(U))\neq \varnothing$. The Baire Category Theorem guarantees that $D=\bigcap_{n\in \mathbb N}\bigcup  \mathcal{V}_n^{\text{ess}}$ is dense in $Z$, hence if $U\subset \mathcal{P}_{\Gamma}$ is open then $\text{int}(\pi_{\Gamma}(U))\cap D\neq \varnothing$. Therefore, $U\cap \mathcal{I}_{\Gamma}\neq \varnothing$.
\end{proof}

\subsection{Dynamic approximation graphs}\label{sec:Dynamicgraphs}Suppose that $(Z,d)$ admits an expansive homeomorphism $\psi$ with expansivity constant $\varepsilon_Z>0$. We will show how to build an essential approximation graph for $Z$ whose infinite path space provides a symbolic representation for the $\psi$-orbits. 

The first step is to construct a closed cover $\mathcal{V}_1$ of diameter at most $\varepsilon_Z$, where $\bigcup \mathcal{V}_1^{\text{ess}}$ is dense in $Z$. One way to do this is by considering a cover by open sets $\{U_1,\ldots, U_{\ell}\}$ of diameter at most $\varepsilon_Z$, where each $U_k$ is not covered by the closure of the other sets. Then let $u_1=U_1$ and inductively define 
\begin{equation}
u_{i+1}=U_{i+1}\setminus \bigcup_{j=1}^{i}\cl(u_j),
\end{equation}
for $i=1,\ldots \ell-1$. It holds that each $u_k$ is open in $Z$ and if $k\neq l$ then $u_k\cap u_l=\varnothing$. Also, it is immediate that $\bigcup_{k=1}^{\ell}u_k$ is dense in $Z$. Let $\mathcal{V}_1=\{\cl(u_1),\ldots , \cl(u_{\ell})\}$ and we claim that $\bigcup \mathcal{V}_1^{ess} $ is dense in $Z$. Indeed, for the overlapping set $\mathcal{Y}_1$, see equation (\ref{eq:overlappingset}), it holds that $\bigcup \mathcal{Y}_1$ is a closed nowhere-dense subset of $Z$. Therefore, since $\bigcup \mathcal{V}_1^{\text{ess}} \supset \bigcup_{k=1}^{\ell}u_k \setminus \bigcup \mathcal{Y}_1$ the conclusion follows.

The second step is to define a refining sequence $(\mathcal{V}_n)_{n\geq 0}$ of $Z$ using the dynamics on $\mathcal{V}_1$. Let $\mathcal{V}_0=\{Z\}$ and for $n\in \mathbb N$ consider
\begin{equation}\label{eq:dynamicAG}
\mathcal{V}_n=\{v\in \bigvee_{i=1-n}^{n-1}\psi^{-i}(\mathcal{V}_1):\Int(v)\neq \varnothing\}.
\end{equation} 
Each $\mathcal{V}_n$ is a cover since $\bigcup \mathcal{V}_n$ is closed in $Z$ and contains $\bigcup \mathcal{V}_n^{\text{ess}}$ which in a similar fashion is proved to be dense in $Z$. Moreover, from \cite[Theorem 5.21]{Walters} it holds that $\lim\limits_{n\to \infty}\overline{\text{diam}}(\mathcal{V}_n)=0$. Consequently, the sequence $(\mathcal{V}_n)_{n\geq 0}$ is refining and induces an essential approximation graph $\Gamma$. 

Finally, a (left) \textit{shift map} $\sigma_{\Gamma}:\mathcal{P}_{\Gamma}\to \mathcal{P}_{\Gamma}$ which commutes with the quotient map $\pi_{\Gamma}:\mathcal{P}_{\Gamma}\to Z$ is defined as follows. Let $\widetilde{p} \in \mathcal{P}_{\Gamma}$ and each coordinate $p_n$ can be written uniquely in the form $(w_{-n}\ldots w_0 \ldots w_{n},w_{1-n}\ldots w_0 \ldots w_{n-1})$, where each $w_i\in \psi^{-i}(\mathcal{V}_1)$ and every word $w_{-k}\ldots w_0\ldots w_{k}$ corresponds to $\bigcap_{i=-k}^{k}w_i$. Recursively define the path $\widetilde{q} \in \mathcal{P}_{\Gamma}$ with $q_0=(\psi (w_1), Z)$ and for $n\geq 0$, 
\begin{equation}\label{eq:shiftmap}
q_{n+1} = (\psi (w_{-n})\ldots \psi (w_1) \ldots \psi (w_{n+2}),  r(q_n)),
\end{equation}
and let $\sigma_{\Gamma}(\widetilde{p})=\widetilde{q}.$ The map $\sigma_{\Gamma}$ is bijective with an inverse constructed by right shifting and continuous because for every $ \widetilde{p} $ it holds that $\sigma_{\Gamma}(C_{p_0\ldots p_{n}})\subset C_{q_0\ldots q_{n}}$. Since $\mathcal{P}_{\Gamma}$ is compact and Hausdorff the map $\sigma_{\Gamma}$ is a homeomorphism. In this setting, the quotient map $\pi_{\Gamma}$ becomes a \textit{factor map} $\pi_{\Gamma}: (\mathcal{P}_{\Gamma}, \sigma_{\Gamma})\to (Z,\psi)$ since for any $\widetilde{p} \in \mathcal{P}_{\Gamma} $
\begin{align*}
\pi_{\Gamma}(\sigma_{\Gamma}(\widetilde{p}))&= \bigcap_{n\geq 0}\cl({r (q_n)})\\
& = \bigcap_{n\geq 1} \psi (\cl({r(p_{n-1}))})\\
& = \psi (\bigcap_{n\geq 1}  \cl({r(p_{n-1}))})\\
& = \psi (\pi_{\Gamma} (\widetilde{p})).
\end{align*}
\begin{cor}\label{cor:quotientofCantor}
Every expansive dynamical system $(Z,\psi)$ is the quotient of some $(\mathcal{P}_{\Gamma},\sigma_{\Gamma})$, where $\mathcal{P}_{\Gamma}$ is a compact zero-dimensional space constructed as above. If $Z$ does not have isolated points then $\mathcal{P}_{\Gamma}$ is a Cantor space. Moreover, the factor map $\pi_{\Gamma}:(\mathcal{P}_{\Gamma},\sigma_{\Gamma})\to (Z,\psi)$ is injective on a dense $G_{\delta}$-set.
\end{cor}
Factor maps have been studied extensively by Adler \cite{Adler}, particularly factor maps which are uniformly bounded to one. We will come back to this fact later. For now we can argue that our construction captures some of the dynamical behaviour of $(Z,\psi)$. 
\begin{prop}\label{thm: irreducibility of shift}
Let $\Gamma$ be an essential approximation graph. Then
\begin{enumerate}[(1)]
\item if $(Z,\psi)$ is irreducible so is $(\mathcal{P}_{\Gamma},\sigma_{\Gamma})$;
\item if $(Z,\psi)$ is mixing so is $(\mathcal{P}_{\Gamma},\sigma_{\Gamma})$.
\end{enumerate}
\end{prop}

\begin{proof}
We will only prove (1) since (2) is similar. Let $U,W\subset \mathcal{P}_{\Gamma}$ be non-empty and open. We need to find $N\in \mathbb N$ such that $\sigma_{\Gamma}^N(U)\cap W \neq \varnothing.$ We can find small enough cylinder sets $C_{\mu}\subset U, C_{\mu'}\subset W$ and since $\Gamma$ is essential (\ref{eq: equation of quotient map}) implies that $\varnothing \neq \pi_{\Gamma}^{-1}(r(\mu_{\ell})^{\text{ess}})\subset C_{\mu}$ and $\varnothing \neq \pi_{\Gamma}^{-1}(r(\mu_{\ell'})^{\text{ess}})\subset C_{\mu'}$. Since $(Z,\psi)$ is irreducible there is some $N\in \mathbb N$ such that $\psi^N(r(\mu_{\ell})^{\text{ess}})\cap r(\mu_{\ell'})^{\text{ess}}\neq \varnothing$. Then

\begin{align*}
\varnothing &\neq \pi_{\Gamma}^{-1}( \psi^N(r(\mu_{\ell})^{\text{ess}})\cap r(\mu_{\ell'})^{\text{ess}})\\
&= \pi_{\Gamma}^{-1}( \psi^N(r(\mu_{\ell})^{\text{ess}}))\cap \pi_{\Gamma}^{-1}(r(\mu_{\ell'})^{\text{ess}})\\
&= \sigma_{\Gamma}^N(\pi_{\Gamma}^{-1} (r(\mu_{\ell})^{\text{ess}})) \cap \pi_{\Gamma}^{-1}(r(\mu_{\ell'})^{\text{ess}})\\
&\subset \sigma_{\Gamma}^N(C_{\mu})\cap C_{\mu'}\\
&\subset \sigma_{\Gamma}^N(U)\cap W. \qedhere
\end{align*}
\end{proof}

An essential approximation graph provides a combinatorial model for $(Z,\psi)$ that allows us to study its topological structure. However, it is not the right tool to study the geometric or metric properties of $(Z,\psi)$. The reason is that it consists of closed covers with nowhere-dense overlaps and hence all the Lebesgue numbers of the covers are zero. So the idea is that, given such a graph $\Gamma'$, we try to enlarge it to a graph $\Gamma$ consisting of open covers so that $\Gamma'$ is isomorphic to a spanning subgraph of $\Gamma$; that is, a subgraph which contains every vertex of $\Gamma$. 

\begin{definition}\label{def:quasiessentialgraph}
An approximation graph $\Gamma=(\mathcal{V},\mathcal{E})$ of open covers of $Z$ will be called \textit{metrically-essential} if there is an essential approximation graph $\Gamma'=(\mathcal{V}',\mathcal{E}')$ of $Z$ with bijections $$F_n:\mathcal{V}'_n\to \mathcal{V}_n,$$ given by $v'_n\mapsto v_n$ if $v'_n\subset v_n$, such that $\coprod_{n\geq 0}F_n:\Gamma' \to \Gamma$ is a graph homomorphism and the induced map $F:\mathcal{P}_{\Gamma'}\to \mathcal{P}_{\Gamma}$ satisfies $\pi_{\Gamma}\circ F=\pi_{\Gamma'}$. 
\end{definition}
\begin{remark}
A metrically-essential graph is not necessarily regular. Also, an essential graph is metrically essential only if $Z$ is zero-dimensional.
\end{remark}

\section{Geometric approximations}\label{sec:Geometric approximations and embeddings}
\subsection{Geometric approximation graphs} We introduce another class of approximation graphs which now encode geometric properties of their base spaces. Recall the notation (\hyperlink{Q1}{$\text{Q}1$}-\hyperlink{Q5}{$\text{Q}5$}) introduced in Subsection \ref{sec:TDS} and that $\text{C}(v_n)$ denotes the set of descendants of $v_n\in \mathcal{V}_n$. We first define geometric approximation graphs and then discuss their properties.
\begin{definition}\label{def:geometricgraph}
An approximation graph $\Gamma=(\mathcal{V},\mathcal{E})$ of open covers of $Z$ will be called \textit{geometric} if there exist constants $\lambda,\Lambda
>1$ with $\lambda \leq \Lambda$, constants $\eta,\theta>0$ with $\eta \leq \theta$ and $C_{\Gamma},N_{\Gamma}\in \mathbb N$ so that, for all $n\in \mathbb N$,
\begin{enumerate}
\item $\overline{\text{diam}}(\mathcal{V}_n)\leq \lambda^{-n+1}\theta;$
\item $\text{Leb}(\mathcal{V}_n)\geq \Lambda^{-n+1}\eta;$
\item $\# \text{C}(v_n) \leq C_{\Gamma},$ for every $v_n\in \mathcal{V}_n$;
\item $\# \text{N}_{\mathcal{V}_n}(v_n) \leq N_{\Gamma},$ for every $v_n\in \mathcal{V}_n$.
\end{enumerate}
If $\lambda=\Lambda$, the approximation graph $\Gamma$ will be called \textit{homogeneous}. Moreover, a geometric approximation graph which is also metrically-essential (see Definition \ref{def:quasiessentialgraph}) will be called \textit{geometrically-essential}.
\end{definition}

Geometric approximation graphs are related to dimension theory and, in particular, with the following concepts.

From Theorem V.1 of \cite{Nagata} we see that $Z$ has finite covering dimension at most $m$ if and only if there is a sequence of arbitrarily small open covers $(\mathcal{U}_n)_{n\in \mathbb N}$ (not necessarily refining) with multiplicities $\text{m}(\mathcal{U}_n)\leq m+1$, for all $n\in \mathbb N$. Due to condition (4), geometric approximation graphs are related to finite covering dimension. Note though that condition (4) is stronger than having uniformly bounded multiplicities. From the sequence $(\mathcal{U}_n)_{n\in \mathbb N}$ of Theorem V.1 it is possible to obtain a refining sequence by considering a subsequence. However, this increases the cardinality of the covers and the rate of decay of the Lebesgue numbers. Condition (1) can be satisfied though. 

An important example of the above situation is the case where $Z$ admits an expansive homeomorphism. In \cite{Mane}  Ma$\tilde{\text{n}}${\'e} proves that $Z$ has covering dimension at most $(\#\mathcal{V}_1)^{2}-1$ by constructing arbitrarily small open covers with multiplicity at most $(\#\mathcal{V}_1)^{2}.$ It turns out that the missing ingredient, which leads to a geometric approximation graph, is a Markov partition which we introduce in Section \ref{sec:Markovpartitions}. An expansive dynamical system has a Markov partition if and only if it has a local product structure \cite{Fried}.

Assume now that $Z$ admits an expansive homeomorphism $\psi:Z\to Z$ which is also $\Lambda$-bi-Lipschitz. The approximation graph $\Gamma$ induced by a generator in (\ref{eq: dynamic refining sequence}) satisfies conditions (2) and (3). More precisely, $\Leb(\mathcal{V}_n)\geq \Lambda^{-n+1}\Leb(\mathcal{V}_1)$ and $\# \text{C}(v_n)\leq (\#\mathcal{V}_1)^{2}$. Nonetheless, the other conditions are not necessarily satisfied and the upper bounds on $\# \text{C}(v_n)$ may not be sharp.

The concept that is closest to our approximation graphs comes from the Nagata-Assouad dimension \cite{LS}. The main characterisation is that $Z$ will have Nagata-Assouad dimension $\dim_{\text{NA}}Z\leq n$ if and only if there is some $c>0$ such that for every $r>0$ there is a cover $\mathcal{U}_r$ of $Z$ with $\mul(\mathcal{U}_r)\leq n+1$, $\overline{\diam}(\mathcal{U}_r)\leq cr$ and $\Leb(\mathcal{U}_r)\geq r$, see \cite[Proposition 2.2]{BDHM}. Therefore, a homogeneous approximation graph provides a discrete version of the above characterisation but in slightly stronger form that allows us to prove the following proposition. First, we should mention that Assouad dimension is an upper bound for the Nagata-Assouad dimension \cite{DR}.

\begin{prop}\label{prop:geometricembedding}
A compact metric space $(Z,d)$ with a homogeneous approximation graph has finite Assouad dimension.
\end{prop}
\begin{proof}
Let $0<a\leq b$ and consider a finite subset $Y\subset Z$ with $a\leq d(y,y')\leq b$ if $y,y' \in Y$ and $y\neq y'$. We claim that there are $s,C\geq 0$ independent of $a,b$ and $Y$ so that $\# Y \leq C(b/a)^s$. Let $\Gamma=(\mathcal{V},\mathcal{E})$ be a homogeneous approximation graph with constants as in Definition (\ref{def:geometricgraph}). First we prove the claim in the case where $b\leq \eta$. Define
\begin{align*}
n_0&=\text{min}\{n\in \mathbb N: \lambda^{-n+1}\theta <a\}\\
m_0&=\text{max}\{n\in \mathbb N: b\leq \lambda^{-n+1}\eta\}
\end{align*}
and an easy computation shows that $n_0=1+\max\{1,\ceil{\log_{\lambda}(\theta/a)}\}$ and $m_0=1+\floor{\log_{\lambda}(\eta/b)}$. Note that since $b\leq \eta$ we cannot have $\theta <a$. We have $m_0< n_0$ since $\lambda^{-n_0+1}\theta < a \leq b \leq \lambda^{-m_0+1}\eta$, meaning that $n_0-m_0>\log_{\lambda}(\theta/\eta)\geq 0.$ From our definition of $n_0$ there cannot be two elements of $Y$ in one element of $\mathcal{V}_{n_0}$ since $\overline{\text{diam}}(\mathcal{V}_{n_0})\leq \lambda^{-n_0+1}\theta <a$. Also since $\text{diam}(Y)\leq b \leq \lambda^{-m_0+1}\eta \leq \text{Leb}(\mathcal{V}_{m_0})$, there is some element $v_{m_0}\in \mathcal{V}_{m_0}$ that contains $Y$. Actually, any element of $\mathcal{V}_{m_0}$ that intersects $Y$ is a neighbour of $v_{m_0}$ and from the definition of the geometric approximation graph there cannot be more than $N_{\Gamma}$ neighbours. The descendants in $\mathcal{V}_{n_0}$ of the neighbours of $v_{m_0}$ are the only ones that cover the whole $Y$, because if $v_{n_0}\in \mathcal{V}_{n_0}$ contains some $y\in Y$, then its ancestor in $\mathcal{V}_{m_0}$ should also contain $y$ and hence be a neighbour of $v_{m_0}$. From condition (3) in the Definition \ref{def:geometricgraph} we conclude that $\# Y\leq N_{\Gamma}C_{\Gamma}^{n_0-m_0}$. If $\theta =a$ then $n_0=2$ and $m_0=1$. If $\theta >a$ then we have 
\begin{align*}
n_0-m_0&= \ceil{\log_{\lambda}(\theta/a)}-\floor{\log_{\lambda}(\eta/b)}\\
&\leq 2+\log_{\lambda}(\theta/a)-\log_{\lambda}(\eta/b)\\
&=2+\log_{\lambda}(\theta /\eta)+\log_{\lambda}(b/a),
\end{align*}
and for simplicity let $c=2+\log_{\lambda}(\theta /\eta)\geq 0$. 

We now have $\# Y\leq N_{\Gamma}C_{\Gamma}^{c}C_{\Gamma}^{\log_{\lambda}(b/a)}=N_{\Gamma}C_{\Gamma}^{c}(b/a)^s$, for $s=1/\log_{C_{\Gamma}}(\lambda)$. In the case where $b>\eta$, let $K_{\eta}$ be the cardinality of a minimal cover $\{B(x_i,\eta/2)\}_{i\in I}$ of $Z$. Then apply the above construction for each $Y\cap B(x_i,\eta/2)$ and in general we get that 
\begin{equation}\label{eq:homogeneousdegree}
\# Y\leq C(b/a)^s,
\end{equation}
for $C=K_{\eta}N_{\Gamma}C_{\Gamma}^{c}$ and $s=1/\log_{C_{\Gamma}}(\lambda)$.
\end{proof}

\section{Smale spaces}\label{sec:Smalespaces}
\subsection{Preliminaries}\label{sec:SmalePrelim}Roughly speaking, a Smale space ($X,\varphi$) is a dynamical system consisting of a homeomorphism $\varphi$ acting on a compact metric space $X$ that is locally hyperbolic under $\varphi$, in the sense that every $x\in X$ has a small neighbourhood that can be decomposed into the product of contracting and expanding sets. 
\begin{definition}[{\cite[Section~7.1]{Ruelle}}]
Let $(X,d)$ be a compact metric space and let $\varphi:X\to X$ be a homeomorphism. The dynamical system ($X,\varphi$) is a \textit{Smale space} if there are constants $\varepsilon_X>0$, $\lambda_X >1$ and a locally defined bi-continuous map, called the \textit{bracket map} 
$$\{ (x,y)\in X\times X: d(x,y)\leq \varepsilon_X\} \mapsto [x,y]\in X$$
that satisfies the axioms:
\begin{align*}
\tag{B1} [x,x]&=x,\\
\tag{B2} [x,[y,z]]&=[x,z],\\
\tag{B3} [[x,y],z]&=[x,z],\\
\tag{B4} \varphi([x,y])&=[\varphi(x),\varphi(y)];
\end{align*}
for any $x,y,z \in X$, whenever both sides are defined. For $x\in X$ and $0<\varepsilon \leq \varepsilon_X$ let 
\begin{align}
X^s(x,\varepsilon)&=\{y\in X: d(x,y)<\varepsilon, [x,y]=y\}\\
X^u(x,\varepsilon)&=\{y\in X: d(x,y)<\varepsilon, [y,x]=y\}
\end{align}
be the \textit{local stable} and \textit{unstable sets}. On these sets we have the \textit{contraction axioms}: 
\begin{align*}
\tag{C1} d(\varphi(y),\varphi(z))&\leq \lambda_X^{-1}d(y,z), \enspace \text{for any} \enspace y,z \in X^s(x,\varepsilon),\\
\tag{C2} d(\varphi^{-1}(y),\varphi^{-1}(z))&\leq \lambda_X^{-1}d(y,z),\enspace \text{for any} \enspace y,z \in X^u(x,\varepsilon).
\end{align*}
\end{definition}

Quite often, we will consider the Lipschitz constants $\Lip(\varphi)$ and $\Lip(\varphi^{-1})$ which are both greater than $\lambda_X>1$ and, in general, are allowed to be infinite. In particular, we will focus on 
\begin{equation}
\ell_X=\min\{\Lip(\varphi),\Lip(\varphi^{-1})\} \enspace \text{and}\enspace \Lambda_X=\max\{\Lip(\varphi),\Lip(\varphi^{-1})\}.
\end{equation}

We say that the bracket map defines a \textit{local product structure} on $X$ because, for any $x\in X$ and $0<\ep \leq \ep_X/2$, the bracket map 
\begin{equation}\label{eq:bracketmap}
[\cdot, \cdot]:X^u(x,\varepsilon)\times X^s(x,\varepsilon)\to X
\end{equation}
is a homeomorphism onto its image (see \cite[Proposition~2.1.8]{Putnam_Book}). Also, due to the uniform continuity, there is a constant $0<\ep_X'\leq \ep_X/2$ such that, if $d(x,y)\leq \ep_X'$, then both $d(x,[x,y]),d(y,[x,y])< \ep_X/2$ and hence 
\begin{equation}\label{eq:uniquebracket}
X^s(x,\ep_X/2)\cap X^u(y,\ep_X/2)=[x,y].
\end{equation}

Equation (\ref{eq:uniquebracket}) together with a bracket independent description of the local stable and unstable sets (see \cite[Subsection~4.1]{Putnam_Lec}) imply that the bracket map is unique on $X$ (but of course depends on $\ep_X$ and $\lambda_X$). Moreover, an important property of Smale spaces is that they are \textit{expansive} \cite[Proposition~2.1.9]{Putnam_Book}. Expansiveness immediately implies finiteness of covering dimension \cite{Mane}, topological entropy \cite[Theorem~3.2]{Walters} and upper box-counting dimension (for certain metrics) \cite{Fathi}. Moreover, Smale spaces (without any assumption on the metric) have finite Hausdorff dimension, see Ruelle's Exercise $1$ in \cite[Chapter 7]{Ruelle}. According to Smale's program \cite{Smale}, the interesting dynamics of a Smale space lie in the non-wandering set which is the closure of its periodic points \cite[Corollary~3.7]{Bowen} and which can be studied through its irreducible and mixing components \cite[Section~7.4]{Ruelle}.
\begin{thm}[Smale's Decomposition Theorem] \label{thm: Smale decomposition}
Assume that the Smale space $(X,\varphi)$ is non-wandering. Then $X$ can be decomposed into a finite disjoint union of clopen, $\varphi$-invariant, irreducible sets $X_0,\ldots , X_{N-1}$. Each of these sets can be decomposed into a finite disjoint union of clopen sets $X_{i0},\ldots ,X_{iN_i}$ that are cyclically permuted by $\varphi$, and where $\varphi^{N_i+1}|_{X_{ij}}$ is mixing, for every $0\leq j\leq N_i$.
\end{thm}

As a corollary of Theorem \ref{thm: Smale decomposition} and Proposition \ref{thm:no isolated points} we note the following.

\begin{cor}\label{cor:noisolatedpoints}
Assume that the Smale space $(X,\varphi)$ is non-wandering and all its irreducible parts are infinite. Then $X$ has no isolated points.
\end{cor}

Every irreducible Smale space admits a distinguished measure, referred to as the Bowen measure (see \cite{Bowen4}, \cite[Section 20]{Katok_Has}), which we denote by $\mu_{\text{B}}$. Roughly speaking, it is exhibited as a limit distribution of periodic orbits and is the unique $\varphi$-invariant probability measure that maximises the topological entropy $\ent(\varphi)$. Moreover, it is compatible with the bracket map \cite{RS}.

Smale spaces are ubiquitous in the theory of expansive dynamical systems, see \cite[Lemma~2]{Fried}. Smale spaces were defined by Ruelle \cite{Ruelle} to give a topological description of the non-wandering sets of differentiable dynamical systems satisfying \textit{Axiom A} \cite{Smale}. Zero dimensional Smale spaces are exactly the \textit{subshifts of finite type} (SFT), see \cite[Section 3]{BS}, \cite[Theorem 2.2.8]{Putnam_Book}. More recently, Wieler characterised the Smale spaces with totally disconnected stable or unstable sets, now called \textit{Wieler solenoids} \cite{Wieler}. Among these, the SFT  play an important role in coding the orbits of a Smale space, see Section \ref{sec:Markovpartitions}. Since any SFT is topologically conjugate to a \textit{topological Markov chain} \cite[Prop. 3.2.1]{BS}, let us introduce the latter. 

\subsection{Topological Markov chains}\label{sec:TMC} We equip $\{1,\ldots , N\}$ with the discrete topology and $\{1,\ldots , N\}^{\mathbb Z}$ with the product topology that makes it a compact Hausdorff space. Let $M$ be a square matrix indexed by $N$, with $0$ and $1$ entries, and consider the closed subspace of \textit{allowable sequences} 
\begin{equation}\label{eq:SFT}
\Sigma_M=\{x=(x_i)_{i\in \mathbb Z}\in \{1,\ldots , N\}^{\mathbb Z}: M_{x_i,x_{i+1}}=1\}.
\end{equation}

The \textit{cylinder sets} 
\begin{equation}\label{eq:cylindersets}
C_{\mu_{-n},\ldots ,\mu_m}=\{x\in \Sigma_M: x_i=\mu_i, \enspace \text{for} \enspace -n\leq i \leq m\}
\end{equation}
form a basis of clopen sets for the product topology on $\Sigma_M$. The number of fixed digits will be called the \textit{rank of the cylinder}, which for the cylinder set in (\ref{eq:cylindersets}) is equal to $m+n+1$. If $m=n$ the cylinders will be called \textit{symmetric}. The metric 
\begin{equation}\label{eq:SFTmetric}
d(x,y)=\text{inf}\{2^{-n}:n\geq 0, \enspace x_i=y_i \enspace \text{for} \enspace |i|<n\}
\end{equation}
induces the product topology on $\Sigma_M$ and it is straightforward to prove that it is actually an ultrametric. Also, with this metric, every symmetric cylinder of rank $2n-1$ is a ball of radius $2^{-n}$ around each of its points.

The shift map $\sigma_M:\Sigma_M\to \Sigma_M$ given by $\sigma_M(x)_{i}=x_{i+1}$ for any $i\in \mathbb Z$, is a homeomorphism and ($\Sigma_M,\sigma_M$) is called a \textit{topological Markov chain}. It admits a Smale space structure with the bracket map defined by
\begin{equation}\label{eq:SFTbracket}
([x,y])_n=
\begin{cases}
y_n, & \text{for} \enspace n\leq 0\\
x_n, & \text{for} \enspace n\geq 1
\end{cases}
\end{equation}
for any $x,y \in \Sigma_M$ such that $d(x,y)\leq 2^{-1}$. The expansivity constant is $\varepsilon_{\Sigma_M}=2^{-1}$ and the contraction constant is $\lambda_{\Sigma_M}=2$.
\begin{remark} \label{thm: Markov property remark}
A word $x_1\ldots x_{n-1}$ can only be concatenated on the right by a letter $x_n\in \{1,\ldots ,N\}$ if the value of $M_{x_{n-1},x_n}$ is $1$. This is called the \textit{Markov property} and it appears in the general setting of Smale spaces.
\end{remark}
We would like to estimate the number of symmetric cylinder sets of rank $2n-1$ for the shift space $\Sigma_M$. Denote this number by $N_M(2n-1)$. The estimation is obtained by replicating the computation for the topological entropy $\ent(\sigma_M)$, see \cite[p. 121]{Katok_Has}. However, we sketch the proof for completeness. 
\begin{lemma} \label{thm: upper bound on cardinality in SFT}
Let $(\Sigma_M,\sigma_M)$ be a topological Markov chain with matrix $M$ where no row contains only $0$'s. There exist constants $C,c>0$ so that for every $\varepsilon \in (0,1)$ there is some $n_0\in \mathbb N$ such that, for every $n\geq n_0$, we have $$ce^{2(\ent(\sigma_M)-\ep) n}< N_M(2n-1)<Ce^{2(\ent(\sigma_M)+\ep) n}.$$
\end{lemma} 
\begin{proof}
Let $M=(m_{ij})_{i,j=1}^{N}$. Then $$N_M(2n-1)=\sum_{i,j=1}^{N}m_{ij}^{2n-2}$$ and there is some $C'>0$ such that $N_M(2n-1)<C'\|M^{2n-2}\|.$ Also, since the numbers $m_{ij}^{2n-2}$ are non-negative, there is some $c'>0$ such that $N_M(2n-1)>c'\|M^{2n-2}\|.$ Following the equalities (3.2.3) in \cite[p. 121]{Katok_Has} one obtains that $$\lim_{n\to \infty} \frac{1}{2n-2}\log \|M^{2n-2}\| = \ent(\sigma_M)$$ and the result follows.
\end{proof}
The Bowen measure on a mixing topological Markov chain $(\Sigma_M,\sigma_M)$ has a very nice description, and is known as the Parry measure $\mu_{\text{P}}$ \cite[Chapter 4]{Katok_Has}. Since $(\Sigma_M,\sigma_M)$ is mixing, $M$ is a primitive matrix, meaning that there is some power of $M$ with only positive entries. For primitive matrices the Perron-Frobenius Theorem \cite[Theorem 1.9.11]{Katok_Has} yields a unique (up to a scalar) eigenvector of strictly positive coordinates whose eigenvalue $\lambda_{\max}>0$ is greater than the absolute value of all the other eigenvalues. Let $u,v$ be the Perron-Frobenius eigenvectors for $M$ and $M^{T}$, respectively, which are normalised so that 
\begin{equation}
\sum_{i=1}^{N} v_iu_i=1.
\end{equation}
The distribution $p=(p_1,\ldots , p_N)$, with $p_i=v_iu_i$, induces the Parry measure.
\begin{lemma}\label{lem:Parrymeasure}
Let $(\Sigma_M,\sigma_M)$ be a mixing topological Markov chain. Then there is a constant $D>0$ so that every non-empty symmetric cylinder set has Parry measure $$D^{-1}\lambda^{-2n}_{\max}\leq \mu_{\Par}(C^{-n,\ldots , n}_{\mu_{-n},\ldots ,\mu_n})\leq D\lambda^{-2n}_{\max}.$$
\end{lemma}
\begin{proof}
Following the computations on page 176 of \cite{Katok_Has}, if we look at $(\Sigma_M,\sigma_M)$ as an edge shift then $$\mu_{\Par}(C^{-n,\ldots , n}_{\mu_{-n},\ldots ,\mu_n})=v_iu_j\lambda^{-2n+2}_{\max}$$ where $i=(\mu_{-n},\mu_{-n+1})$ and $j=(\mu_{n-1},\mu_n)$ are edges. Since the coordinates of $u,v$ are all positive the result follows.
\end{proof}
\begin{cor}\label{cor:AhlforsSFT}
Let $(\Sigma_M,\sigma_M)$ be a mixing topological Markov chain equipped with the metric defined in (\ref{eq:SFTmetric}). Its Parry measure $\mu_{\Par}$ is Ahlfors $s_0$-regular and therefore, $$\dim_H\Sigma_M = \dim_B \Sigma_M= \dim_A \Sigma_M=s_0$$ where $s_0=2\ent(\sigma_M)/\log(2).$

\end{cor}
\subsection{Wieler solenoids}\label{sec:Wielersol} In \cite{Wieler} Wieler characterised Smale spaces with totally disconnected local stable sets as stationary inverse limits of (eventually) Ruelle expanding dynamical systems \cite[p.138]{Ruelle}. Wieler's results generalise the earlier work of Williams for expanding attractors \cite{Williams} to the topological setting.

Let $(Y,d)$ be a compact metric space and $g:Y\to Y$ be a continuous and surjective map. We say that $(Y,g)$ satisfies \textit{Wieler's axioms} if there are constants $\beta >0 , K\in \mathbb N$ and $\gamma \in (0,1)$ such that for every $x,y,z\in Y$ with $d(x,y)\leq \beta$ and $0<\varepsilon \leq \beta$ we have 
\begin{align*}
\tag{W1} d(g^K(x),g^K(y))&\leq \gamma^Kd(g^{2K}(x),g^{2K}(y))\\
\tag{W2} g^K(B(g^K(z),\varepsilon))&\subset g^{2K}(B(z,\gamma\varepsilon)).
\end{align*}
The \textit{Wieler solenoid} associated with such $(Y,g)$ is the stationary inverse limit $(\widehat{Y},\widehat{g})$ where
\begin{equation}
\widehat{Y}=\{(y_0,y_1,y_2,\ldots ):y_n\in Y,\, y_n=g(y_{n+1}),\, n \geq 0\}
\end{equation}
and 
\begin{equation}
\widehat{g}(y_0,y_1,y_2,\ldots )=(g(y_0),y_0,y_1,\ldots ),
\end{equation}
with metric $\widehat{d}$ given by 
\begin{equation}\label{eq:Wielermetric}
\widehat{d}(x,y)=\sum_{k=0}^{K-1}\gamma^{-k}d'(\widehat{g}^{-k}(x),\widehat{g}^{-k}(y)),
\end{equation}
where $d'(x,y)=\sup \{\gamma^n d(x_n,y_n):n\geq 0\}$. Wieler's main result is the following.

\begin{thm}[{\cite[Theorem A and B]{Wieler}}]
\phantom{newline}
\begin{enumerate}[(A)]
\item Any Wieler solenoid $(\widehat{Y},\widehat{g})$ is a Smale space with totally disconnected stable sets. Moreover, $(\widehat{Y},\widehat{g})$ is irreducible if and only if $(Y,g)$ is non-wandering and has a dense forward orbit.
\item Any irreducible Smale space with totally disconnected stable sets is topologically conjugate to a Wieler solenoid.
\end{enumerate}
\end{thm}
\begin{remark}\label{rem:Wielerremark}
Following \cite{Wieler}, the contraction constant $\lambda_{\widehat{Y}}$ of $(\widehat{Y},\widehat{g})$ is equal to $\gamma^{-1}$. Also, from the definition of the metric, we see that the map $\widehat{g}^{-1}$ is $\lambda_{\widehat{Y}}$-Lipschitz, while $\widehat{g}$ need not be Lipschitz. However, if $g$ is Lipschitz then $\widehat{g}$ will be too. Moreover, one has $$\widehat{d}(\widehat{g}(x),\widehat{g}(y))=\lambda_{\widehat{Y}}^{-1}\widehat{d}(x,y),$$\enlargethispage{\baselineskip}whenever $y\in \widehat{Y}^s(x,\varepsilon_{\widehat{Y}})$.
\end{remark}

\subsection{Dyadic solenoid}\label{sec:2solenoid} We equip the unit circle $\mathbb T=\{e^{2\pi i\theta}:\theta \in [0,1)\}$ with the (normalised) arc length distance 
\begin{equation}
d(e^{2\pi i\theta_1},e^{2\pi i\theta_2})=\min \{ |\theta_1-\theta_2|,1-|\theta_1-\theta_2|\},
\end{equation}
and consider the \textit{doubling map} $g:\mathbb T\to \mathbb T$ given by $g(z)=z^2$. The map $g$ is expanding because, if $d(z,w)\leq 1/4$ then $d(g(z),g(w))=2d(z,w).$ 

As in the case of Wieler solenoids, we construct the stationary inverse limit $(\widehat{\mathbb T},\widehat{g})$ and equip it with the metric 
\begin{equation}\label{eq:2solenoidmetric}
\widehat{d}(x,y)=\sum_{n=0}^{\infty}2^{-n}d(x_n,y_n),
\end{equation}
which induces the product topology on $\widehat{\mathbb T}$. Note that the metric (\ref{eq:2solenoidmetric}) is different from the metric (\ref{eq:Wielermetric}) defined for Wieler solenoids. This inverse limit is called the \textit{dyadic solenoid} and is topologically conjugate to the Smale-Williams solenoid \cite{BS}. In fact, $(\widehat{\mathbb T},\widehat{g})$ is a Smale space with expansivity and contraction constants $\varepsilon_{\widehat{\mathbb T}}=1/4$ and $\lambda_{\widehat{\mathbb T}}=2$. If $\widehat{d}(x,y)\leq 1/4$ then, in particular, $d(x_0,y_0)\leq 1/4$ and hence there is a unique $|t|\leq 1/4$ such that $x_0=y_0e^{2\pi it}.$ The bracket map for such $x,y$ is defined as the sequence 
\begin{equation}\label{eq:2solenoidbracket}
[x,y]=(y_0e^{2\pi it},y_1e^{\pi it}, y_2e^{\pi i t/2},\ldots )
\end{equation}
which is in $\widehat{\mathbb T}$. The (largest) local stable set around $x\in \widehat{\mathbb T}$ consists of sequences whose $0$-th coordinate is $x_0$. Therefore, local stable sets are Cantor sets and global stable sets are totally disconnected. On the other hand, the global unstable sets are one-dimensional. Moreover, one can check that the doubling map $g$ is $2$-Lipschitz and hence $\widehat{g}$ is $5/2$-Lipschitz. Moreover, the map $\widehat{g}$ is the $2^{-1}$-multiple of an isometry on local stable sets. Finally, the inverse $\widehat{g}^{-1}$ is $2$-Lipschitz and the $2^{-1}$-multiple of an isometry on local unstable sets. 

\subsection{Metrics and smoothing of Smale spaces}\label{sec:LipschitzSmalespaces}
There is an abundance of Smale spaces ($X,\varphi$) where $\varphi$ is a \textit{bi-Lipschitz} homeomorphism, meaning $\Lambda_X<\infty$, where $\Lambda_X=\max\{\Lip(\varphi),\Lip(\varphi^{-1})\}$. The obvious examples are the SFT or the non-wandering sets of Axiom A diffeomorphisms. This is often true for Wieler solenoids, see Remark \ref{rem:Wielerremark}. For such Smale spaces the following holds.
\begin{lemma}[{\cite[p.234-235]{Sakai}}]\label{lem:Lipschitzbracket}
Let $(X,\varphi)$ be a Smale space with $\Lambda_X<\infty$. Then there exists $A_X>0$ with $$A_X\leq\frac{\Lambda_X\lambda_X}{\lambda_X^2-1}$$ such that for any $x,y \in X$ with $d(x,y)\leq \varepsilon_X'$ one has $d(x,[x,y])\leq A_Xd(x,y)$ and $d(y,[x,y])\leq A_Xd(x,y)$.
\end{lemma}
In \cite{Fried}, Fried showed that any expansive dynamical system $(Z,\psi)$ admits a compatible hyperbolic metric $d_{\Fr}$ for which $\psi$ becomes bi-Lipschitz. This means that any Smale space $(X,d,\varphi)$ admits a compatible metric $d_{\Fr}$ for which $\varphi$ is bi-Lipschitz. Now, the new dynamical system $(X,d_{\Fr},\varphi)$ is still a Smale space. Indeed, the existence of a bracket map satisfying axioms (B1)-(B4) is not affected by changing to another compatible metric. Moreover, $d_{\Fr}$ is hyperbolic, meaning that the contraction axioms (C1) and (C2) will still be satisfied (possibly with a different constant). Therefore we obtain the following.

\begin{thm}[\cite{Fried}]\label{lem:Friedlemma}
Any Smale space is topologically conjugate to a Smale space with bi-Lipschitz dynamics.
\end{thm}

We note that Lemma \ref{lem:Lipschitzbracket} and Theorem \ref{lem:Friedlemma} solved a question posed by Ruelle in \cite[Appendix B.7]{Ruelle}. Later, Fathi \cite[Theorem 5.1]{Fathi} showed that Fried's metric $d_{\Fr}$, defined on an expansive $(Z,\psi)$, satisfies an additional property which can be used to obtain upper bounds for $\overline{\dim}_B(Z,d_{\Fr})$. In our case, \textit{Fathi's property} is the following.

\begin{thm}[\cite{Fathi}]\label{lem:Fathilemma}
For any Smale space $(X,d_{\Fr},\varphi)$ equipped with Fried's metric there exist constants $k>1,\xi>0$ such that $$\max\{d_{\Fr}(\varphi(x),\varphi(y)),d_{\Fr}(\varphi^{-1}(x),\varphi^{-1}(y))\}\geq \min \{kd_{\Fr}(x,y),\xi\}$$ for every $x,y\in X$.
\end{thm}

Fathi's property implies that the contraction axioms (C1) and (C2) of the Smale space $(X,d_{\Fr},\varphi)$, hold more globally than just on local stable or unstable sets. This complicated statement will be extremely useful in the sequel. At this point we should say that Fried's metric $d_{\Fr}$ is not concretely related to the original metric $d$ of $(X,d,\varphi)$, even if $(X,d,\varphi)$ has bi-Lipschitz dynamics. However, in the latter case it may be still possible to obtain Fathi's property for $(X,d,\varphi)$ without changing the metric $d$. 

More precisely, suppose that $\Lambda_X<\infty$ for the Smale space $(X,d,\varphi)$. Also, assume that the contraction constant satisfies $\lambda_X>2A_X$, where $A_X>0$ is obtained from Lemma \ref{lem:Lipschitzbracket}. Let $0<\widetilde{\varepsilon_X}\leq \varepsilon_X'$ be small enough so that if $d(x,y)\leq \widetilde{\varepsilon_X}$ then $d(\varphi^i(x),\varphi^i(y))\leq \varepsilon_X'$, for $i\in \{-1,1\}.$ For $x,y\in X$ with $d(x,y)\leq \widetilde{\varepsilon_X}$ one has
\begin{align*}
d(x,y)&\leq d(x,[x,y])+d(y,[x,y])\\
&\leq \frac{1}{\lambda_X}(d(\varphi^{-1}(x),\varphi^{-1}([x,y]))+d(\varphi(y),\varphi([x,y])))\\
&=\frac{1}{\lambda_X}(d(\varphi^{-1}(x),[\varphi^{-1}(x),\varphi^{-1}(y)])+d(\varphi(y),[\varphi(x),\varphi(y)]))\\
&\leq \frac{A_X}{\lambda_X}(d(\varphi^{-1}(x),\varphi^{-1}(y))+d(\varphi(x),\varphi(y)))\\
&\leq \frac{2A_X}{\lambda_X}\max\{d(\varphi(x),\varphi(y)),d(\varphi^{-1}(x),\varphi^{-1}(y))\}.
\end{align*}
Consequently, we can choose $k=\lambda_X/(2A_X)$. Moreover, for every $x,y\in X$ we have $\max\{d(\varphi(x),\varphi(y)),d(\varphi^{-1}(x),\varphi^{-1}(y))\} \geq \Lambda_X^{-1}d(x,y)$ and hence we can choose $\xi=\Lambda_X^{-1}\widetilde{\varepsilon_X}.$ 

\begin{remark}
The computations above give a lower bound for $A_X$. In particular, given a Smale space $(X,d,\varphi)$ with $\Lambda_X<\infty$ (no other restriction) one has $$ A_X\geq \frac{\lambda_X}{2\Lambda_X}.$$
\end{remark}

Finding the best such $k$ is crucial for obtaining good estimates for the Hausdorff and box-counting dimensions, see Section \ref{sec:SemiconformalSmalespaces}. The computations above indicate that for estimating $k$ one should first try to estimate the constant $A_X$ of Lemma \ref{lem:Lipschitzbracket}. In general though, for $\lambda_X>2A_X$ to be true it suffices to restrict to Smale spaces with $\lambda_X\in (1+\sqrt{2},\infty)$ and $\Lambda_X\in [\lambda_X, (\lambda_X^2-1)/2).$

Returning to the discussion of Fried's metric, Artigue \cite{Artigue} recently constructed compatible metrics on expansive dynamical systems for which the systems exhibit self-similarity. We describe his construction in the context of Smale spaces. 

Let $(X,d,\varphi)$ be a Smale space. One can construct the Smale space $(X,d_{\Fr},\varphi)$ for which Fathi's property in Theorem \ref{lem:Fathilemma} is satisfied for some $k>1,\xi>0$. Then Artigue defines the metric $d_{\Ar}$ given by 
\begin{equation}
d_{\Ar}(x,y)=\max\limits_{n\in \mathbb Z} \frac{d_{\Fr}(\varphi^n(x),\varphi^n(y))}{k^{|n|}}
\end{equation}
and proves that 
\begin{equation}\label{eq:metricinequality}
d_{\Fr}\leq d_{\Ar}\leq c (d_{\Fr})^{\gamma},
\end{equation}
where $\gamma=\log_{\Lambda_{X,\Fr}}(k)\in (0,1)$ ($\Lambda_{X,\Fr}$ is the maximum of the two Lipschitz constants for the metric $d_{\Fr}$) and $c>0$. Most importantly, the new contraction constant and the new Lipschitz constants of $\varphi,\varphi^{-1}$ for the Smale space $(X,d_{\Ar},\varphi)$ are equal to $k>1$. 

\begin{definition}\label{def:selfsimilarSmalespace}
A Smale space $(X,\varphi)$ is called \textit{self-similar} if $\lambda_X=\Lambda_X$, where $\lambda_X>1$ is its contraction constant and $\Lambda_X=\max\{\Lip(\varphi),\Lip(\varphi^{-1})\}$. 
\end{definition}

\begin{remark}\label{rem:Artigueargument}
One can observe that Artigue's construction works for any metric $d$ that satisfies Fathi's property and hence the metric inequalities (\ref{eq:metricinequality}) hold with $d$ in the place of $d_{\Fr}$. For instance, we proved that this is true for Smale spaces with $\Lambda_X<\infty$ and $\lambda_X>2A_X$. This fact will be used in Corollary \ref{cor:Hausdorffdimension}.
\end{remark}

As we can see from the next lemma, there is a plethora of self-similar Smale spaces.

\begin{lemma}\label{lem:Artiguelemma}
Any Smale space is topologically conjugate to a self-similar Smale space.
\end{lemma}

Obvious examples are the SFT. Self-similarity means that the dynamics is very tight (see \cite[Remark 2.22]{Artigue} and compare with the case of SFT, Wieler solenoids and Lemma \ref{lem:Lipschitzbracket}). Note that in a self-similar Smale space, $\varphi$ acts on the local stable and unstable sets as the $\lambda_X^{-1}$-multiple of an isometry.

\section{Markov partitions}\label{sec:Markovpartitions}
Roughly speaking, for an irreducible Smale space $(X,\varphi)$ a Markov partition is a partition of the space $X$ into closed subsets that have a local product structure and which overlap only on their boundaries. Such partitions yield a dynamically defined refining sequence, as in Subsection \ref{sec:Dynamicgraphs}, which induces an essential approximation graph $\Pi$ of $X$. From Proposition \ref{prop:stronginjectivityset}, the factor map $\pi_{\Pi}:(\mathcal{P}_{\Pi},\sigma_{\Pi})\to (X,\varphi)$ is injective on a dense $G_{\delta}$-set and since $(X,\varphi)$ is irreducible, $(\mathcal{P}_{\Pi},\sigma_{\Pi})$ will be too. This would be the end of the story if the so-called \textit{shadowing property} did not hold for Smale spaces \cite[Prop. 3.6]{Bowen}. This property sets up such a partition of $X$ which in addition satisfies the crucial \textit{Markov property}, see Remark \ref{thm: Markov property remark}. In this case, $(\mathcal{P}_{\Pi},\sigma_{\Pi})$ is a topological Markov chain and this is actually where the story begins! 

Except for being essential and providing a combinatorial model for $(X,\varphi)$, the graph $\Pi$ satisfies conditions (1), (3) and (4) in the Definition \ref{def:geometricgraph} of geometric approximation graphs. However, it does not satisfy condition (2) since the overlaps in each cover of the refining sequence are nowhere-dense and hence the Lebesgue numbers are zero. Therefore, our plan is to recursively $\delta$-fatten the closed covers in $\Pi$ for some carefully chosen $\delta >0$. In this way we will keep all the nice properties of $\Pi$ and obtain a geometrically-essential approximation graph $\Pi^{\delta}$. The Markov property makes this possible. The dynamical systems that admit Markov partitions have been characterised in \cite{Fried} and are known as \textit{finitely presented dynamical systems}. These contain sofic systems and pseudo-Anosov homeomorphisms. It is quite possible that the method we described can be generalised to this larger class of dynamical systems.

\subsection{Basics on Markov partitions} We deal with  the classical Markov partitions and not with topological partitions, as defined in \cite{Adler, Putnam_Lec}. This means that we also consider the boundaries of the partitions. However, the ideas are of the same nature. Also, we focus on irreducible (infinite) Smale spaces, but most of the statements still hold for non-wandering Smale spaces due to Smale's Decomposition Theorem \ref{thm: Smale decomposition}.

First, let $(X,\varphi)$ be a Smale space and recall the definition of $0<\varepsilon_X'\leq \varepsilon_X/2$ from (\ref{eq:uniquebracket}), and similarly let $0<\varepsilon_X''<\varepsilon_X'/12$ be so small that whenever $d(x,y)\leq \varepsilon_X''$, we have
\begin{equation}\label{eq:double_epsilon}
d(\varphi^i(x),\varphi^i(y))\leq \varepsilon_X'/2, \enspace d([x,y],y)\leq \varepsilon_X'/4, \enspace d([x,y],x)\leq \varepsilon_X'/4,
\end{equation}
for every $|i|\leq 2$.
\begin{definition}\label{def:rectangles}
A non-empty subset $R\subset X$ is called a \textit{rectangle} if $\text{diam}(R)\leq \varepsilon_X'$ and $[x,y]\in R$, for any $x,y \in R$.
\end{definition}

If $R$ is a rectangle and $x\in R$, let 
\begin{equation}
X^s(x,R)= X^s(x,2\varepsilon_X')\cap R \enspace \text{and} \enspace X^u(x,R)= X^u(x,2\varepsilon_X')\cap R.
\end{equation}
From the bracket axioms and the definition of rectangles it holds that
\begin{equation}\label{eq:localproduct2}
R=[X^u(x,R),X^s(x,R)].
\end{equation}
In fact, the local product structure on $X$ (see (\ref{eq:bracketmap})) implies that 
\begin{equation}\label{eq:productstructuretopology}
\Int(R)=[\Int(X^u(x,R)),\Int(X^s(x,R))] \enspace \text{and} \enspace \cl(R)=[\cl(X^u(x,R)),\cl(X^s(x,R))],
\end{equation}
where the interiors and the closures are taken in $X^u(x,2\varepsilon_X')$ and $X^s(x,2\varepsilon_X')$. Define the \textit{stable boundary} to be 
\begin{equation}
\partial^sR=\{x\in R:  X^s(x,R)\cap \text{int}(R)=\varnothing\}
\end{equation}
and the \textit{unstable boundary} to be
\begin{equation}
\partial^uR=\{x\in R:  X^u(x,R)\cap \text{int}(R)=\varnothing\}.
\end{equation}

It follows that 
\begin{equation}\label{eq: boundaries of a rectangle}
\partial^sR= [\partial X^u(x,R), X^s(x,R)]\enspace \text{and} \enspace \partial^uR= [X^u(x,R), \partial X^s(x,R)],
\end{equation}
where $\partial X^u(x,R)$ and $\partial X^s(x,R)$ are the boundaries of $X^u(x,R)$ and $X^s(x,R)$ as subsets of $X^u(x,2\varepsilon_X')$ and $X^s(x,2\varepsilon_X')$, respectively. Also \cite[Lemma 3.11]{Bowen} states that 
\begin{equation}
\partial R= \partial^sR\cup \partial^u R.
\end{equation}
\begin{lemma}\label{thm:closure of rectangles}
For any rectangles $R,R'\subset X$ with $R\cap R'\neq \varnothing$ and diameter small enough, the set $$[R,R']= \{[r,r']:r\in R,\enspace r'\in R'\}$$ 
\begin{enumerate}[(1)]
\item is a rectangle;
\item $\Int([R,R'])=[\Int(R),\Int(R')]$ and $\cl([R,R'])=[\cl(R),\cl(R')]$.
\end{enumerate}
\end{lemma}
\begin{proof}
Part (1) follows from the bracket axioms (B2) and (B3). For part (2) let $x\in R\cap R'$ and then $[R,R']=[X^u(x,R),X^s(x,R')]$ using bracket axioms (B2), (B3) and equation (\ref{eq:localproduct2}). The local product structure implies that $$\Int([R,R'])=[\Int(X^u(x,R)),\Int(X^s(x,R'))]=[\Int(R),\Int(R')]$$ using axioms (B2), (B3) and equation (\ref{eq:productstructuretopology}). The proof for the closures is similar.
\end{proof}
A rectangle $R$ is called \textit{proper} if it is closed and $R=\cl(\Int(R))$.
\begin{definition}[{\cite[Section 3]{Bowen2}}]\label{def:Definition of Markov partition}
A \textit{Markov partition} is a finite covering $\mathcal{R}_1=\{R_1,\ldots , R_{\ell}\}$ of $X$ by non-empty proper rectangles such that 
\begin{enumerate}[(1)]
\item $\text{int}(R_i)\cap \text{int}(R_j)=\varnothing$ for $i\neq j$;
\item $\varphi (X^u(x,R_i))\supset X^u(\varphi(x), R_j) $ and 
\item $\varphi (X^s(x,R_i))\subset X^s(\varphi(x), R_j) $ when $x\in \text{int}(R_i)\cap\varphi^{-1}(\text{int}(R_j))$.
\end{enumerate}
\end{definition}
The following is Bowen's seminal theorem \cite[Theorem 12]{Bowen2}.
\begin{thm}[Bowen's Theorem]
If the Smale space $(X,\varphi)$ is irreducible, then it has Markov partitions of arbitrarily small diameter.
\end{thm}
Consequently, if the Smale space $(X,\varphi)$ is irreducible, then it is a factor of a topological Markov chain. More precisely, let $\mathcal{R}_1$ be a Markov partition for $(X,\varphi)$ and $M$ be the transition matrix given by 
$$M_{i,j}=
\begin{cases}
1, &\text{if}\enspace \text{int}(R_i) \cap \varphi^{-1}(\text{int}(R_j))\neq \varnothing; \\
0, & \text{otherwise}
\end{cases}
$$
for $1\leq i,j\leq \ell$. The following theorem is due to Bowen, see  \cite[Theorem 3.18]{Bowen}, \cite[Prop. 30]{Bowen2} and \cite[Prop. 10]{Bowen3}. Recall the Bowen measure $\mu_{\Bow}$ from Subsection \ref{sec:SmalePrelim}.
\begin{thm}\label{thm:Bowen factor map}
Define the map $\pi_M:(\Sigma_M,\sigma_M)\to (X,\varphi)$ by $$(x_i)_{i\in \mathbb Z}\mapsto \bigcap_{i\in \mathbb Z}\varphi^{-i}(R_{x_i}).$$ Then $\pi_M$ is 
\begin{enumerate}[(1)]
\item a factor map;
\item injective on the residual set $X\setminus \bigcup_{i\in \mathbb Z}\varphi^{i}(\partial\mathcal{R}_1)$, where $\partial \mathcal{R}_1= \bigcup_{R\in \mathcal{R}_1} \partial R$;
\item a metric isomorphism between $(\Sigma_M,\sigma_M,\mu_{\Bow})$ and $(X,\varphi, \mu_{\Bow})$.
\item For every $x\in X$, the pre-image $\pi_M^{-1}(x)$ has at most $(\#\mathcal{R}_1)^2$ elements.
\item $(\Sigma_M,\sigma_M)$ is irreducible (or mixing if $(X,\varphi)$ is mixing). 
\end{enumerate}
\end{thm}

\subsection{Approximation graph of Markov partitions}
Let $(X,\varphi)$ be an irreducible Smale space and $\mathcal{R}_1$ be a Markov partition with $\overline{\text{diam}}(\mathcal{R}_1)\leq \varepsilon_X''$. We consider the sequence $(\mathcal{R}_n)_{n\geq 0}$  given by $\mathcal{R}_0=\{X\}$ and 
\begin{equation}\label{eq:Markov refining sequence}
\mathcal{R}_n=\{R\in \bigvee_{i=1-n}^{n-1}\varphi^{-i}(\mathcal{R}_1): \text{int}(R)\neq \varnothing\},
\end{equation}
for $n\in \mathbb N$. In order to show that each $\mathcal{R}_n$ covers $X$, let $\mathcal{R}_n^{o}=\{\Int (R): R\in \mathcal{R}_n\}$ and observe that $\bigcup \mathcal{R}_1^{o}$ is dense in $X$. Moreover, since
\begin{equation}
\mathcal{R}_n^{o}=\bigvee_{i=1-n}^{n-1}\varphi^{-i}(\mathcal{R}_1^{o})
\end{equation}
we obtain that $\bigcup \mathcal{R}_n^{o}$ is dense in $X$. Therefore, $\mathcal{R}_n$ covers $X$. From \cite[Theorem 5.21]{Walters} we obtain that $(\mathcal{R}_n)_{n\geq 0}$ is a refining sequence and the corresponding graph $\Pi=(\mathcal{R},\mathcal{A})$ has no sources and no sinks. Therefore, $\Pi$ is an essential approximation graph (see Lemma \ref{lem:Markovapproximationgraph1}) and $(X,\varphi)$ is a factor of $(\mathcal{P}_{\Pi},\sigma_{\Pi})$. Also, $(\mathcal{P}_{\Pi},\sigma_{\Pi})$ is irreducible since $(X,\varphi)$ is irreducible. This means that $\mathcal{P}_{\Pi}$ is a Cantor space, see Subsection \ref{sec:Essentialgraphs}.

Now we present the main result of this section. We note that part (1) will be proved in a slightly more general setting in Lemma \ref{lem:sequenceofdiameters}, hence we defer the proof until then. Recall the notation (\hyperlink{Q1}{$\text{Q}1$}-\hyperlink{Q5}{$\text{Q}5$}) introduced in Subsection \ref{sec:TDS} and we have the following.

\begin{prop}\label{prop:mainresultMarkovpartitions}
For every $n\in \mathbb N$, the cover $\mathcal{R}_n$ in (\ref{eq:Markov refining sequence}) is a Markov partition. Moreover, there exists constants $\theta \in(0, \varepsilon_X]$ and $C,c>0$ so that
\begin{enumerate}[(1)]
\item $\overline{\diam}(\mathcal{R}_n)\leq \lambda_X^{-n+1}\theta;$
\item $\mul(\mathcal{R}_n)\leq \#(\mathcal{R}_1)^2;$
\item the number of neighbouring rectangles is uniformly bounded, meaning $$\sup\limits_n \max\limits_{R\in \mathcal{R}_n}\#\num_{\mathcal{R}_n}(R)<\infty;$$
\item for every $\varepsilon \in (0,1)$, there is some $n_0\in \mathbb N$ such that, for $n\geq n_0$, we have $$ce^{2(\ent(\varphi)-\varepsilon) n}< \#\mathcal{R}_n <Ce^{2(\ent(\varphi)+\varepsilon) n};$$
\item the approximation graph $\Pi=(\mathcal{R},\mathcal{A})$, associated to the refining sequence of Markov partitions $(\mathcal{R}_n)_{n\geq 0}$, is essential, and $$(\mathcal{P}_{\Pi},\sigma_{\Pi})=(\Sigma_M,\sigma_M),$$ where $M$ is the transition matrix of $\mathcal{R}_1$.
\end{enumerate} 
\end{prop}

The proof of Proposition \ref{prop:mainresultMarkovpartitions} will be achieved by establishing the following lemmas.

\begin{lemma}\label{thm:Markov property rectangles}
For the Markov partition $\mathcal{R}_1$ with $\overline{\text{diam}}(\mathcal{R}_1)\leq \varepsilon_X''$, the following hold.
\begin{enumerate}[(1)]
\item $\varphi(R),\varphi^{-1}(R)$ are rectangles, for any $R\in \mathcal{R}_1$.
\item For any $R_i,R_j\in \mathcal{R}_1$ such that $\Int(R_i)\cap \varphi^{-1}(\Int(R_j))\neq \varnothing$ we have 
\begin{equation}
[\varphi^{-1}(R_j), R_i]=R_i\cap \varphi^{-1}(R_j),
\end{equation}
and hence $R_i\cap \varphi^{-1}(R_j)$ is a proper rectangle. We will refer to this condition as the Markov property.
\item If $R_i,R_j,R_k\in \mathcal{R}_1$ with $\varphi(\Int(R_i))\cap \Int(R_j)\cap \varphi^{-1}(\Int(R_k))\neq \varnothing$ then we have 
\begin{equation}
[\varphi^{-1}(R_k),\varphi (R_i)]=\varphi (R_i) \cap R_j \cap \varphi^{-1}(R_k).
\end{equation}
Consequently, $\varphi (R_i) \cap R_j \cap \varphi^{-1}(R_k)$ is a proper rectangle that is equal to $\varphi (R_i) \cap \varphi^{-1}(R_k).$ 
\end{enumerate}
\end{lemma}
\begin{proof}
Part (1) follows from the fact that $\varepsilon_X''$ is so small that $\text{diam}(\varphi^{\pm 1}(\mathcal{R}_1))\leq \varepsilon_X'/2$ and the $\varphi$-invariance of the bracket map. For part (2) it is clear that $ R_i\cap \varphi^{-1}(R_j)\subset [\varphi^{-1}(R_j), R_i]$. 

For the reverse inclusion let $x\in \text{int}(R_i)\cap \varphi^{-1}(\text{int}(R_j))$ and then $[\varphi^{-1}(R_j),R_i]=[X^u(x,\varphi^{-1}(R_j)),X^s(x,R_i)].$ First we claim that 
\begin{equation}\label{eq:localunstablepart}
X^u(x,\varphi^{-1}(R_j))=\varphi^{-1}(X^u(\varphi(x),R_j)).
\end{equation}
Indeed, since $\varphi(x)\in R_j$ write $R_j=[X^u(\varphi(x),R_j),X^s(\varphi(x),R_j)]$ and hence $$\varphi^{-1}(R_j)=[\varphi^{-1}(X^u(\varphi(x),R_j),\varphi^{-1}(X^s(\varphi(x),R_j))],$$ where $\varphi^{-1}(X^u(\varphi(x),R_j)\subset X^u(x,\varepsilon_X)$ and $\varphi^{-1}(X^s(\varphi(x),R_j))\subset X^s(x,\varepsilon_X)$. Also, $$\varphi^{-1}(R_j)=[X^u(x,\varphi^{-1}(R_j)),X^s(x,\varphi^{-1}(R_j))].$$\enlargethispage{\baselineskip} The claim follows since the bracket map is bijective around $x$. 

Using conditions (2) and (3) in the definition of the Markov partition, it is easy to show that $$[\varphi^{-1}(X^u(\varphi(x),R_j)),X^s(x,R_i)]\subset R_i\cap \varphi^{-1}(R_j).$$ We have $X^s(x,R_i)\subset R_i$ and $$\varphi^{-1}(X^u(\varphi(x),R_j))\subset X^u(x,R_i)\subset R_i.$$ Therefore, $$[\varphi^{-1}(X^u(\varphi(x),R_j)),X^s(x,R_i)]\subset [R_i,R_i]=R_i,$$ since $R_i$ is a rectangle. In the same way, $$\varphi(X^s(x,R_i))\subset X^s(\varphi(x),R_j)\subset R_j,$$ which gives $X^s(x,R_i)\subset \varphi^{-1}(R_j)$. Also $\varphi^{-1}(X^u(\varphi(x),R_j))\subset \varphi^{-1}(R_j)$, and one gets $$[\varphi^{-1}(X^u(\varphi(x),R_j)),X^s(x,R_i)]\subset [\varphi^{-1}(R_j),\varphi^{-1}(R_j)]=\varphi^{-1}(R_j).$$ This shows that $[\varphi^{-1}(R_j), R_i]\subset R_i\cap \varphi^{-1}(R_j)$, and hence $[\varphi^{-1}(R_j), R_i]= R_i\cap \varphi^{-1}(R_j).$ Consequently, $R_i\cap \varphi^{-1}(R_j)$ is proper since using Lemma \ref{thm:closure of rectangles} we obtain $$\text{int}([\varphi^{-1}(R_j), R_i])=[\varphi^{-1}(\text{int}(R_j)),\text{int}(R_i)]$$ and $\cl(\text{int}([\varphi^{-1}(R_j), R_i]))=[\cl(\varphi^{-1}(\text{int}(R_j))),\cl(\text{int}(R_i))]=[\varphi^{-1}(R_j),R_i]$. 

Part (3) is proved in a similar fashion. First we observe that $$R_i\cap \varphi^{-1}(R_j)\cap \varphi^{-2}(R_k) = [\varphi^{-1}(R_j)\cap \varphi^{-2}(R_k), R_i].$$ The inclusion $$R_i\cap \varphi^{-1}(R_j)\cap \varphi^{-2}(R_k) \subset [\varphi^{-1}(R_j)\cap \varphi^{-2}(R_k), R_i]$$ should be clear. For the reverse inclusion let $x\in \varphi^{-1}(R_j)\cap \varphi^{-2}(R_k)$ and $y\in R_i$. In particular, $x\in \varphi^{-1}(R_j)$ and using part (2) we obtain that $[x,y]\in [\varphi^{-1}(R_j), R_i]= R_i\cap \varphi^{-1}(R_j).$ Also, $\varphi (x) \in R_j \cap \varphi^{-1}(R_k)$ and $\varphi ([x,y]) \in R_j$. Therefore, \begin{align*}
[x,y]=\varphi^{-1}[\varphi (x),\varphi [x,y]]&\in \varphi^{-1} [R_j \cap \varphi^{-1}(R_k), R_j]\\
& = \varphi^{-1}[[\varphi^{-1}(R_k),R_j],R_j]\\
& = \varphi^{-1}(R_j \cap \varphi^{-1}(R_k))\\
& \subset \varphi^{-2}(R_k)
\end{align*}
and consequently $[x,y]\in R_i\cap \varphi^{-1}(R_j)\cap \varphi^{-2}(R_k).$ 

We now  have that 
\begin{align*}
\varphi (R_i) \cap R_j \cap \varphi^{-1}(R_k)&= \varphi (R_i\cap \varphi^{-1}(R_j)\cap \varphi^{-2}(R_k))\\
&= \varphi [\varphi^{-1}(R_j)\cap \varphi^{-2}(R_k), R_i]\\
&= \varphi [\varphi^{-1}[\varphi^{-1}(R_k),R_j],R_i]\\
&=\varphi [ [ \varphi^{-2}(R_k),\varphi^{-1}(R_j)],R_i]\\
&= \varphi [\varphi^{-2}(R_k), R_i]\\
&= [\varphi^{-1}(R_k),\varphi (R_i)],
\end{align*}
where the fourth equality makes sense because $\overline{\text{diam}}(\mathcal{R}_1)\leq \varepsilon_X''$. The properness of the rectangle follows as in part (2) and $\varphi (R_i) \cap R_j \cap \varphi^{-1}(R_k)= \varphi (R_i) \cap \varphi^{-1}(R_k)$ because $\varphi (R_i) \cap \varphi^{-1}(R_k)\subset [\varphi^{-1}(R_k),\varphi (R_i)]$.
\end{proof}
\begin{remark}\label{rem: stable / unstable boundary of refinement}
Let $R=\varphi (R_i) \cap \varphi^{-1}(R_k)=[\varphi^{-1}(R_k),\varphi (R_i)]$ be a rectangle as in part (2) of Proposition \ref{thm:Markov property rectangles}. Then it is not hard to see that 
\begin{equation}\label{eq:Markovboundaries}
\partial^s R\subset \partial^s(\varphi^{-1}(R_k))\enspace \text{and} \enspace \partial^u R\subset \partial^u(\varphi (R_i)).
\end{equation}

Similarly as in equation (\ref{eq:localunstablepart}) the local product structure of the space implies that $X^u(x, R)= X^u(x, \varphi^{-1}(R_k))$, for $x\in R$. Consequently, it holds that
\begin{align*}
\partial^s R&=\partial^s [X^u(x, R),  X^s(x, R)]\\
&=[ \partial X^u(x, R), X^s(x, R)]\\
&\subset [ \partial X^u(x, \varphi^{-1}(R_k)), X^s(x, \varphi^{-1}(R_k))]\\
&=\partial^s(\varphi^{-1}(R_k)).
\end{align*}
The unstable case is proved similarly.
\end{remark}

The next lemma will allow us to apply Lemma \ref{thm:Markov property rectangles} inductively.

\begin{lemma}\label{lem:inductive step in Markov partitions}
For every $n\in \mathbb N$ one has $$\mathcal{R}_{n+1}=\{R\in \varphi (\mathcal{R}_n)\vee \mathcal{R}_n \vee \varphi^{-1}(\mathcal{R}_n): \Int(R)\neq \varnothing\}.$$
\end{lemma}
\begin{proof}
To prove that $\mathcal{R}_{n+1}\subset \{R\in \varphi (\mathcal{R}_n)\vee \mathcal{R}_n \vee \varphi^{-1}(\mathcal{R}_n): \Int(R)\neq \varnothing\}$ we just need the fact that the interior of a finite intersection is the intersection of the interiors. For the reverse inclusion it suffices to observe that if $R,S\in \mathcal{R}_1$ with $\Int(R)\cap \Int(S)\neq \varnothing$ then $R=S$.
\end{proof}
\begin{lemma}\label{lem:inductiveMarkovpartition}
For every $n\in \mathbb N$ the cover $\mathcal{R}_n$ is a Markov partition.
\end{lemma}
\begin{proof}
The statement holds for $n=1$ since $\mathcal{R}_1$ is a Markov partition by definition. Assume that $\mathcal{R}_n$ is a Markov partition and the claim is that $\mathcal{R}_{n+1}$ is one too. Let $R\in \mathcal{R}_{n+1}$ and we certainly have $\text{diam}(R)\leq \varepsilon_X''.$ Using Lemma \ref{lem:inductive step in Markov partitions} and parts (1) and (3) of Proposition \ref{thm:Markov property rectangles} we deduce that $R$ is a proper rectangle. The fact that the interiors of any two rectangles in $\mathcal{R}_{n+1}$ are mutually disjoint follows from $\varphi$ being a homeomorphism. 

For conditions (2) and (3) of Definition \ref{def:Definition of Markov partition}, again by Lemma \ref{lem:inductive step in Markov partitions}, it is enough to consider an arbitrary $$x\in \Int (\varphi(R_{\ell})\cap R_i\cap \varphi^{-1}(R_j)\cap \varphi^{-1}(\varphi(R_i)\cap R_j \cap \varphi^{-1}(R_k)))$$ for some $R_{\ell},R_i,R_j,R_k\in \mathcal{R}_n$. 

\noindent Then one has 
\begin{align*}
X^u(\varphi(x),\varphi(R_i)\cap R_j\cap \varphi^{-1}(R_k)) &\subset X^u(\varphi(x),R_j)\cap \varphi(R_i)\\
&\subset \varphi (X^u(x,R_i))\cap \varphi(R_i) \cap R_j \\
&\subset \varphi (\varphi(X^u(\varphi^{-1}(x),R_{\ell})))\cap \varphi(R_i) \cap R_j\\
&=\varphi(X^u(x,\varphi(R_{\ell})\cap R_i\cap \varphi^{-1}(R_j))),
\end{align*}
where the second inclusion holds because $x\in \text{int}(R_i)\cap \varphi^{-1}(\text{int}(R_j))$ and the third because $x\in \varphi(\text{int}(R_{\ell}))\cap \text{int}(R_i)$. For the last equality use the same argument as in (\ref{eq:localunstablepart}). 
\enlargethispage{\baselineskip}
For the stable part, we have 
\begin{align*}
\varphi(X^s(x,\varphi(R_{\ell})\cap R_i\cap \varphi^{-1}(R_j))) &\subset \varphi(X^s(x,R_i))\cap R_j\\
& \subset X^s(\varphi(x),R_j)\cap \varphi(R_i)\cap R_j\\
&\subset \varphi^{-1}(X^s(\varphi^2(x),R_k))\cap \varphi(R_i)\cap R_j\\
&= X^s(\varphi(x),\varphi(R_i)\cap R_j \cap \varphi^{-1}(R_k)),
\end{align*}
where the second inclusion holds because $x\in \text{int}(R_i)\cap \varphi^{-1}(\text{int}(R_j))$, the third because $\varphi(x)\in \text{int}(R_j) \cap \varphi^{-1}(\text{int}(R_k))$, and finally the last equality is true for the same reason as in the unstable case.
\end{proof}
\begin{lemma}\label{lem:Markovapproximationgraph1}
The approximation graph $\Pi=(\mathcal{R},\mathcal{A})$ is essential. Also $(\mathcal{P}_{\Pi},\sigma_{\Pi})=(\Sigma_M,\sigma_M)$, where $M$ is the transition matrix of $\mathcal{R}_1$.
\end{lemma}
\begin{proof}
The graph $\Pi$ is essential because the rectangles in $\mathcal{R}_{n}$ are proper with mutually disjoint interiors. Finally, $(\mathcal{P}_{\Pi},\sigma_{\Pi})=(\Sigma_M,\sigma_M)$ following the definition of $\mathcal{R}_n$ and the cylinder sets of $\Sigma_M$.
\end{proof}

We now estimate the growth rate of $\# \mathcal{R}_n$ with respect to $n$. Recall that an irreducible Smale space has the Bowen measure which maximises the topological entropy. 
\begin{lemma}[{\cite[Theorem 33]{Bowen2}}]\label{thm:same entropy}
For a topological Markov chain $(\Sigma_M,\sigma_M)$ that is induced by a Markov partition of an irreducible Smale space $(X,\varphi)$ one has $\ent(\varphi)=\ent(\sigma_M)$. 
\end{lemma}
\begin{lemma}\label{lem:upper bound for number of rectangles}
There exist constants $C,c>0$ such that for every $\varepsilon \in (0,1)$ there is some $n_0\in \mathbb N$ so that, for $n\geq n_0$, we have $$ce^{2(\ent(\varphi)-\varepsilon) n}< \#\mathcal{R}_n <Ce^{2(\ent(\varphi)+\varepsilon) n}.$$
\end{lemma}
\begin{proof}
For any rectangle $R\in \mathcal{R}_n$ its essential part is given by $R^{\text{ess}}=\Int(R)$. Now (\ref{eq: equation of quotient map}) together with the essentiality of $\Pi$ imply that $R^{\text{ess}}\subset \pi_{\Pi}(C_{R})\subset R$, where $C_R$ is the cylinder set of $R$ in $\mathcal{P}_{\Pi}$. Since $R$ is proper and $\pi_{\Pi}$ is a closed map, we get $\pi_{\Pi}(C_{R})=R$. Now, let $(\Sigma_M,\sigma_M)$ be the topological Markov chain that is induced by $\mathcal{R}_1$ and recall the definition of the factor map $\pi_M:(\Sigma_M,\sigma_M)\to (X,\varphi)$ where $\pi_M=\pi_{\Pi}$ from Lemma \ref{lem:Markovapproximationgraph1}. Then every $R\in \mathcal{R}_n$ is the $\pi_M$-image of a cylinder set of rank $2n-1$. Actually, $\#\mathcal{R}_n=N_M(2n-1)$; the number of non-empty cylinder sets of $\Sigma_M$ of rank $2n-1$. From Lemma \ref{thm: upper bound on cardinality in SFT}, there exist constants $C,c > 0$ such that for every $\varepsilon \in (0,1)$ there is some $n_0\in \mathbb N$ so that for $n\geq n_0$ we have $$ce^{2(\ent(\sigma_M)-\varepsilon) n}< \#\mathcal{R}_n <Ce^{2(\ent(\sigma_M)+\varepsilon) n}.$$ Using Lemma \ref{thm:same entropy} we obtain the result.
\end{proof}
\begin{lemma}\label{lem:bounded multiplicity}
For every $n\geq 0$ we have $\mul(\mathcal{R}_n)\leq (\#\mathcal{R}_1)^2$.
\end{lemma}
\begin{proof}
It follows from the same argument used in the proof of Lemma \ref{lem:upper bound for number of rectangles} and part (4) of Theorem \ref{thm:Bowen factor map}.
\end{proof}

For proving condition (3) of Proposition \ref{prop:mainresultMarkovpartitions} we will use the \textit{diamond trick} found in \cite{Adler}. More precisely, the map $\pi_M$ would have a \textit{diamond} if there were two sequences $x=(x_i)_{i\in \mathbb Z}$ and $y=(y_i)_{i\in \mathbb Z}$ with $\pi_M(x)=\pi_M(y)$ for which there exist indices $k<l<m$ such that $x_k=y_k,x_l\neq y_l$ and $x_m=y_m$. However, $\pi_M$ does not have a diamond since $\overline{\text{diam}}(\mathcal{R}_1)<\varepsilon_X/2$, see \cite[Lemma 6.9]{Adler}. So our strategy for proving condition (3) is by contradiction. We will assume that it does not hold and then we will be able to create a diamond for $\pi_M$.

We say that $R,S\in \mathcal{R}_n$ have a \textit{common neighbour} if they intersect some $T\in \mathcal{R}_n$. Also, recall that for any $m\geq n$ and $R\in \mathcal{R}_m$ there is a unique $R'\in \mathcal{R}_n$  such that $R\subset R'$. We call $R'$ an \textit{ancestor} of $R$. The next lemma will follow from the Markov property. 
\begin{lemma}\label{lem:normalityconstant}
There is some $N\in \mathbb N$ such that for every $R,S\in \mathcal{R}_{n+N}$ with a common neighbour, their ancestors in $\mathcal{R}_n$ intersect.
\end{lemma}

Before proceeding to the proof let us make an observation. The above statement is equivalent to the one saying that there is some $N\in \mathbb N$ such that for every two disjoint $R,S\in \mathcal{R}_n$, all their descendants in $\mathcal{R}_{n+N}$ do not have a common neighbour. Another way to rephrase this is to say that, for every two disjoint $R,S\in \mathcal{R}_n$ we have $\text{N}_{\mathcal{R}_{n+N}}(R)\cap \text{N}_{\mathcal{R}_{n+N}}(S)=\varnothing.$ 

Of course, any two disjoint $R,S\in \mathcal{R}_n$ can be separated by neighbourhoods since they are closed sets and $X$ is normal. What Lemma \ref{lem:normalityconstant} says is that \textit{Smale spaces are normal in a highly controlled way.}
\begin{proof}[Proof of Lemma \ref{lem:normalityconstant}]
We will use induction on $n$. First, let $N\in \mathbb N$ be the smallest number that $$\text{N}_{\mathcal{R}_{1+N}}(R)\cap \text{N}_{\mathcal{R}_{1+N}}(S)=\varnothing,$$ for every disjoint $R,S\in \mathcal{R}_1$. This can be done since $(\mathcal{R}_n)_{n\geq 0}$ is refining and $X$ is normal. Assume now that $$\text{N}_{\mathcal{R}_{n+N}}(R)\cap \text{N}_{\mathcal{R}_{n+N}}(S)=\varnothing,$$ for every disjoint $R,S\in \mathcal{R}_{n}$ and we claim that the same holds for $n+1$. 

Let $R,S\in \mathcal{R}_{n+1}$ be disjoint. From Lemma \ref{lem:inductive step in Markov partitions} one can write $R=\varphi(R_i)\cap \varphi^{-1}(R_j)$ and $S=\varphi(S_i)\cap \varphi^{-1}(S_j)$, for $R_i,R_j,S_i,S_j\in \mathcal{R}_n$. Actually $R=[\varphi^{-1}(R_j),\varphi(R_i)]$ and $S=[\varphi^{-1}(S_j),\varphi(S_i)]$, and since they are disjoint we have either $\varphi^{-1}(R_j)\cap \varphi^{-1}(S_j)=\varnothing$ or $\varphi(R_i)\cap \varphi(S_i)=\varnothing$. Then, from the inductive step,
\begin{align*}
\text{N}_{\varphi(\mathcal{R}_{n+N})}(\varphi(R_i))&\cap \text{N}_{\varphi(\mathcal{R}_{n+N})}(\varphi(S_i))=\varnothing\\
\intertext{or}
\text{N}_{\varphi^{-1}(\mathcal{R}_{n+N})}(\varphi^{-1}(R_j))&\cap \text{N}_{\varphi^{-1}(\mathcal{R}_{n+N})}(\varphi^{-1}(S_j))=\varnothing.
\end{align*}

It follows that $\text{N}_{\mathcal{R}_{n+N+1}}(R)\cap \text{N}_{\mathcal{R}_{n+N+1}}(S)=\varnothing$. Indeed, assume there is some $T=\varphi^{-1}(T_j)\cap \varphi(T_i) \in \text{N}_{\mathcal{R}_{n+N+1}}(R)\cap \text{N}_{\mathcal{R}_{n+N+1}}(S)$ with $T_i,T_j\in \mathcal{R}_{n+N}$ then, 
\begin{align*}
\varphi(T_i) \in \text{N}_{\varphi(\mathcal{R}_{n+N})}(\varphi(R_i)) &\cap \text{N}_{\varphi(\mathcal{R}_{n+N})}(\varphi(S_i))\\
\intertext{and}
\varphi^{-1}(T_j)\in \text{N}_{\varphi^{-1}(\mathcal{R}_{n+N})}(\varphi^{-1}(R_j)) &\cap \text{N}_{\varphi^{-1}(\mathcal{R}_{n+N})}(\varphi^{-1}(S_j)).
\end{align*}
\enlargethispage{\baselineskip}
But at least one of the two intersections is empty leading to a contradiction.
\end{proof}

Let $M_N=\max \{\#\text{N}_{\mathcal{R}_n}(R): 0\leq n\leq N+1, R\in \mathcal{R}_n\}$. 
\begin{lemma}[Neighbouring Rectangles]\label{lem:neighbours}
For $n\geq N+2$ and every $R\in \mathcal{R}_n$ we have $$\#\num_{\mathcal{R}_n}(R)\leq (\#\mathcal{R}_1)^{2(N+1)}.$$ Consequently, $$\#\num_{\mathcal{R}_n}(R)\leq \max\{ (\#\mathcal{R}_1)^{2(N+1)},M_N\}$$ for every $n\in\mathbb N$ and $R\in \mathcal{R}_n$.
\end{lemma}
\begin{proof}
Assume to the contrary that there is some $n\geq N+2$ and $R\in \mathcal{R}_n$ such that $$\#\text{N}_{\mathcal{R}_n}(R)\geq (\#\mathcal{R}_1)^{2(N+1)}+1.$$ Let $$(R_{-n+1}^{(k)},\ldots , R_{0}^{(k)}, \ldots , R_{n-1}^{(k)}),$$ where $1\leq k\leq \#\text{N}_{\mathcal{R}_n}(R)$, be the sequences that denote the different elements of $\text{N}_{\mathcal{R}_n}(R)$, where each term is a rectangle in $\mathcal{R}_1$. Since there can only be up to $(\#\mathcal{R}_1)^2$ different pairs $(R_i^{(k)},R_j^{(k)})$, from the pigeonhole principle there are two sequences $S=(R_{-n+1}^{(k)},\ldots , R_{0}^{(k)}, \ldots , R_{n-1}^{(k)})$ and $T=(R_{-n+1}^{(j)},\ldots , R_{0}^{(j)}, \ldots , R_{n-1}^{(j)})$ for which 
\begin{equation}\label{eq:nrl}
R_{-n+1+i}^{(k)}=R_{-n+1+i}^{(j)}\enspace \text{and}\enspace R_{n-1-i}^{(k)}=R_{n-1-i}^{(j)},
\end{equation}
for every $0\leq i\leq N$. Since $S$ and $T$ have a common neighbour, that is $R$, by Lemma \ref{lem:normalityconstant} their ancestors $$S'=(R_{-n+1+N}^{(k)},\ldots , R_{0}^{(k)}, \ldots , R_{n-1-N}^{(k)})$$ and $$T'=(R_{-n+1+N}^{(k)},\ldots , R_{0}^{(j)}, \ldots , R_{n-1-N}^{(k)})$$ intersect. Observe that due to (\ref{eq:nrl}) both end-terms of $S'$ and $T'$ agree. At this point we should note that since $n\geq N+2$ the sequences $S'$ and $T'$ have at least three terms. Therefore, since $S\neq T$ again by (\ref{eq:nrl}) we conclude $S'\neq T'$ and hence there is some index $i$ in between such that $R_i^{(k)}\neq R_i^{(j)}$.

Choose a point $x\in S'\cap T'$ and we can find bi-infinite sequences $$(\ldots, R_{-n+1+N}^{(k)},\ldots , R_i^{(k)},\ldots, R_{n-1-N}^{(k)},\ldots )$$ and $$(\ldots, R_{-n+1+N}^{(k)},\ldots , R_i^{(j)},\ldots, R_{n-1-N}^{(k)},\ldots )$$ which both map to $x$ under Bowen's factor map, see Theorem \ref{thm:Bowen factor map}. This means the factor map has a diamond, which is a contradiction. 
\end{proof}

\section{Geometric approximations of Smale spaces}\label{sec:Geometric approximations of Smale spaces}

Let $(X,\varphi)$ be an irreducible Smale space and $\mathcal{R}_1$ be a Markov partition with $\overline{\diam}(\mathcal{R}_1)\leq \varepsilon_X''/2$ (see (\ref{eq:double_epsilon}) for the definition of $\varepsilon_X''$). Proposition \ref{prop:mainresultMarkovpartitions} yields the refining sequence of Markov partitions $(\mathcal{R}_n)_{n\geq 0}$ and the induced approximation graph $\Pi=(\mathcal{R},\mathcal{A})$. Our goal is to modify $\Pi$ so that we obtain a geometrically-essential approximation graph for $(X,\varphi)$.

For now, consider an arbitrary
\begin{equation}\label{eq:delta assumption 1}
0< \delta \leq \varepsilon_X''/4,
\end{equation}
and for every $R\in \mathcal{R}_1$ define its $\delta$-\textit{fattening} to be 
\begin{equation}\label{eq:deltafattening}
R^{\delta}=[R\cup (\partial^sR)^{\delta},R\cup (\partial^uR)^{\delta}],
\end{equation}
where
\begin{equation}
(\partial^sR)^{\delta}=\bigcup \{X^u(z,\delta):z\in \partial^sR\} \enspace \text{and}\enspace (\partial^uR)^{\delta}=\bigcup \{X^s(w,\delta):w\in \partial^uR\}.
\end{equation}
This is a well-defined rectangle with $\diam(R^{\delta})\leq \varepsilon_X',$ since 
\begin{equation}
\diam (R\cup (\partial^sR)^{\delta}\cup (\partial^uR)^{\delta})\leq  \diam(R)+2\delta\leq \varepsilon_X'',
\end{equation}
and a triangle inequality yields $\diam(R^{\delta})\leq \varepsilon_X'/2 +2\delta+\diam(R)< 7\varepsilon_X'/12$.

For $n\geq 2$ and $R\in \mathcal{R}_n$, written uniquely as $\bigcap_{i=1-n}^{n-1}\varphi^{-i}(R_{x_i})$ with $R_{x_i}\in \mathcal{R}_1$, define its $\delta$-fattening by 
\begin{equation}\label{eq:deltafattening2}
R^{\delta}=\bigcap_{i=1-n}^{n-1}\varphi^{-i}(R^{\delta}_{x_i}).
\end{equation}
As a result, for every $n\geq 0$ we obtain the covers
\begin{equation}\label{eq: fattened covers}
\mathcal{R}_n^{\delta}=\{R^{\delta}:R\in \mathcal{R}_n\}.
\end{equation}

For $n\in \mathbb N$, each cover $\mathcal{R}_{n+1}^{\delta}$ refines $\mathcal{R}_n^{\delta}$ since for $R\in \mathcal{R}_{n+1}$ and $S\in \mathcal{R}_n$ such that $R\subset S$ it holds $R^{\delta}\subset S^{\delta}$. Indeed, writing $R=\bigcap_{i=-n}^{n}\varphi^{-i}(R_{x_i})$ and $S=\bigcap_{i=1-n}^{n-1}\varphi^{-i}(S_{y_i})$ for unique $R_{x_i},S_{y_i}\in \mathcal{R}_1$, since $R\subset S$, we have $R_{x_i}=S_{y_i}$ for $|i|\leq n-1$ and hence $R^{\delta}\subset S^{\delta}$. Finally, with the same arguments that we used for $(\mathcal{R}_n)_{n\geq 0}$, we can show that the sequence $(\mathcal{R}_n^{\delta})_{n\geq 0}$ is refining and induces the approximation graph $\Pi^{\delta}$.
\begin{remark}
By choosing $\delta>0$ small enough we can make $\mathcal{R}_1^{\delta}$ to behave like a Markov partition. Using the Markov property we will prove that this behaviour passes on each $\mathcal{R}_n^{\delta}$, since the latter are inductively defined by $\mathcal{R}_1^{\delta}$. Moreover, the inductive definition of the refining sequence $(\mathcal{R}_n^{\delta})_{n\geq 0}$ will allow us to estimate the rate of decay of the Lebesgue covering numbers and the diameters of the covers $\mathcal{R}_n^{\delta}$.
\end{remark}

Recall that $\ell_X=\min \{\Lip (\varphi),\Lip(\varphi^{-1})\}$ and $\Lambda_X=\max \{\Lip (\varphi),\Lip (\varphi^{-1})\}$.  The key tool in this paper is the following.

\begin{thm}\label{thm:theoremgraphSmalespaces}
For every $n\in \mathbb N$, the open cover $\mathcal{R}_n^{\delta}$ in (\ref{eq: fattened covers}) behaves like the Markov partition $\mathcal{R}_n$, given that $\delta$ is sufficiently small. More precisely, there exist constants $\theta \in(0, \varepsilon_X]$ and $C,c>0$ and $\delta_1\in(0, \varepsilon_X''/4]$ that depend on $\mathcal{R}_1$ so that, for every $\delta \in (0,\delta_1]$, we have that
\begin{enumerate}[(1)]
\item $\overline{\diam}(\mathcal{R}_n^{\delta})\leq \lambda_X^{-n+1}\theta;$
\item $\mul(\mathcal{R}_n^{\delta})\leq (\#\mathcal{R}_1)^2;$
\item the number of neighbouring rectangles is uniformly bounded, meaning $$\sup\limits_n \max\limits_{R^{\delta}\in \mathcal{R}_n^{\delta}}\#\num_{\mathcal{R}_n^{\delta}}(R^{\delta})<\infty;$$
\item for every $\varepsilon \in (0,1)$, there is some $n_0\in \mathbb N$ such that, for $n\geq n_0$, we have $$ce^{2(\ent(\varphi)-\varepsilon) n}< \#\mathcal{R}_n^{\delta} <Ce^{2(\ent(\varphi)+\varepsilon) n};$$
\item the approximation graph $\Pi^{\delta}$, associated to the refining sequence of $\delta$-fat Markov partitions $(\mathcal{R}^{\delta}_n)_{n\geq 0}$, is metrically-essential.
\end{enumerate} 
If either $\varphi$ or $\varphi^{-1}$ are Lipschitz, there is some $\zeta\in (0,\theta]$ (independent of $\delta$) so that
\begin{enumerate}[resume]
\item $\underline{\text{diam}}(\mathcal{R}_n^{\delta})\geq \underline{\text{diam}}(\mathcal{R}_n)\geq \ell_X^{-n+1} \zeta.$
\end{enumerate}
If in addition $\varphi$ is bi-Lipschitz, for every $\delta \in(0, \delta_1]$, there is some $\eta\in (0,\theta]$ so that
\begin{enumerate}[resume]
\item $\Leb(\mathcal{R}_n^{\delta})\geq \Lambda_X^{-n+1}\eta$.
\end{enumerate}
\end{thm}

From Theorem \ref{thm:theoremgraphSmalespaces} and Lemma \ref{lem:Artiguelemma} we obtain the following.

\begin{cor}\label{cor:topconjgraph}
The irreducible Smale space $(X,\varphi)$ is topologically conjugate to a Smale space $(Y,\psi)$ which admits a refining sequence that satisfies all conditions of Theorem \ref{thm:theoremgraphSmalespaces} with $\lambda_{Y}=\ell_{Y}=\Lambda_{Y}.$
\end{cor}

\begin{example}
Suppose that $(X,\varphi)$ is an irreducible topological Markov chain equipped with the ultrametric (\ref{eq:SFTmetric}). Also, assume that the Markov partition $\mathcal{R}_1$ consists of sufficiently small symmetric cylinder sets, and hence each $\mathcal{R}_n$ consists of smaller symmetric cylinder sets. Then, for every $0<\delta \leq \varepsilon_X''/4$, the $\delta$-fattening process of $(\mathcal{R}_n)_{n\geq 0}$ is trivial since $\mathcal{R}_n^{\delta}=\mathcal{R}_n$, for every $n\in \mathbb N$. More precisely, for every $R\in \mathcal{R}_1$ we have that $\partial^sR=\partial^uR=\varnothing$, and hence $$(\partial^sR)^{\delta}=(\partial^uR)^{\delta}=\varnothing.$$ Therefore, from (\ref{eq:deltafattening}) we have $R^{\delta}=[R,R]=R$. In order to see how Theorem \ref{thm:theoremgraphSmalespaces} applies, recall that $\lambda_X=\ell_X=\Lambda_X=2$. Finally, the multiplicities of the covers are clearly equal to one, and condition (4) follows from Lemma \ref{thm: upper bound on cardinality in SFT}.
\end{example}

\begin{example}
Suppose now that $(X,\varphi)$ is the dyadic solenoid of Subsection \ref{sec:2solenoid}. From \cite[Prop. 2.3.4]{BS} we see that it is mixing. In order to build a Markov partition $\mathcal{R}_1$, consider the decomposition of the unit circle $\mathbb T$ into the closed upper half $e_0$ and the closed lower half $e_1$. Then, let 
\begin{align*}
E_0&=\{x\in X:x_1\in e_0\}\\
E_1&=\{x\in X:x_1\in e_1\},
\end{align*}
and one can observe that $\Int (\varphi^{-1}(E_i)\cap E_j)\neq \varnothing$, for all $0\leq i,j\leq 1$. Each of the four sets corresponds to one of the quadrants of the circle in a clear way. With a bit more effort one can show that the cover $E=\{E_0,E_1\}$ yields a $2\texttt{-}\text{to}\texttt{-}1$ factor map from the full-two shift to $(X,\varphi)$. Roughly speaking, this map encodes $x\in X$ by a bi-infinite sequence containing all encodings for the dyadic expansions of the coordinates of $x$. However, the sets $E_0,E_1$ are not quite rectangles since their diameter is large. But for sufficiently large $m\in \mathbb N$ we get the Markov partition $$\mathcal{R}_1=\bigvee_{j=1-m}^{m-1}\varphi^{-j}(E),$$ on which we can apply the $\delta$-fattening process. Since the local stable sets are Cantor sets, for every $R\in \mathcal{R}_1$ it holds that $\partial^uR=\varnothing$. As a result, the $\delta$-fattening happens only on the local unstable sets and $R^{\delta}=[R\cup (\partial^sR)^{\delta},R]$. The proof of Lemma \ref{thm: fat rectangle} helps to visualise the rectangles.

Finally, let us understand some important parts of Theorem \ref{thm:theoremgraphSmalespaces} in this case. From Subsection \ref{sec:2solenoid} we have that $\lambda_X=2, \ell_X=2$ and $\Lambda_X\leq 5/2$. Therefore, the diameters $\overline{\diam}(\mathcal{R}_n^{\delta})\sim 2^{-n}$, the Lebesgue numbers $\Leb(\mathcal{R}_n^{\delta})\gtrsim (5/2)^{-n}$ and the cardinalities $\#\mathcal{R}_n^{\delta} \sim 2^{2n}$. Moreover, from the construction of the factor map we have $\mul(\mathcal{R}_n^{\delta})\leq 2$. In fact, $\mul(\mathcal{R}_n^{\delta})= 2$ for all $n\in \mathbb N$, since the two fixed points of the shift space are mapped down to the fixed point $(1,1,1,\ldots)\in E_0\cap E_1$. Finally, since $(\mathcal{R}_n^{\delta})_{n\geq 0}$ is a refining sequence of open covers, it holds that $\dim X \leq 1$ and because $X$ is connected we get the already known fact that $\dim X=1$.  
\end{example}

\subsection{Proof of conditions (1)\texttt{-}(5) of Theorem \ref{thm:theoremgraphSmalespaces}} Their proof consists of the following lemmas and corollaries. For the rest of this subsection we consider the refining sequences $(\mathcal{R}_n^{\delta})_{n\geq 0}$ constructed in (\ref{eq: fattened covers}), for $\delta \in (0,\varepsilon_X''/4].$ But first, recall that if $S$ is a rectangle in $X$ and $x\in S$, the set $X^s(x,2\varepsilon_X')\cap S$ is denoted by $X^s(x,S)$ and similarly $X^u(x,2\varepsilon_X')\cap S$ is denoted by $X^u(x,S)$.
\begin{lemma}\label{thm: fat rectangle}
For $\delta \in (0,\varepsilon_X''/4]$ and $n\in \mathbb N$, the cover $\mathcal{R}_n^{\delta}$ consists of open rectangles.
\end{lemma}
\begin{proof}
We just need to prove that every $R^{\delta}\in \mathcal{R}_1^{\delta}$ is open in $X$, and in the process we will demonstrate how it differs from $R\in \mathcal{R}_1$. For any $x\in R$ define 
\begin{equation}\label{eq:stableunstablefattening}
(\partial X^u(x,R))^{\delta}= \bigcup_{z\in \partial X^u(x,R)}X^u(z,\delta) \enspace \text{and} \enspace (\partial X^s(x,R))^{\delta}= \bigcup_{w\in \partial X^s(x,R)}X^s(w,\delta),
\end{equation}
along with the open rectangle 
\begin{equation}\label{eq:fatlocalisedrectangle}
R_x^{\delta}:= [X^u(x,R)\cup (\partial X^u(x,R))^{\delta}, X^s(x,R)\cup (\partial X^s(x,R))^{\delta}].
\end{equation}
The claim is that $R^{\delta}=\bigcup_{x\in R} R^{\delta}_x$ and hence $R^{\delta}$ is open. Note here that if we let $R_x=[X^u(x,R),X^s(x,R)]$ for $x\in R$, then $R_x=R=R_y$ for any $x,y\in R$. However, in the $\delta$-fattened version we may get $R_x^{\delta}\neq R_y^{\delta}$, for some $x\neq y\in R$, because there is no more control of the local stable and the unstable sets as soon as they get outside of $R$. 

First we prove that $$R\cup (\partial^sR)^{\delta}= \bigcup_{x\in R} X^u(x,R)\cup (\partial X^u(x,R))^{\delta},$$ where it is straightforward to see that $R=\bigcup_{x\in R} X^u(x,R)$. The claim is that
\begin{equation}
(\partial^sR)^{\delta}= \bigcup_{x\in R} (\partial X^u(x,R))^{\delta}.
\end{equation}
To see this, let $x\in R$ and if $z\in \partial X^u(x,R)$ then $z\in \partial^s R$ since, following the equations (\ref{eq: boundaries of a rectangle}), it holds that $z\in X^s(z,R)=[z,X^s(x,R)]\subset [\partial X^u(x,R), X^s(x,R)]=\partial^s R.$ Therefore, $\bigcup_{x\in R} (\partial X^u(x,R))^{\delta} \subset (\partial^sR)^{\delta}$. For the other inclusion, note that if $z\in \partial^s R$, then $z\in \partial X^u(z,R)$, for if $z\in \text{int}(X^u(z,R))$ then $X^s(x,R)\cap \Int(R)\neq \varnothing.$ Similarly, it holds that $$R\cup (\partial^u R)^{\delta}= \bigcup_{y\in R} X^s(y,R)\cup (\partial X^s(y,R))^{\delta}.$$

As a result, 
\begin{align*}
R^{\delta}&= \bigcup_{x,y \in R} [X^u(x,R)\cup (\partial X^u(x,R))^{\delta}, X^s(y,R)\cup (\partial X^s(y,R))^{\delta}]\\
&= \hspace{-0.7mm} \bigcup_{[y,x] \in R} [X^u([y,x],R)\cup (\partial X^u([y,x],R))^{\delta}, X^s([y,x],R)\cup (\partial X^s([y,x],R))^{\delta}]\\
&= \hspace{-0.7mm} \bigcup_{[y,x] \in R} R_{[y,x]}^{\delta}\\
&= \hspace{1.3mm} \bigcup_{x\in R} R_x^{\delta}. \qedhere
\end{align*}
\end{proof}

\begin{lemma}\label{lem:sequenceofdiameters}
For $\delta \in (0,\varepsilon_X''/4]$ and $n\in \mathbb N$, it holds $\overline{\diam}(\mathcal{R}_n^{\delta})\leq \min\{\lambda_X^{-n+1}\varepsilon_X,\varepsilon_X'\}$.
\end{lemma}
\begin{proof}
Let $n\in \mathbb N$ and $R^{\delta}\in \mathcal{R}_n^{\delta}$ with some $x,y\in R^{\delta}$. From the definition of $R^{\delta}$ we have $d(\varphi^i(x),\varphi^i(y))\leq \varepsilon_X'$, for all $|i|\leq n-1$. In particular, using the bracket axiom (B4), we have
$$[\varphi^{-n+1}(x),\varphi^{-n+1}(y)]=\varphi^{-n+1}([x,y])$$
and
$$d(\varphi^{-n+1}(x),[\varphi^{-n+1}(x),\varphi^{-n+1}(y)])\leq \varepsilon_X/2.$$
Therefore, $\varphi^{-n+1}([x,y])\in X^s(\varphi^{-n+1}(x),\varepsilon_X/2)$ and hence $[x,y]\in X^s(x,\lambda_X^{-n+1}\varepsilon_X/2)$. Similarly, $[x,y]\in X^u(y,\lambda_X^{-n+1}\varepsilon_X/2)$ and from triangle inequality, $d(x,y)\leq \lambda_X^{-n+1}\varepsilon_X.$
\end{proof}
We want to show that our method produces covers which are closely related to the Markov partitions. First we study all the possible overlaps that can occur between elements of $\mathcal{R}_n^{\delta}$, for a given $\delta \in (0,\varepsilon_X''/4]$. For a finite cover $\mathcal{U}=\{U_1,\ldots , U_{\ell}\}$ of $X$ let 
\begin{equation}\label{eq:overlapsofcover}
M(\mathcal{U})=\{\{U_i\}_{i\in I}: I\subset \{1,\ldots , \ell\} \enspace \text{and} \enspace \bigcap_{i\in I}U_i \neq \varnothing\}
\end{equation}
be the nerve of $\mathcal{U}$ and $M^{\#}(\mathcal{U})=\{\#I:\{U_i\}_{i\in I}\in M(\mathcal{U})\}.$ It should be clear that $\text{max} (M^{\#}(\mathcal{U}))$ is the multiplicity of $\mathcal{U}$. For $\delta \in (0,\varepsilon_X''/4]$ and $n\in \mathbb N$, consider the map
\begin{equation}\label{eq:overlapmap}
M_n^{\delta}: M(\mathcal{R}_n)\to M(\mathcal{R}_n^{\delta}), \enspace \{R_i\}_{i\in I}\mapsto \{R_i^{\delta}\}_{i\in I} 
\end{equation}
and we aim to show that for small enough $\delta$, if $\bigcap_{i \in I}R_i^{\delta}\neq \varnothing$ then $\bigcap_{i \in I}R_i\neq \varnothing$, meaning that the map $M_n^{\delta}$ is surjective.

Let $\mathcal{R}_1 = \{R_1,\ldots , R_{\ell}\}$ and $E(\mathcal{R}_1)=\{I \subset \{1,\ldots , \ell\}: \bigcap_{i\in I}R_i=\varnothing\}.$ Then, there is $\delta_0\in (0,\varepsilon_X''/4]$ such that for every $I\in E(\mathcal{R}_1)$ it holds $\bigcap_{i\in I}R_i^{\delta_0}=\varnothing$. Indeed, let $I\in E(\mathcal{R}_1)$ and then for every $i\in I$ there is some $\delta'_{I,i}>0$ such that the $\delta'_{I,i}$-neighbourhood of $R_i$, denoted by $B(R_i, \delta'_{I,i})$, satisfies $B(R_i, \delta'_{I,i})\cap B(\bigcap_{j\in I, j\neq i} R_j, \delta'_{I,i})=\varnothing.$ Also, there is $0< \delta_{I,i}< \delta'_{I,i}$ such that $R^{\delta_{I,i}}_i \subset B(R_i, \delta'_{I,i})$ and $\bigcap_{j\in I, j\neq i} R_j^{\delta_{I,i}}\subset B(\bigcap_{j\in I, j\neq i} R_j, \delta'_{I,i}).$ So choosing
\begin{equation}\label{eq: delta assumption 2}
\delta_0= \text{min}\{\varepsilon_X''/4, \text{min} \{\delta_{I,i}: I\in E(\mathcal{R}_1), i\in I\}\}
\end{equation}
has the required property.
\begin{lemma}\label{thm: fattening map multiplicity}
For every $\delta\in (0,\delta_0]$ and $n\in \mathbb N$, the map $M_n^{\delta}$ defined in (\ref{eq:overlapmap}) is surjective.
\end{lemma}
\begin{proof}
For $n=0$ it is trivially true and from the choice of $\delta$ the map $M_1^{\delta}$ is surjective. Lets assume $M_n^{\delta}$ is surjective and we claim that $M_{n+1}^{\delta}$ is too. Let $\{R_i^{\delta}\}_{i\in I}\in M(\mathcal{R}_{n+1}^{\delta})$ and from Lemma \ref{lem:inductive step in Markov partitions} every $$R_i^{\delta}=\varphi (R_{i1}^{\delta}) \cap R_{i2}^{\delta} \cap \varphi^{-1} (R_{i3}^{\delta})$$ with $R_{i1}^{\delta}, R_{i2}^{\delta}, R_{i3}^{\delta}\in \mathcal{R}_n^{\delta}$ and $R_i=\varphi (R_{i1}) \cap R_{i2}\cap \varphi^{-1} (R_{i3})\in \mathcal{R}_{n+1}$. Then it holds $$\bigcap_{i\in I}R_i^{\delta} = \bigcap_{i\in I} \varphi (R_{i1}^{\delta}) \cap R_{i2}^{\delta} \cap \varphi^{-1} (R_{i3}^{\delta})\neq \varnothing$$ and we want to prove that $$\bigcap_{i\in I}R_i = \bigcap_{i\in I} \varphi (R_{i1}) \cap R_{i2} \cap \varphi^{-1} (R_{i3})\neq \varnothing.$$ 

\noindent Equivalently, using part (3) of Proposition \ref{thm:Markov property rectangles} we want to prove that $$\bigcap_{i\in I}[\varphi^{-1} (R_{i3}), \varphi (R_{i1})]\neq \varnothing.$$ From the inductive step we obtain that
$$\bigcap_{i\in I}\varphi (R_{i1}^{\delta})\neq \varnothing \Rightarrow \bigcap_{i\in I}\varphi (R_{i1})\neq \varnothing $$
and also
$$\bigcap_{i\in I}\varphi^{-1} (R_{i3}^{\delta})\neq \varnothing \Rightarrow \bigcap_{i\in I}\varphi^{-1} (R_{i3})\neq \varnothing.$$
Hence
\begin{align*}
\varnothing &\neq [ \bigcap_{i\in I}\varphi^{-1} (R_{i3}), \bigcap_{i\in I}\varphi (R_{i1})]\\
&\subset \bigcap_{i\in I} [\varphi^{-1} (R_{i3}), \varphi (R_{i1})]\\
&=\bigcap_{i\in I} R_i.
\end{align*}
Thus, by induction the maps $M_n^{\delta}$ are surjective.
\end{proof}

We now aim to find $\delta_1\in (0,\delta_0]$ so that, for every $\delta \in (0,\delta_1]$, if $R,S\in \mathcal{R}_n$ with $R\neq S$ then $R\not \subset S^{\delta}$. In particular, for such $\delta$, the $\delta$-fattening maps 
\begin{equation}\label{eq:deltafatmap}
F_n^{\delta}:\mathcal{R}_n\to \mathcal{R}_n^{\delta}, \enspace R\mapsto R^{\delta}
\end{equation}
and the maps $M_n^{\delta}$ defined in (\ref{eq:overlapmap}) will be shown to be bijective, for every $n\in \mathbb N$. 

This is a subtle procedure that requires the following concepts. First, for every $n\in \mathbb N$ and $R\in \mathcal{R}_n$ define 
\begin{equation}\label{eq:interiorboundaries1}
\partial^{s,o}R=\{x\in \partial^s R: X^u(x,R)\cap \Int(R)\neq \varnothing\}
\end{equation}
and 
\begin{equation}\label{eq:interiorboundaries2}
\partial^{u,o}R=\{x\in \partial^u R: X^s(x,R)\cap \Int(R)\neq \varnothing\}.
\end{equation}
Then, consider the sets
\begin{equation}\label{eq:stableneighbours}
\num_{\mathcal{R}_n}^s(R)=\{S\in \num_{\mathcal{R}_n}(R)\setminus \{R\}:S\cap \partial^{u,o}R =\varnothing\}
\end{equation}
and
\begin{equation}\label{eq:unstableneighbours}
\num_{\mathcal{R}_n}^u(R)=\{S\in \num_{\mathcal{R}_n}(R)\setminus \{R\}:S\cap \partial^{s,o}R =\varnothing\}
\end{equation}
which are the \textit{stable} and \textit{unstable neighbours} of $R$, respectively. These sets provide a decomposition in the sense that 
\begin{equation}\label{eq:totalneighbours}
\num_{\mathcal{R}_n}(R)\setminus \{R\}=\num_{\mathcal{R}_n}^s(R)\cup \num_{\mathcal{R}_n}^u(R).
\end{equation}

Indeed, if $S\in \num_{\mathcal{R}_n}(R)\setminus \{R\}$ with $S\not \in \num_{\mathcal{R}_n}^s(R)$ and $S\not \in \num_{\mathcal{R}_n}^u(R)$ then there are some $x\in S\cap \partial^{u,o}R$ and $y\in S\cap \partial^{s,o} R$, and hence $[x,y]\in S\cap \Int(R)$. However, $S\cap \Int(R)=\varnothing$ which results in a contradiction.

For every $R\in \mathcal{R}_1$ choose some $x_R\in \Int(R)$ and then for all $T\in \num_{\mathcal{R}_1}^s(R)$ and $S\in \num_{\mathcal{R}_1}^u(R)$ we have $T\cap X^s(x_R,R)=\varnothing$ and $S\cap X^u(x_R,R)=\varnothing.$ Consequently, there is some small enough $\delta'>0$ so that 
\begin{equation}
T^{\delta'}\cap X^s(x_R,R)=\varnothing
\end{equation}
and 
\begin{equation}
S^{\delta'}\cap X^u(x_R,R)=\varnothing
\end{equation}
for every $R\in \mathcal{R}_1$ and $T\in \num_{\mathcal{R}_1}^s(R), \thinspace S\in \num_{\mathcal{R}_1}^u(R)$. Our last choice for how small $\delta$ should be is,
\begin{equation}\label{eq: delta assumption 3}
\delta_1= \text{min}\{\delta_0,\delta'\}.
\end{equation}

\begin{lemma}\label{lem: inductive lemma for injectivity of fat map}
Let $\delta \in (0,\delta_1].$ For every $n \in \mathbb N$ and $S,R \in \mathcal{R}_n$ with $S\in \num_{\mathcal{R}_n}(R)\setminus \{R\}$ we have that
\begin{enumerate}[(1)]
\item if  $S\in \num_{\mathcal{R}_n}^u(R)$ and $x\in S\cap R$ then $X^s(x,R)\not \subset X^s(x,S^{\delta})$;
\item if  $S\in \num_{\mathcal{R}_n}^s(R)$ and $x\in S\cap R$ then $X^u(x,R)\not \subset X^u(x,S^{\delta})$.
\end{enumerate}
\end{lemma}

\begin{proof}
We will prove part (1) by induction on $n$ and the proof for part (2) is similar. Assume to the contrary that there are $S,R \in \mathcal{R}_1$ with $S\in \num_{\mathcal{R}_1}^u(R)$ and some $x\in S\cap R$ such that $X^s(x,R)\subset X^s(x,S^{\delta})$. In particular, $X^s(x,R)\subset S^{\delta}$ and since $R$ is a rectangle we have $$[x,x_R]\in X^s(x,R)\cap X^u(x_R,R)\subset S^{\delta}.$$ However, from the choice of $\delta_1$ in (\ref{eq: delta assumption 3}) we have $S^{\delta}\cap X^u(x_R,R)=\varnothing$ leading to a contradiction.

Suppose now that part (1) is true for some $n\in \mathbb N$. We claim that it is also true for $n+1$. Again assume to the contrary that there are $S,R \in \mathcal{R}_{n+1}$ with $S\in \num_{\mathcal{R}_{n+1}}^u(R)$ and some $x\in S\cap R$ such that $X^s(x,R)\subset X^s(x,S^{\delta})$. 

As usual write $R=[\varphi^{-1}(R_j),\varphi(R_i)]$ and $S=[\varphi^{-1}(S_j),\varphi(S_i)]$ for $R_i,R_j,S_i,S_j\in \mathcal{R}_n$. It holds that $$X^s(x,R)=X^s(x,\varphi(R_i)), \thinspace X^s(x,S)= X^s(x,\varphi (S_i))$$ and $X^s(x,S^{\delta})\subset X^s(x,\varphi(S_i^{\delta}))$ since $S^{\delta} \subset \varphi(S_i^{\delta})$. From the assumption we obtain that $X^s(x,\varphi(R_i))\subset X^s(x,\varphi(S_i^{\delta}))$. Then applying the map $\varphi^{-1}$ results to $$X^s(\varphi^{-1}(x),R_i)\subset X^s(\varphi^{-1}(x),S_i^{\delta})$$ since $R_i,S_i^{\delta}$ are rectangles and the bracket map is locally bijective. Clearly $\varphi^{-1}(x)\in S_i\cap R_i$ and we have to show that $S_i\in \num_{\mathcal{R}_n}^u(R_i)$ in order to get a contradiction from the inductive step.

First note that 
\begin{equation}\label{eq:maininjectivitylemma1}
X^s(x,S)\cap \Int (X^s(x,R))=\varnothing,
\end{equation}
for if not, we would have $$[X^u(x,S),X^s(x,S)]\cap [X^u(x,R),\Int (X^s(x,R))]\neq \varnothing.$$ Then since the intersection of $S$ and $R$ can happen only on their boundaries, this is equivalent to $S\cap \partial^{s,o}R\neq \varnothing,$ meaning that $S\not \in \num_{\mathcal{R}_{n+1}}^u(R)$. Equation (\ref{eq:maininjectivitylemma1}) implies that 
\begin{equation}\label{eq:maininjectivitylemma2}
X^s(\varphi^{-1}(x),S_i)\cap \Int(X^s(\varphi^{-1}(x),R_i))=\varnothing
\end{equation}
and is easy to observe that $S_i\in \num_{\mathcal{R}_n}(R_i)\setminus \{R_i\}$. Now assume to the contrary that $S_i\not \in \num_{\mathcal{R}_n}^u(R_i)$ and hence there is some $y\in S_i\cap \partial^{s,o}R_i$. Then 
\begin{align*}
X^s(\varphi^{-1}(x),S_i)\cap X^u(y,\varepsilon_X/2)&=[\varphi^{-1}(x),y]\\
\intertext{since both $\varphi^{-1}(x),y\in S_i$ but also}
\Int(X^s(\varphi^{-1}(x),R_i))\cap X^u(y,\varepsilon_X/2)&=[\varphi^{-1}(x),y]
\end{align*}
since $y\in \partial^{s,o}R_i$. However, this contradicts (\ref{eq:maininjectivitylemma2}).  
\end{proof}

\begin{lemma}\label{lem: injective fattening}
Let $\delta \in (0,\delta_1].$ For every $n\in \mathbb N$ and $S,R\in \mathcal{R}_n$ such that $S\neq R$ we have that $R\not \subset S^{\delta}$. Consequently, the maps $F_n^{\delta}$ and $M_n^{\delta}$ are bijective.
\end{lemma}

\begin{proof}
Let $n\in \mathbb N$ and assume to the contrary that there are some $S,R\in \mathcal{R}_n$ with $S\neq R$ such that $R\subset S^{\delta}$. From Lemma \ref{thm: fattening map multiplicity} we obtain that $S\cap R\neq \varnothing$ and so for any $x\in S\cap R$ it holds $X^s(x,R)\subset X^s(x,S^{\delta})$ and $X^u(x,R)\subset X^u(x,S^{\delta})$. However, from Lemma \ref{lem: inductive lemma for injectivity of fat map} we have that $X^s(x,R)\not \subset X^s(x,S^{\delta})$ or $X^u(x,R)\not \subset X^u(x,S^{\delta})$. This gives the desired contradiction.
\end{proof}

Lemma \ref{lem: injective fattening} yields the following corollaries.
\begin{cor}\label{cor:quasiessentialSmalespaces}
For every $\delta \in (0,\delta_1],$ the approximation graph $\Pi^{\delta}$ is metrically-essential.
\end{cor}
\begin{proof}
From Lemma \ref{lem:Markovapproximationgraph1} the underlying approximation graph $\Pi$ is essential and from Lemma \ref{lem: injective fattening} the $\delta$-fattening maps $F_n^{\delta}$ are bijective. It is easy to see that $\coprod_{n\geq 0} F_n^{\delta}: \Pi\to \Pi^{\delta}$ is a graph homomorphism which induces the map $F^{\delta}:\mathcal{P}_{\Pi}\to \mathcal{P}_{\Pi^{\delta}}$ that satisfies $\pi_{\Pi^{\delta}}\circ F^{\delta}= \pi_{\Pi}$ because, if $\widetilde{p}\in \mathcal{P}_{\Pi}$ then $$\varnothing \neq \bigcap_{n\geq 0}r(p_n)\subset \bigcap_{n\geq 0}r(p_n^{\delta}).$$
\end{proof}
\enlargethispage{\baselineskip}
\begin{cor}\label{cor:samestructure}
Let $\delta \in (0,\delta_1].$ For every $n\geq 0$ and $R\in \mathcal{R}_n$ the following hold.
\begin{enumerate}[(1)]
\item $\#\mathcal{R}_n^{\delta}=\#\mathcal{R}_n$;
\item $\#\num_{\mathcal{R}_n^{\delta}}(R^{\delta})=\#\num_{\mathcal{R}_n}(R)$;
\item $\mul(\mathcal{R}_n^{\delta})=\mul(\mathcal{R}_n)$.
\end{enumerate}
\end{cor}
\begin{proof}
Observe that Lemma \ref{lem: injective fattening} implies that any $R^{\delta}\in \mathcal{R}_n^{\delta}$ is written uniquely in the form $\varphi(R_i^{\delta})\cap R_j^{\delta}\cap \varphi^{-1}(R_k^{\delta})$, for $R_i^{\delta},R_j^{\delta},R_k^{\delta}\in \mathcal{R}_n^{\delta}$.
\end{proof}

\subsection{Proof of conditions (6), (7) of Theorem \ref{thm:theoremgraphSmalespaces}} Consider the refining sequence $(\mathcal{R}_n^{\delta})_{n\geq 0}$ defined in (\ref{eq: fattened covers}), for some $\delta \in (0,\delta_1]$. The constant $\delta_1$ is defined in (\ref{eq: delta assumption 3}). Our aim is to investigate the properties of $(\mathcal{R}_n^{\delta})_{n\geq 0}$ when $\ell_X=\min \{\Lip (\varphi),\Lip(\varphi^{-1})\}$ or $\Lambda_X=\max \{\Lip (\varphi),\Lip (\varphi^{-1})\}$ is finite. 

Recall again that, if $S$ is a rectangle in $X$ and $x\in S$, the set $X^s(x,2\varepsilon_X')\cap S$ is denoted by $X^s(x,S)$ and similarly $X^u(x,2\varepsilon_X')\cap S$ is denoted by $X^u(x,S)$. Also, note that both sets are compact subsets of $X$.

\begin{lemma}\label{lem:Lebesgue numbers}
There exists some $\eta >0$ such that for every $n\geq 0$ it holds that $$\Leb(\mathcal{R}_n^{\delta})\geq \Lambda_X^{-n+1}\eta.$$
\end{lemma}
\begin{proof}
The case $n=0$ is trivial. We claim that there exists some $\eta >0$ such that for every $n\in \mathbb N$, we have $d(x, X\setminus R^{\delta})\geq \Lambda_X^{-n+1}\eta, $ for every $x\in X$ and $R\in \num_{\mathcal{R}_n}(\{x\})$. Then we will have
\begin{align*}
\text{Leb}(\mathcal{R}_n^{\delta})&= \text{inf}_{x\in X}\text{sup}_{R^{\delta} \in \mathcal{R}_n^{\delta}}d(x, X\setminus R^{\delta})\\
&\geq  \text{inf}_{x\in X}\text{sup}_{R\in \num_{\mathcal{R}_n}(\{x\})}d(x, X\setminus R^{\delta})\\
&\geq \Lambda_X^{-n+1}\eta.
\end{align*}

To prove the claim, let $R_{1,k}\in \mathcal{R}_1$, for $1\leq k\leq \# \mathcal{R}_1$. We have that $(X\setminus R_{1,k}^{\delta})\cap R_{1,k}=\varnothing$ and define $\eta_k>0$ so that $d(X\setminus R_{1,k}^{\delta}, R_{1,k})=2\eta_k$. Let $$\eta=\min\{\eta_k: 1\leq k\leq \# \mathcal{R}_1\}$$ and then, for every $x\in X$ and $R\in \num_{\mathcal{R}_1}(\{x\})$, it holds that $d(x, X\setminus R^{\delta})\geq \eta.$ This proves the case $n=1$.

Assume that for some $n\in \mathbb N$ we have $d(x, X\setminus R^{\delta})\geq \Lambda_X^{-n+1}\eta, $ for every $x\in X$ and $R\in \num_{\mathcal{R}_n}(\{x\})$. We claim that $$d(x,X\setminus ( \varphi (R_i^{\delta})\cap R_j^{\delta}\cap \varphi^{-1}(R_k^{\delta})))\geq \Lambda_X^{-n}\eta,$$ for every $x\in X$ and $\varphi (R_i)\cap R_j\cap \varphi^{-1}(R_k)\in N_{\mathcal{R}_{n+1}}(\{x\})$ with $R_i,R_j,R_k\in \mathcal{R}_n$. For $x\in X$ and $\varphi (R_i)\cap R_j\cap \varphi^{-1}(R_k)\in \num_{\mathcal{R}_{n+1}}(\{x\})$ one has that $d(x,X\setminus ( \varphi (R_i^{\delta})\cap R_j^{\delta}\cap \varphi^{-1}(R_k^{\delta})))$ is equal to 
$$\text{min} \{d(x,X\setminus \varphi (R_i^{\delta})), d(x,X\setminus R_j^{\delta}) , d(x,X\setminus \varphi^{-1}(R_k^{\delta}))\}$$ which is greater or equal to $$\min \{\Lambda_X^{-1}\Lambda_X^{-n+1}\eta,\Lambda_X^{-n+1}\eta, \Lambda_X^{-1}\Lambda_X^{-n+1}\eta\}$$ which is $\Lambda_X^{-n}\eta$, concluding the induction argument.
\end{proof}

The next results show that Smale spaces with Lipschitz dynamics can be controlled in a refined way. Our approach makes use of the next lemma that holds for an arbitrary Smale space.
 
For any closed rectangle $R\subset X$ that has a local stable and unstable set of cardinality at least two, let 
\begin{equation}\label{eq:stablediameter}
\underline{\text{diam}}_s(R)=\text{inf}\{\text{diam} (X^s(x,R)):x\in R\}
\end{equation}
and
\begin{equation}\label{eq:unstablediameter}
\underline{\text{diam}}_u(R)=\text{inf}\{\text{diam} (X^u(x,R)):x\in R\}.
\end{equation}

It is clear that for every $x\in R$, the diameters $\text{diam} (X^s(x,R))$ are non-zero, since all the local stable sets are mutually homeomorphic and hence of cardinality at least two. For the same reason the diameters $\text{diam} (X^u(x,R))$ are non-zero. The fact that $\underline{\text{diam}}_s(R),\underline{\text{diam}}_u(R) >0$ follows from the compactness of $R$ and the next lemma. But first, for a closed rectangle $R\subset X$, denote by $\mathcal{K}(R)$ the set of its compact subsets and by $d_H$ the usual Hausdorff metric on $\mathcal{K}(R)$ which is described for the reader's convenience in the proof. 
\begin{lemma}\label{lem: continuous diameter}
For every closed rectangle $R\subset X$, the maps from $(R,d)$ to $(\mathcal{K}(R),d_H)$ given by $x\mapsto X^s(x,R)$ and $x\mapsto X^u(x,R)$ are continuous. In particular, the maps $R\ni x \mapsto \text{diam} (X^s(x,R))$ and $R\ni x \mapsto \text{diam} (X^u(x,R))$ are continuous.
\end{lemma}
\begin{proof}
We prove the unstable case and the stable case is similar. Let $y\in R$ and $(x_n)_{n\geq 0}$ be a sequence in $R$ which converges to $y$. We will prove that $$\lim_{n\to \infty}d_H(X^u(x_n,R),X^u(y,R))=0$$ where $d_H$ is the Hausdorff distance that, since $X^u(x_n,R),X^u(y,R)$ are compact, is given by
\begin{equation}
d_H(X^u(x_n,R),X^u(y,R))=\text{max}\{ C(x_n,y,R), C'(x_n,y,R)\},
\end{equation}
where 
\begin{align*}
C(x_n,y,R)&=\max\limits_{z\in X^u(x_n,R)}d(z, X^u(y,R))\\
\intertext{and}
C'(x_n,y,R)&= \max\limits_{w\in X^u(y,R)} d(w, X^u(x_n,R)).
\end{align*}
A straightforward computation shows that 
\begin{equation}
\lvert \text{diam}X^u(x_n,R) - \text{diam}X^u(y,R)\rvert \leq 2 d_H(X^u(x_n,R),X^u(y,R)) 
\end{equation}
and hence the map $R\ni x \mapsto \text{diam} X^u(x,R)$ is continuous.

We now claim that $$\lim_{n\to \infty} C(x_n,y,R)=0.$$ Assume to the contrary that it does not converge to zero. Then there is some $\varepsilon >0$ and a subsequence $(x_{n_k})_{k\geq 0}$ in $R$ such that $$\max\limits_{z\in X^u(x_{n_k},R)}d(z, X^u(y,R))\geq \varepsilon,$$ for every $k\geq 0$. Hence, there are some $z_{n_k}\in X^u(x_{n_k},R)$ so that $d(z_{n_k}, X^u(y,R))\geq \varepsilon$, for every $k\geq 0$. However, by compactness of $R$, there is a convergent subsequence $(z_{n_{k_\ell}})_{\ell \geq 0}$ that converges to some $z' \in R$. But then, $z_{n_{k_\ell}} = [z_{n_{k_\ell}}, x_{n_{k_\ell}}]$ converges to $[z',y]\in X^u(y,R)$ which leads to a contradiction. 

Now we show that $$\lim_{n\to \infty} C'(x_n,y,R)=0.$$ Again assume to the contrary that it does not converge to zero. Then there is some $\varepsilon >0$ and a subsequence $(x_{n_k})_{k\geq 0}$ in $R$ such that $$\max\limits_{w\in X^u(y,R)}d(w, X^u(x_{n_k},R))\geq \varepsilon,$$ for every $k\geq 0$. Similarly, this means that there are $w_{n_k}\in X^u(y,R)$ such that $$d(w_{n_k}, X^u(x_{n_k},R))\geq \varepsilon,$$ for every $k\geq 0$. Since $X^u(y,R)$ is compact, there is a converging subsequence $(w_{n_{k_\ell}})_{\ell \geq 0}$ that converges to some $w'\in X^u(y,R)$. In particular, there is some $\ell_0 \in \mathbb N$ such that $d(w_{n_{k_\ell}},w') < \varepsilon/2$, for every $\ell \geq \ell_0$. Then $d(w', X^u(x_{n_{k_\ell}},R))\geq \varepsilon/2$, for every $\ell \geq \ell_0$, by using the inequality $$\lvert d(w_{n_{k_\ell}}, X^u(x_{n_{k_\ell}},R)) - d(w', X^u(x_{n_{k_\ell}},R))\rvert \leq d(w_{n_{k_\ell}},w').$$ 

However, for big enough $\ell$, since the diameter of $R$ is small, $w'$ and $x_{n_{k_\ell}}$ will be close enough to be bracketed and hence
$$
d(w', X^u(x_{n_{k_\ell}},R)) \leq d(w',[w',x_{n_{k_\ell}}])=d([w',y],[w',x_{n_{k_\ell}}]),
$$
where the last expression converges to zero. This leads to a contradiction.
\end{proof}
We consider the positive quantities
\begin{align*}\label{eq: smallest leave lengths}
\hypertarget{Q6}{\tag{Q6}} \underline{\text{diam}}_{s}\mathcal{R}_1 &=\text{min}\{\underline{\text{diam}}_s(R):R\in \mathcal{R}_1\}\\
\hypertarget{Q7}{\tag{Q7}}\underline{\text{diam}}_{u}\mathcal{R}_1 &=\text{min}\{\underline{\text{diam}}_u(R):R\in \mathcal{R}_1\}
\end{align*}

\begin{lemma}\label{lem: lower bounds in diameter}
Suppose that the homeomorphism $\varphi^{-1}$ is Lipschitz. Then it holds that
\begin{align*}
\underline{\text{diam}}(\mathcal{R}_n)&\geq \Lip(\varphi^{-1})^{-n+1} \underline{\text{diam}}_{s}\mathcal{R}_1.\\
\intertext{If $\varphi$ is Lipschitz then }
\underline{\text{diam}}(\mathcal{R}_n)&\geq \Lip(\varphi)^{-n+1} \underline{\text{diam}}_{u}\mathcal{R}_1.
\end{align*}
\end{lemma}

\begin{proof}
We just prove the first inequality and the second is similar. Note that the case $n=0$ is trivial. Instead of proving the statement directly we will prove, by induction on $n$, that for every $R\in \mathcal{R}_n$ and every $x\in R$ we have
\begin{align*}
\text{diam}(X^s(x,R))&\geq \Lip(\varphi^{-1})^{-n+1}\underline{\text{diam}}_{s}\mathcal{R}_1.\\
\intertext{For $n=1$ the above inequality is true from the definition of $\underline{\text{diam}}_{s}\mathcal{R}_1$. Assume that it holds for some $n\in\mathbb N$ and we claim that for every $R\in \mathcal{R}_{n+1}$ and every $x\in R$ it holds that }
\text{diam}(X^s(x,R))&\geq \Lip(\varphi^{-1})^{-n}\underline{\text{diam}}_{s}\mathcal{R}_1.
\end{align*}

Write $R=[\varphi^{-1}(R_j),\varphi(R_i)]$, for $R_i,R_j\in \mathcal{R}_n$, and with a similar argument as in equation \ref{eq:localunstablepart} one has $X^s(x,R)=\varphi(X^s(\varphi^{-1}(x),R_i)).$ Then
\begin{align*}
\text{diam}(X^s(x,R))&= \text{diam}(\varphi(X^s(\varphi^{-1}(x),R_i)))\\
& \geq \Lip(\varphi^{-1})^{-1}\text{diam}(X^s(\varphi^{-1}(x),R_i))\\
& \geq \Lip(\varphi^{-1})^{-1}\Lip(\varphi^{-1})^{-n+1}\underline{\text{diam}}_{s}\mathcal{R}_1\\
&= \Lip(\varphi^{-1})^{-n}\underline{\text{diam}}_{s}\mathcal{R}_1
\end{align*}
where the first inequality is true because $\varphi^{-1}$ is Lipschitz. 
\end{proof}

From Remark \ref{rem:Wielerremark} we obtain the following.

\begin{cor}\label{cor:Wielerdiameters}
For any irreducible Wieler solenoid $(X,\varphi)$ it holds that
$$\underline{\text{diam}}(\mathcal{R}_n)\geq \lambda_X^{-n+1}\underline{\text{diam}}_{s}\mathcal{R}_1,$$ where $\lambda_X$ is the contraction constant.
\end{cor}

\section{Semi-conformal Smale spaces and Ahlfors regularity}\label{sec:SemiconformalSmalespaces}

In this section we study regularity properties of the Bowen measure and derive dimension estimates for Smale spaces. In particular we focus on the following class of Smale spaces.

\begin{definition}\label{def:semiconformalSmalespace}
A Smale space $(X,\varphi)$ is called \textit{semi-conformal} if $\lambda_X=\ell_X$, where $\lambda_X>1$ is its contraction constant and $\ell_X=\min\{\Lip(\varphi),\Lip(\varphi^{-1})\}$.
\end{definition} 

By definition a self-similar Smale space is semi-conformal. Also, any Wieler solenoid is semi-conformal, see Remark \ref{rem:Wielerremark}. Note that if $\varphi^{-1}$ is $\lambda_X$-Lipschitz then $\varphi$ acts as the $\lambda_X^{-1}$-multiple of an isometry on local stable sets. The dual happens if $\varphi$ is $\lambda_X$-Lipschitz. In what follows $(\mathcal{R}_n)_{n\geq 0}$ will be a refining sequence of Markov partitions (see (\ref{eq:Markov refining sequence})) for an irreducible or mixing Smale space, with $\diam(\mathcal{R}_1)\leq \varepsilon''_X/2$.

\begin{prop}\label{prop:semiconformalmeasure}
Let $(X,\varphi)$ be a mixing semi-conformal Smale space. There is a constant $K>0$ such that, for every $n\in \mathbb N$ and $R\in \mathcal{R}_n$, the Bowen measure satisfies $$K^{-1}\diam(R)^{s_0}\leq \mu_{\Bow}(R)\leq K\diam(R)^{s_0}$$ where $s_0=2\ent(\varphi)/\log(\lambda_X)$.
\end{prop}

\begin{proof}
Since $(X,\varphi)$ is mixing, the corresponding topological Markov chain $(\Sigma_M,\sigma_M)$ will be mixing, too. Recall that the Bowen measure on $(\Sigma_M,\sigma_M)$ is the Parry measure $\mu_{\Par}$. 

Moreover, since $(X,\varphi)$ is semi-conformal, from Theorem \ref{thm:theoremgraphSmalespaces} and Lemma \ref{lem: lower bounds in diameter}, we obtain constants $\theta \geq \zeta >0$ such that 
\begin{equation}\label{eq:semiconformalinequality}
\lambda_X^{-n+1}\zeta \leq \diam(R)\leq \lambda_X^{-n+1}\theta,
\end{equation}
for every $n\in \mathbb N$ and $R\in \mathcal{R}_n$.

Let $R\in \mathcal{R}_n$ and $C\in \Sigma_M$  be the symmetric cylinder set of rank $2n-1$ such that $\pi_M(C)=R$. Theorem \ref{thm:Bowen factor map} says that $\pi_M$ is a metric isomorphism between $(\Sigma_M,\sigma_M,\mu_{\Par})$ and $(X,\varphi, \mu_{\Bow})$, hence $\mu_{\Bow}(R)=\mu_{\Par}(C)$. From Lemma \ref{lem:Parrymeasure} there is $D>0$ so that $$D^{-1}\lambda^{-2n}_{\max}\leq \mu_{\Par}(C)\leq D\lambda^{-2n}_{\max}$$ where $\lambda_{\max}$ is the Perron-Frobenius eigenvalue. Then using the inequality (\ref{eq:semiconformalinequality}) together with $$K=\max \{ D (\theta \lambda_X)^{s_0}, D (\zeta \lambda_X)^{-s_0}\}$$ and  $$s_0=2\ent(\varphi)/\log(\lambda_X)=\log_{\lambda_X}(\lambda_{\max}^2)$$ we obtain the result. 
\end{proof}

\begin{remark}\label{rem:MoranMarkov}
Results similar to Proposition \ref{prop:semiconformalmeasure} have been obtained in the setting of Moran constructions \cite[Def. 2.2]{KR} built from iterated function systems on complete metric spaces \cite{KR,KV}. Roughly speaking, a Moran construction $\mathcal{M}$ on a complete metric space $Z$, is a countable poset (by inclusion) of closed, bounded subsets of $Z$ with positive diameter, that has a unique maximum, the infimum of every chain is a point in $Z$, and where the elements of $\mathcal{M}$ correspond to finite words occurring in an one-sided subshift on finitely many symbols. Each Moran construction on $Z$ describes a limit set in $Z$, and one aims to control the diameter of the sets in $\mathcal{M}$. Measures that satisfy the inequality of Proposition \ref{prop:semiconformalmeasure} are called semi-conformal \cite{KV}. We note that the aforementioned notion of Moran constructions, although it has similarities, is different than the one used in \cite{Barreira, Pesin} that works in the Euclidean setting. 

We believe that refining sequences of Markov partitions on Smale spaces correspond to some kind of inverse limits of Moran constructions, as these refining sequences produce two-sided subshifts. We intend to investigate this connection in a future project. 
\end{remark}

We now introduce a homogeneity property for Smale spaces. It is related to the \textit{uniform finite clustering property} (UFCP) for Moran constructions used in \cite{KR,KV}, but is adjusted in the setting of refining sequences of Markov partitions.

\begin{definition}\label{def:UFCP}
A refining sequence of Markov partitions $(\mathcal{R}_n)_{n\geq 0}$ for a Smale space $(X,\varphi)$ satisfies the \textit{uniform finite clustering property} (UFCP) if 
$$
\sup\limits_{x\in X} \sup\limits_{r}\#\num_{\mathcal{R}_{n_r}}(\overline{B}(x,r))<\infty,
$$ 
where $r$ takes values in $(0,\diam(X))$ and $n_r=\min\{n\in \mathbb N: \overline{\diam}( \mathcal{R}_n) \leq r\}$. 
\end{definition}

We are interested in semi-conformal Smale spaces which admit refining sequences of Markov partitions that satisfy UFCP. In particular the following holds.

\begin{prop}\label{prop:UFCP}
Any refining sequence of Markov partitions $(\mathcal{R}_n)_{n\geq 0}$ for an irreducible self-similar Smale space $(X,\varphi)$ satisfies UFCP.
\end{prop}

\begin{proof}
Let $x\in X$ and $0<r<\diam(X)$. We claim that $\#\num_{\mathcal{R}_{n_r}}(\overline{B}(x,r))$ is bounded above by a constant which does not depend on $x$ and $r$. From Theorem \ref{thm:theoremgraphSmalespaces}, for sufficiently small $\delta >0$, we obtain the $\delta$-fattening $(\mathcal{R}^{\delta}_n)_{n\geq 0}$ and constants $0<\eta \leq \theta$ such that $$\overline{\diam}(\mathcal{R}_n^{\delta})\leq \lambda_X^{-n+1}\theta\enspace \text{and}\enspace  \Leb(\mathcal{R}_n^{\delta})\geq \lambda_X^{-n+1}\eta$$ for every $n\in \mathbb N$. Note that it suffices to prove the statement for $0<r < \eta/2$. Fix such $r$ and let $$m_r=\min \{n \in \mathbb N: \lambda_X^{-n+1}\theta \leq r\}.$$ It is easy to check that $m_r=1+\ceil{\log_{\lambda_X}(\theta /r)}$ and that $m_r\geq n_r$. Therefore, $\#\num_{\mathcal{R}_{n_r}}(\overline{B}(x,r))\leq \#\num_{\mathcal{R}_{m_r}}(\overline{B}(x,r))$. Moreover, from Lemma \ref{lem: injective fattening} the map $$\num_{\mathcal{R}_{m_r}}(\overline{B}(x,r))\to \num_{\mathcal{R}^{\delta}_{m_r}}(\overline{B}(x,r))$$ given by $R\to R^{\delta}$ is injective. In particular, 
$$\num_{\mathcal{R}_{m_r}}(\overline{B}(x,r))\leq \num_{\mathcal{R}^{\delta}_{m_r}}(\overline{B}(x,r)).
$$

Define $\ell_r =1+ \floor{\log_{\lambda_X}(\eta/(2r))}$ and one has $2r\leq \lambda_X^{-\ell_r+1}\eta \leq \Leb(\mathcal{R}_{\ell_r}^{\delta}).$ Then there is some $T^{\delta}\in \mathcal{R}_{\ell_r}^{\delta}$ which contains $\overline{B}(x,r)$ and clearly 
$$
\num_{\mathcal{R}^{\delta}_{\ell_r}}(\overline{B}(x,r))\subset \num_{\mathcal{R}^{\delta}_{\ell_r}}(T^{\delta}),
$$
where $\# \num_{\mathcal{R}^{\delta}_{\ell_r}}(T^{\delta}) \leq N_{\Pi}$, the uniform upper bound in the number of neighbouring rectangles, see Theorem \ref{thm:theoremgraphSmalespaces}.

We will work in a similar fashion as in Proposition \ref{prop:geometricembedding}. First we note that $m_r\geq \ell_r$ and every $R^{\delta}\in \num_{\mathcal{R}^{\delta}_{m_r}}(\overline{B}(x,r))$ is a descendant of depth $m_r-\ell_r$ of some element in $\num_{\mathcal{R}^{\delta}_{\ell_r}}(\overline{B}(x,r))$. We have $m_r-\ell_r \leq 2+\log_{\lambda_X}(2\theta / \eta)$ and since there is also a uniform upper bound $C_{\Pi}$ (see Theorem \ref{thm:theoremgraphSmalespaces}) on the number of descendants, due to the finite entropy, it holds 
$$
\# \num_{\mathcal{R}^{\delta}_{m_r}}(\overline{B}(x,r))\leq N_{\Pi}C_{\Pi}^{2+\log_{\lambda_X}(2\theta / \eta)}.
$$
As a result, 
$$
\#\num_{\mathcal{R}_{n_r}}(\overline{B}(x,r))\leq N_{\Pi}C_{\Pi}^{2+\log_{\lambda_X}(2\theta / \eta)}. \qedhere
$$
\end{proof}

We are now in a position to prove one of the main results of this paper. The power of this can be seen in Corollary \ref{cor:TopconjAhlfors}, where we obtain a sweeping result for all mixing Smale spaces.

\begin{thm}\label{thm:AhflorsregularitySmalespaces}
Let $(X,\varphi)$ be a mixing semi-conformal Smale space. Assume there is a refining sequence of Markov partitions that satisfies the UFCP. Then the Bowen measure is Ahlfors $s_0$-regular and therefore, $$\dim_H X= \dim_B X= \dim_A X=s_0$$ where $s_0=2\ent(\varphi)/\log(\lambda_X).$ Moreover, the $s_0$-dimensional Hausdorff measure is strictly positive.
\end{thm}

\begin{proof}
Suppose $(\mathcal{R}_n)_{n\geq 0}$ is a refining sequence of Markov partitions that satisfies the UFCP and let $M>0$ such that $\#\num_{\mathcal{R}_{n_r}}(\overline{B}(x,r))\leq M$ for every $x\in X$ and $0<r<\diam(X)$. Now fix some $x\in X$ and we want to estimate the measure $\mu_{\Bow}(\overline{B}(x,r))$ for $0<r<\diam(X)$, but since $\diam(X)<\infty$ it suffices to consider $0<r < \underline{\diam}(\mathcal{R}_1)/2$. Using the constant $K>0$ from Proposition \ref{prop:semiconformalmeasure} we obtain
\begin{align*}
\mu_{\Bow}(\overline{B}(x,r))&\leq \mu_{\Bow}(\bigcup \num_{\mathcal{R}_{n_r}}(\overline{B}(x,r)))\\
&\leq \sum_{R} \mu_{\Bow}(R)\\
&\leq \sum_{R} K\diam(R)^{s_0}\\
&\leq M K r^{s_0}
\end{align*}
where the sum is taken over all $R\in \num_{\mathcal{R}_{n_r}}(\overline{B}(x,r))$.

For the lower bound, take an infinite path in the approximation graph $\mathcal{P}_{\Pi}$ of $(\mathcal{R}_n)_{n\geq 0}$ that converges to $x$. From this path let $R$ be the first rectangle that is contained in $\overline{B}(x,r)$. Then its first ancestor $\widehat{R}$ will not be contained and hence $\diam(\widehat{R})>r$. Now we observe that $R\in \mathcal{R}_n$ for some $n\geq 2$, for if $R\in \mathcal{R}_1$ then $\underline{\diam}(\mathcal{R}_1)\leq \diam(R)\leq 2r <\underline{\diam}(\mathcal{R}_1).$ This means $\widehat{R}\in \mathcal{R}_{n-1}$ and $n-1\geq 1$. Therefore, we can apply inequality (\ref{eq:semiconformalinequality}) to obtain that $\diam(R)\geq c \diam(\widehat{R})$ for some $c\leq \zeta /(\lambda_X \theta)$. 

As a result, $$\mu_{\Bow}(\overline{B}(x,r))\geq \mu_{\Bow}(R)\geq K^{-1}\diam(R)^{s_0}\geq K^{-1}c^{s_0}r^{s_0},$$ and hence the Bowen measure is Ahlfors $s_0$-regular. The rest follows from Proposition \ref{prop:equalityofdimensions} and Remark \ref{rem:Ahlforsrem}.
\end{proof}

\begin{remark}\label{rem:Hausdorff_equals_box}
We now explain Theorem \ref{thm:AhflorsregularitySmalespaces} by focusing on a mixing self-similar Smale space $(X,d,\varphi)$. First, we should note that the coincidence of the Hausdorff and box-counting dimensions can be equivalently obtained using Barreira's techniques. Specifically, one can use Theorem 3.15 in \cite{Barreira}, which concerns the dimension theory of Smale spaces with bi-Lipschitz local product structure and with asymptotically conformal dynamics on stable and unstable sets. However, Barreira's result cannot be related to Ahlfor regularity or Assouad dimension. 

The stable and unstable sets of $(X,d,\varphi)$ have asymptotically conformal dynamics in a strong sense. Moreover, using part (4) of Proposition \ref{prop:mainresultMarkovpartitions} and the proof of Lemma \ref{lem: lower bounds in diameter}, we obtain the dimensions of local stable and unstable sets as the roots $r_s=r_u=\ent(\varphi)/\log(\lambda_X)$ of Bowen's equation for the topological pressure. 

What needs to be observed is that the local product structure is bi-Lipschitz. Then, we can add the dimensions and obtain $\dim_H X=\dim_B X=2\ent(\varphi)/\log(\lambda_X).$ From \cite[Lemma 4.3]{Artigue}, every stable and unstable holonomy map in a sufficiently small rectangle is Lipschitz. In particular, for every $\varepsilon>0$ there is $c>0$ so that, for all $y\in X$ and $x,x' \in X^u(y,c),\, z,z' \in X^s(y,c)$, it holds that 
\begin{equation}\label{eq:Hausdorff_equals_box_1}
\begin{split}
d([x,z],[x',z])&\leq (1+\varepsilon) d(x,x')\\
d([x,z],[x,z'])&\leq (1+\varepsilon) d(z,z').
\end{split}
\end{equation}
Moreover, the self-similar version of Lemma \ref{lem:Lipschitzbracket} (see \cite[Remark 2.22]{Artigue}) is that for a (possibly) smaller $c>0$, if $z,w\in X$ with $d(z,w)\leq c$, then 
\begin{equation}\label{eq:Hausdorff_equals_box_2}
\begin{split}
d(z,[z,w])&\leq (1+\varepsilon)d(z,w)\\
d(w,[z,w])&\leq (1+\varepsilon)d(z,w).
\end{split}
\end{equation}
Fix $\varepsilon>0$ and let $c>0$ be small enough so that both (\ref{eq:Hausdorff_equals_box_1}) and (\ref{eq:Hausdorff_equals_box_2}) hold. Then, it is straightforward to see that, for every $y\in X$, if $X^u(y,c)\times X^s(y,c)$ is equipped with the product metric, the bracket map $[\cdot , \cdot]: X^u(y,c)\times X^s(y,c)\to X$ is bi-Lipschitz onto its image, with a small constant depending on $\varepsilon$. 

Further, from Theorem \ref{thm:AhflorsregularitySmalespaces} we obtain that the Assouad dimension of a self-similar Smale space $(X,d,\varphi)$ is finite. Note that this weaker result can also be obtained by Theorem \ref{thm:theoremgraphSmalespaces} and Proposition \ref{prop:geometricembedding}. Now, Assouad's Theorem \ref{thm:Assouad} asserts that for every $\varepsilon \in (0,1)$, the snowflaked metric space $(X,d^{\varepsilon})$ is bi-Lipschitz embeddable in a Euclidean space. Note that $(X,d^{\varepsilon},\varphi)$ is still a self-similar Smale space. However, it is not clear whether the embedding is a Smale space, because the contraction axioms (C1) and (C2) depend on the Lipschitz constant of the embedding. But even if it were a Smale space, the embedding may no longer be self-similar or even conformal in a broader sense, so that Pesin's techniques \cite{Pesin} on Ahlfors regularity could be applied. One would require some sort of isometric embedding in the Euclidean space, and this seems extremely difficult, if not unlikely. But even if such a fine embedding would exist, the methods in \cite{Pesin} would apply only on the snowflaked version $(X,d^{\varepsilon},\varphi)$. 
\end{remark}

Theorem \ref{thm:AhflorsregularitySmalespaces}, together with Proposition \ref{prop:UFCP} and  Lemma \ref{lem:Artiguelemma}, yields the following result.

\begin{cor}\label{cor:TopconjAhlfors}
Any mixing Smale space is topologically conjugate to a mixing Smale space on which the Bowen measure is Ahlfors regular.
\end{cor}

\begin{remark}\label{rem:non_Ahlfors_regular}
Not all mixing Smale spaces have Ahlfors regular measures. Such Smale spaces exist in the context of non-conformal hyperbolic dynamical systems, where the Hausdorff and box-counting dimensions may not agree. An example of a Smale space whose dimensions do not coincide can be found in the work of Pollicott and Weiss \cite{PolW} who studied the dimension theory of certain linear horseshoes in $\mathbb R^3$. We note that a horseshoe in $\mathbb R^3$ is constructed in a way similar to the classical Smale's horseshoe in $\mathbb R^2$. For a specific definition of the diffeomorphism in $\mathbb R^3$ we refer to \cite{SimonSol}. 
\enlargethispage{\baselineskip}
Pollicott and Weiss considered a linear horseshoe $(\Lambda,f)$ so that $\Lambda=F\times E$, where $F$ is a certain self-affine limit set in the plane and $E$ is a uniform Cantor set. The limit set $F$ is constructed by two affine contractions $A_0, A_1$ on the unit square $I$, where $A_0(I), A_1(I)$ are disjoint rectangles in $I$ placed in the lower left corner and the upper right corner of $I$, respectively, each having height $\lambda_1<1/2$ and width $\lambda_2$ equal to the reciprocal of the golden mean. The horseshoe $(\Lambda,f)$ is a Smale space (after considering a Lipschitz equivalent adapted metric \cite[Prop. 5.2.2]{BS}) and is topologically conjugate to the full-two shift $(\Sigma_2,\sigma_2)$. 

Now, the specific construction of $F$ gives that $\dim_H F < \dim_B F$, and since $\dim_H \Lambda= \dim_H F+\dim_H E,\, \dim_B \Lambda= \dim_B F+\dim_B E$ and $\dim_H E=\dim_B E,$ we obtain that $\dim_H \Lambda <\dim_B \Lambda$. This argument is independent of $\lambda_1<1/2$, and hence there is a family of horseshoes indexed by an open interval, whose Hausdorff and box-counting dimensions do not coincide. Therefore, every such $\Lambda$ does not have Ahlfors regular measures. In \cite{PolW} one can find several interesting linear horseshoes whose dimension depends on fine number theoretic properties of the contraction coefficients. Also, from \cite[Section 16]{Pesin} one can build linear horseshoes in $\mathbb R^4$ whose Hausdorff dimension is strictly smaller than the box-counting dimension.

According to Theorem \ref{thm:AhflorsregularitySmalespaces}, in order to make the Bowen measure on the example horseshoe $(\Lambda,f)$ Ahlfors regular, it suffices to change the metric of $\Lambda$ to a self-similar one. Since $(\Lambda,f)$ is topologically conjugate to $(\Sigma_2,\sigma_2)$, one possibility is to equip $\Lambda$ with a self-similar ultrametric of $\Sigma_2$. Another approach is to see whether $(\Lambda,f)$, equipped with its original metric, satisfies Fathi's property (Theorem \ref{lem:Fathilemma}) and then follow the method discussed in Subsection \ref{sec:LipschitzSmalespaces}. This depends on the Lipschitz and contraction constants of $f,f^{-1}$. Finally, one can follow the philosophy of Fried \cite{Fried} and reconstruct a metric on $(\Lambda,f)$ that will satisfy Fathi's property with parameters obtained from the original geometry of $\Lambda$. This approach is more abstract since it requires to view $\Lambda$ as a uniform space. However, we believe it can produce natural metrics on $\Lambda$. In a future project we aim to construct natural self-similar metrics for specific classes of Smale spaces.
\end{remark}

From Theorem \ref{thm:AhflorsregularitySmalespaces}, Proposition \ref{prop:UFCP}, Remark \ref{rem:Artigueargument} and the metric inequalities (\ref{eq:metricinequality}) in Subsection \ref{sec:LipschitzSmalespaces} we obtain the following dimension estimates. Recall that $\Lambda_X=\max\{\Lip(\varphi),\Lip(\varphi^{-1})\}$.

\begin{cor}\label{cor:Hausdorffdimension}
Let $(X,\varphi)$ be a mixing Smale space with $\Lambda_X<\infty$. Suppose that $\lambda_X >2 A_X$, where $A_X>0$ is the constant obtained in Lemma \ref{lem:Lipschitzbracket}. Then it holds $$\frac{2\ent(\varphi)}{\log \Lambda_X}\leq \dim_H X \leq \underline{\dim}_B X\leq \overline{\dim}_B X\leq \frac{2\ent(\varphi)}{\log \lambda_X- \log(2A_X)}.$$
\end{cor}
We should point out again that since $A_X\leq (\Lambda_X \lambda_X)/(\lambda_X^2-1),$ for $\lambda_X >2 A_X$ to be true in general, it suffices to restrict $\lambda_X\in (1+\sqrt{2}, \infty)$ and $\Lambda_X\in [\lambda_X, (\lambda_X^2-1)/2)$. Then, considering the behaviour of the Hausdorff and box dimensions with H{\"o}lder equivalent metrics, it is possible to obtain upper and lower bounds for Smale spaces with contraction/expansion constants in $(1,1+\sqrt{2}]$. However, the goal should be to estimate $A_X$. The Assouad dimension is not included in the inequality since it does not behave well with arbitrary H{\"o}lder transformations \cite{Heinonen}. 

Secondly, the upper bound can also be obtained from \cite[Theorem 5.3]{Fathi} and the discussion in Subsection \ref{sec:LipschitzSmalespaces}. Finally, the lower bound $2\ent(\varphi)/\log \Lambda_X \leq \dim_H X$ enhances the previous bound $\ent(\varphi)/\log \Lambda_X \leq \underline{\dim}_B X$ obtained in \cite[Theorem 5.6]{Fathi}.

\end{document}